\pdfoutput=1
\documentclass[11pt]{article}
\usepackage{epsfig,epsf,fancybox}
\usepackage{amsmath}
\usepackage{mathrsfs}
\usepackage{amssymb}
\usepackage{graphicx}
\usepackage{color}
\usepackage{multirow}
\usepackage{paralist}
\usepackage{verbatim}
\usepackage{galois}
\usepackage{algorithm}
\usepackage[toc,page,header]{appendix}
\usepackage{titletoc}
\newcommand{\insertseparator}{%
  \par
  \noindent\rule{\linewidth}{1.5pt} 
  \par
}

\usepackage{algpseudocode}

\providecommand{\tabularnewline}{\\}
\floatstyle{ruled}
\newfloat{algorithm}{tbp}{loa}
\providecommand{\algorithmname}{Algorithm}
\floatname{algorithm}{\protect\algorithmname}
\usepackage{boxedminipage}
\usepackage{booktabs}
\usepackage{accents}
\usepackage{stmaryrd}
\usepackage[table]{xcolor}
\usepackage{hhline}
\DeclareSymbolFont{extraup}{U}{zavm}{m}{n}
\DeclareMathSymbol{\varheart}{\mathalpha}{extraup}{86}
\usepackage{subfig}

\usepackage{natbib}

\usepackage{url}
\usepackage[colorlinks,linkcolor=magenta,citecolor=blue, pagebackref=true,backref=true]{hyperref}
\renewcommand*{\backrefalt}[4]{%
    \ifcase #1 \footnotesize{(Not cited.)}%
    \or        \footnotesize{(Cited on page~#2.)}%
    \else      \footnotesize{(Cited on pages~#2.)}%
    \fi}

\newtheorem{lem}{Lemma}[section]

\newtheorem{prop}{Proposition}[section]

\newtheorem{defn}{Definition}[section]

\newtheorem{thm}{Theorem}[section]

\newtheorem{remcou}{Remark}
\newtheorem{rem}[remcou]{Remark}
\newtheorem{assumption}{Assumption}

\numberwithin{equation}{section}

\title{\bf{\LARGE{Decentralized Gradient-Free Methods for Stochastic Non-Smooth Non-Convex Optimization}}}

\author{Zhenwei Lin\footnote{Equal Contribution} \thanks{zhenweilin@163.sufe.edu.cn, Shanghai University of Finance and Economics} \quad\quad\quad  Jingfan Xia$^*$\thanks{jf.xia@163.sufe.edu.cn, Shanghai University of Finance and Economics} \quad\quad\quad  Qi Deng\thanks{qideng@sufe.edu.cn, Shanghai University of Finance and Economics} \quad\quad\quad Luo Luo\thanks{luoluo@fudan.edu.cn, Fudan University}
}

\date{}
\usepackage{bm}
\global\long\def\dgfm{\text{DGFM}$^{+}$}%
\global\long\def\dgfmlin{\text{DGFM}}%
\global\long\def\mnist{\text{MNIST}}%
\global\long\def\fashion{\text{Fashion-MNIST}}%
\global\long\def\binprod#1#2{\big\langle #1,#2\big\rangle }%
\global\long\def\Binprod#1#2{\Big\langle #1,#2\Big\rangle }%
\global\long\def\norm#1{\big\Vert#1\big\Vert}%
\global\long\def\bnorm#1{\big\Vert#1\big\Vert}%
\global\long\def\Bnorm#1{\Big\Vert#1\Big\Vert}%
\global\long\def\brbra#1{\big(#1\big)}%
\global\long\def\Brbra#1{\Big(#1\Big)}%
\global\long\def\lrrbra#1{\left(#1\right)}%
\global\long\def\bsbra#1{\big[#1\big]}%
\global\long\def\lrsbra#1{\left[#1\right]}

\global\long\def\Bsbra#1{\Big[#1\Big]}%
\global\long\def\abs#1{\vert#1\vert}%
\global\long\def\iid{\textit{i.i.d}}%
\global\long\def\bcbra#1{\big\{#1\big\}}%
\global\long\def\Bcbra#1{\Big\{#1\Big\}}%
\global\long\def\vertiii#1{\left\vert \kern-0.25ex  \left\vert \kern-0.25ex  \left\vert #1\right\vert \kern-0.25ex  \right\vert \kern-0.25ex  \right\vert }%
\global\long\def\til#1{\tilde{#1}}%
\global\long\def\mcal#1{\mathcal{#1}}%
\global\long\def\mbb#1{\mathbb{#1}}%
\global\long\def\onebf{\mathbf{1}}%
\global\long\def\zerobf{\mathbf{0}}%
\global\long\def\Expe{\mathbb{E}}%
\global\long\def\and{\mathrm{and}}%
\global\long\def\aleq{\overset{(a)}{\leq}}%
\global\long\def\aeq{\overset{(a)}{=}}%
\global\long\def\ageq{\overset{(a)}{\geq}}%
\global\long\def\bleq{\overset{(b)}{\leq}}%
\global\long\def\beq{\overset{(b)}{=}}%
\global\long\def\bgeq{\overset{(b)}{\geq}}%
\global\long\def\Var{\operatornamewithlimits{Var}}%
\global\long\def\conv{\operatorname{conv}}%
\global\long\def\vep{\varepsilon}%
\global\long\def\tsum{{\textstyle {\sum}}}%
\global\long\def\Rbb{\mathbb{R}}%
\global\long\def\Pbb{\mathbb{P}}%
\global\long\def\Fcal{\mathcal{F}}%
\global\long\def\Ocal{\mathcal{O}}%
\global\long\def\Tcal{\mathcal{T}}%
\global\long\def\abf{\mathbf{a}}%
\global\long\def\bbf{\mathbf{b}}%
\global\long\def\gbf{\mathbf{g}}%
\global\long\def\vbf{\mathbf{v}}%
\global\long\def\xbf{\mathbf{x}}%
\global\long\def\ybf{\mathbf{y}}%
\global\long\def\zbf{\mathbf{z}}%
\global\long\def\Ibf{\mathbf{I}}%
\global\long\def\rone{\textrm{I}}%
\global\long\def\rtwo{\textrm{II}}%
\global\long\def\rthree{\textrm{III}}%
\global\long\def\rfour{\textrm{IV}}%
\global\long\def\rfive{\textrm{V}}%
\global\long\def\rsix{\textrm{VI}}%
\begin{document}
\maketitle


\begin{abstract}    
We consider decentralized gradient-free optimization of minimizing Lipschitz continuous functions that satisfy neither smoothness nor convexity assumption. We propose two novel gradient-free algorithms, the Decentralized Gradient-Free Method (DGFM) and its variant, the Decentralized Gradient-Free Method$^+$ (DGFM$^{+}$).
Based on the techniques of randomized smoothing and gradient tracking, DGFM requires the computation of the zeroth-order oracle of a single sample in each iteration, making it less demanding in terms of computational resources for individual computing nodes. Theoretically, DGFM achieves a complexity of $\mathcal O(d^{3/2}\delta^{-1}\varepsilon^{-4})$ for obtaining an $(\delta,\varepsilon)$-Goldstein stationary point.
DGFM$^{+}$, an advanced version of DGFM, incorporates variance reduction to further improve the convergence behavior. It samples a mini-batch at each iteration and periodically draws a larger batch of data, which improves the complexity to $\mathcal O(d^{3/2}\delta^{-1}\varepsilon^{-3})$. Moreover, experimental results underscore the empirical advantages of our proposed algorithms when applied to real-world datasets.
\end{abstract}

\section{Introduction}
In this paper, we consider decentralized optimization where the data are distributed among multiple agents, also known as nodes or entities.
For a network with $m$ agents, the optimization
problem can be written in the following form:
\vspace{-0.5em}
\begin{equation}
\min\limits_{x\in\mathbb{R}^{d}}\ \ f(x)=\frac{1}{m}\sum_{i=1}^{m}f^{i}(x),\label{eq:P}
\end{equation}
where $f^{i}(x)=\mbb E_{\xi}[f^{i}(x;\xi)]$ is a local
cost function on the $i$-th node and $\xi$ is the index of the random sample. 
Instead of having a central server, each node $i$ makes decisions based on its local data and information received from its neighbors. 
Throughout the paper, we do not require any smoothness or convexity assumption but only suppose that each $f^i(\cdot,\xi)$ is Lipschitz continuous. 
Moreover, we focus on the gradient-free methods that exclusively rely on function values, avoiding access for any first-order information.

Decentralized optimization has found extensive applications in signal processing and machine learning~\citep{ling2010decentralized,giannakis2017decentralized,vogels2021relaysum}.
In the context of smooth non-convex objective that each $f^{i}(x)$ has the finite-sum structure, a variety of deterministic methods have been proposed~\citep{zeng2018nonconvex,hong2017prox,sun2019distributed,scutari2019distributed,xin2022fast,luo2022optimal}. 
Notably,  \citet{xin2022fast} achieved a network topology-independent
convergence rate in a big-data regime. 
\citet{luo2022optimal}
integrated variance reduction, gradient tracking, and multi-consensus techniques, yielding an algorithm that meets the tighter communication requirement and complexity level of first-order oracle algorithms.  
In the realm of stochastic decentralized optimization, a significant body of literature has explored acceleration techniques that incorporate variance reduction~\citep{pan2020d,sun2020improving,xin2021hybrid,xin2021improved}.

Moreover, a substantial volume of literature exists regarding resolving non-convex non-smooth optimization~\citep{di2016next,scutari2019distributed,wang2021distributed,xin2021stochastic,mancino2022proximal,xiao2023one,chen2021distributed,wang2023decentralized}. However, most existing works require the objective to adhere to a specific structure. 
Predominantly, studies focus on composite optimization, where the objective sums up a smooth non-convex part and a possibly non-smooth part. In this vein, \citet{scutari2019distributed} introduced a decentralized algorithmic framework for minimization of the sum of a smooth non-convex function and a non-smooth difference-of-convex function over a time-varying directed graph. \citet{mancino2022proximal} introduced a single-loop algorithm with a small batch size which achieved a network topology-independent complexity. 
Conversely, other recent investigations have focused on the decentralized optimization of non-smooth weakly-convex functions~\citep{chen2021distributed,wang2023decentralized}. 


Previous decentralized algorithms still require gradient computation, while this oracle may be computationally prohibitive~\citep{liu2020primer}, such as sensor selection~\citep{liu2018zeroth}. Moreover, the gradient-free method has a promising application in adversarial machine learning, especially in black-box adversarial attacks~\citep{chen2020hopskipjumpattack,moosavi2017universal}. 
Recently, some works studied decentralized optimization problem with zeroth-order  methods~\citep{tang2020distributed,sahu2018distributed,yu2021distributed,hajinezhad2019zone,tang2023zeroth}. 
For example, \citet{sahu2018distributed} considered convex problems, while~\citet{hajinezhad2019zone} focused on non-convex problems. 
Moreover, some attention has been paid to applying zero-order decentralized algorithms to constrained optimization~\citep{yu2021distributed,tang2023zeroth}.
However, these researches only focus on smooth optimization.

The decentralized algorithms described above can be classified into smooth non-convex and non-smooth non-convex with specific structures.
This leads us to raise the following question: 
\emph{{Can we develop a decentralized gradient-free algorithm that has provable complexity guarantees for non-smooth, non-convex but Lipschitz continuous problems?}}
To address this research problem, a natural idea
is to extend the centralized algorithms designed for non-smooth non-convex problems to the decentralized
setting. 
\citet{zhang2020complexity} introduced $(\delta,\vep)$-Goldstein stationarity as a valid criterion for non-smooth non-convex optimization, which makes it possible to analyze the non-asymptotic convergence. 
They utilize a random sampling approach to choose an interpolation point on the segment connecting two iterates. This method guarantees a substantial descent of the objective function, given the assumption that the function is Hadamard directionally differentiable and access to a generalized gradient oracle is available.
This algorithm can achieve complexity 
$\Ocal(\Delta L_f^3 \delta^{-1} \vep^{-4})$, where $L_f$ is the Lipschitz continuous constant of the objective and $\Delta$ is the inital function value gap.
Later, \citet{davis2022gradient} and~\citet{tian2022finite} relaxed the subgradient selection oracle assumption and Hadamard directionally differentiable assumption by adding random perturbation.
More recently,~\citet{cutkosky2023optimal} found a connection between
non-convex stochastic optimization and online learning and established a stochastic first-order oracle
complexity of $\Ocal(\Delta L_f^2 \delta^{-1} \vep^{-3})$, which is the optimal in the case of $\vep\leq\Ocal(\delta)$. 
In light of recent advances in designing zeroth-order algorithms for non-smooth non-convex problems, an effective way is by applying the randomized smoothing technique~\cite{nesterov2017random,shamir2017optimal}. 
This approach constructs a smooth surrogate function to which algorithms for smooth functions can be applied.
\citet{lin2022gradient} established the relationship between Goldstein stationarity of the original objective function and $\vep$-stationarity of the surrogate function, 
and presented an algorithm for finding a $(\delta,\vep)$-Goldstein stationary point within at most $\Ocal(d^{3/2}L_f^4\vep^{-4}+ d^{3/2}\Delta L_f^3 \delta^{-1}\vep^{-4})$ stochastic zeroth-order oracle calls. Later, \citet{chen2023faster} constructed stochastic
recursive gradient estimators to accelerate and achieve a stochastic zeroth-order oracle complexity of $\Ocal(d^{3/2}L_f^3 \vep^{-3}+ d^{3/2}\Delta L_f^2 \delta^{-1}\vep^{-3})$.
All above work applied the first-order or zero-order methods for finding the approximate stationary point of the smooth surrogate function. Very recently, \citet{kornowski2023algorithm} replaced the goal of finding an $\vep$-stationary point of the smoothed function with that of finding a Goldstein stationary point and then usd a stochastic first-order nonsmooth nonconvex algorithm.
This change led to an improved dependence on the dimension, reducing it from $d^{3/2}$ to $d$.



\subsection{Contributions}
In this work, we propose two gradient-free decentralized algorithms for non-smooth non-convex optimization: the Decentralized Gradient Free Method ({\dgfmlin}) and the Decentralized Gradient Free Method$^{+}$ ({\dgfm}). 

{\dgfmlin}  is a decentralized approach that leverages randomized smoothing and gradient tracking techniques. {\dgfmlin} only requires the computation of the zeroth-order oracle of a single sample, thus reducing the computational demands on individual computing nodes and enhancing practicality.  Theoretically, {\dgfmlin} achieves a complexity bound of $\mcal O(d^{3/2}\delta^{-1}\vep^{-4})$ for reaching a $(\delta, \vep)$-Goldstein stationary solution in expectation. 
Furthermore, {\dgfmlin} requires the same number of communication rounds as the number of iterations. To the best of our knowledge, \dgfmlin{} is the first decentralized algorithm for general non-smooth non-convex optimization problems. Our complexity result matches the complexity bound of the standard gradient-free stochastic method~\citep{lin2022gradient}.

We also propose an enhanced algorithm {\dgfm}  by incorporating the variance reduction technique SPIDER~\citep{fang2018spider}. 
{\dgfm} samples a mini-batch at each iteration and periodically samples a mega-batch of data. 
This strategy improve the zeroth-order oracle complexity to $\mcal O(d^{3/2}\delta^{-1} \vep^{-3})$ for reaching $(\delta, \vep)$-Goldstein stationarity in expectation. 
In comparison to {\dgfmlin}, {\dgfm} not only employs randomized smoothing and gradient tracking but also introduces a multi-communication module. 
This addition ensures that the order of communication complexity is on par with that of the iteration complexity.

\section{Preliminaries}
We give some notations and introduce basic concepts in non-smooth and non-convex analysis.

\paragraph{Notations.}
We use subscripts to indicate the nodes to which the variables belong
and superscripts to indicate the iteration numbers. We use bold letters
such as $\xbf$ to represent stack variables e.g., $\xbf^k = [(x_1^k)^{\top},\cdots,(x_m^k)^\top]^{\top}\in\mbb R^{md}$.
We denote $\|x\|_{q}=(\sum_{i=1}^{d}\abs{x_{(i)}}^{q})^{1/q}$, $q>0$
for the $\ell_{q}$-norm, where~$x_{(i)}$ denote the $i$-th
element of $x$. For brevity, $\norm x$ stands for $\ell_{2}$-norm.
For a matrix $A$, we denote its spectral norm as $\norm A$.  
The notation $\mcal B_{\nu}(x)$ presents a closed Euclidean ball centered at $x$
with radius~$\nu>0$, i.e., $\mcal B_{\nu}(x)=\{y : \|y-x\|\leq\nu\}$. 
We use $\mbb S^{d-1}=\{x\in \mbb R^d : \|x\|=1\}$ to denote the sphere of the unit ball in $\ell_2$-norm. We work with a probability space $\bcbra{\Omega,  \Fcal,\Pbb}$, where $\Omega$ is a sample space containing all possible outcomes, $\Fcal$ is the sigma-algebra on $\Omega$ representing the set of events, and $\Pbb$ is the probability measure. In this paper, we use $\Pbb^d$
to denote the uniform distribution on $\mbb S^{d-1}$. 
We consider decentralized algorithms that generate a sequence  $\{x_i^k\}_{k\geq 0}$ to approximate the stationary point of $f(\cdot)$. 
At each iteration $k$, each node $i$ observes a random vector set $S_i^k = \{(\xi_i^{k,j},w_{i}^{k,j})\}_{j=1}^{b}$, where $\xi$ is the data sample, $w$ is a random vector sample for gradient estimation and $b$ is the batch size. We introduce a natural filtration induced by these random vector sets observed sequentially by these nodes: $\mcal{F}_0:=\bcbra{\Omega,\varnothing}$ and $\mcal F_k:=\sigma(\bcbra{S_i^0,S_i^1,\cdots,S_i^{k-1}:i\in[m]})$ for any $k\geq 1$.
   

\paragraph{Stationary condition.}
In the non-convex setting, Clarke's subdifferential (or generalized gradient) \citep{clarke1990optimization} is perhaps the most natural and well-known extension of the standard convex subdifferential. 
\begin{defn}
Given a point $x\in\mbb R^{d}$
and the direction $v\in\mbb R^{d}$, the generalized directional derivative
of a Lipschitz continuous function $f$ is given by
\vspace{-0.5em}
\begin{equation*}
    Df(x;v):=\limsup_{y\to x,t\downarrow0}\frac{f(y+tv)-f(y)}{t}.
\end{equation*}
The generalized gradient of $f$ is defined as
\vspace{-0.5em}
\[\partial f(x):=\{g\in\mbb R^{d}:g^{\top}v\leq Df(x;v),\forall v\in\mbb R^{d}\}.\]
\end{defn}
We need to properly define the approximate stationary condition for the efficiency analysis. An intuitive choice is to consider the $\vep$-Clarke's stationary point which is defined by $\textrm{dist}(0,\partial f(x))\le \vep$. However, \citet{kornowski2021oracle} demonstrated that  accessing such approximate stationarity for sufficiently small $\vep$ tends to be generally intractable. 
Therefore, as suggested by \citet{lin2022gradient,zhang2020complexity}, it is more sensible to target a ($\delta,\vep$)-Goldstein stationary point~\citep{goldstein1977optimization}.
\begin{defn}
[($\delta,\vep$)-Goldstein Stationary Point] Given $\delta>0$, the $\delta$-Goldstein subdifferential of Lipschitz
function $f(\cdot)$ at $x$ is given by $\partial_{\delta}f(x):=\conv(\cup_{y\in\mbb B_{\delta}(x)}\partial f(y)).$
Then we say point $x$ is a ($\delta,\vep$)-Goldstein
stationary point if $\min\{\norm g:g\in\partial_{\delta}f(x)\}\leq\vep$.
\end{defn}

\paragraph{Randomized Smoothing.}
For non-smooth problems, a natural idea
is first to apply smoothing techniques to these problems and then minimize the resulting smoothed surrogate function. We highlight some key properties and refer to \citep{lin2022gradient} for more details.
\begin{defn}
[Randomized smoothing] 
We say function $f_{\delta}(x)$
is a randomized smooth approximation of the non-smooth function $f(x)$ if $f_{\delta}(x):=\mbb E_{w\sim\Pbb}[f(x+\delta\cdot w)].$
\end{defn}

 The randomized smooth approximation requires the objective function $f(x)=\frac{1}{m}\sum_{i=1}^{m}\mbb E_{\xi}[f^{i}(x;\xi)]$ to satisfy the following assumption to have good properties.
%
\begin{assumption}\label{assu:2.2}For  $\forall i\in[m]$, assume 
condition $\norm{f^{i}(x;\xi)-f^{i}(y;\xi)}\leq L(\xi)\|x-y\|$ holds, namely,~$f^{i}(\,\cdot\,;\xi)$ is $L(\xi)$-Lipschitz continuous.
Furthermore,  assume there exists a constant $L_{f}$ such that
$\mbb E[L(\xi)^{2}]\leq~L_{f}^{2}$. 
Moreover, we assume $f(\cdot)$ is lower bounded and define $f^*=\inf_{x\in \mbb R^d} f(x)$.
\end{assumption}
Then we give a specific gradient-free method to approximate the gradient.
\begin{defn}[Zeroth-order oracle estimators]\label{def:Given-a-stochastic}Given
a stochastic component $f(\cdot;\xi):\mbb R^{d}\to\mbb R$, we define its zeroth-order oracle estimator at $x\in\mbb R^{d}$
by:
\begin{align*}
g(x;w,\xi)=\frac{d}{2\delta}(f(x+\delta\cdot w;\xi)-f(x-\delta\cdot w;\xi))w,   
\end{align*}
where $w$ is uniformly sampled from $\mbb S^{d-1}$.  Let $S=\{(\xi_{i},w_{i})\}_{i=1}^{b}$,
where vectors $w_{1},\cdots,w_{b}\in\mbb R^{d}$ are $\iid$ sampled
from $\mbb S^{d-1}$ and random indices
$\xi_{1},\cdots,\xi_{b}$ are $\iid$. We define the mini-batch zeroth-order
gradient estimator of $f(\cdot;\xi)$ in terms of
$S$ at $x\in\mbb R^{d}$ by 
\begin{align*}
g(x;S)=\frac{1}{b}\sum_{i=1}^{b}g(x;w^{i},\xi^{i}).    
\end{align*}
\end{defn}
\begin{prop}[Lemma D.1 \citep{lin2022gradient}]\label{prop:zeroth-order-estimator}
For the zeroth-order oracle estimator in Definition~\ref{def:Given-a-stochastic},
we have $\mbb E_{w,\xi}[g(x;w,\xi)]=\nabla f_{\delta}(x)$ and
$\mbb E_{w,\xi}\bsbra{\norm{g(x;w,\xi)}^{2}}\leq16\sqrt{2\pi}dL_{f}^{2}$.
\end{prop}
In the remains of this paper, we use the notation $\sigma^2=16\sqrt{2\pi}dL_{f}^{2}$. 
%

%

We summarize the main properties of randomized smoothing in the following proposition. 
\begin{prop}[Proposition 2.2~\citep{chen2023faster}]
Suppose Assumption~\ref{assu:2.2} holds,
then for the local function $f^{i}(x)=\mbb E[f^{i}(x;\xi)]$, we have
\begin{enumerate}
\item $\abs{f^{i}(\cdot)-f_{\delta}^{i}(\cdot)}\leq\delta L_{f}$.
\item $f_{\delta}^{i}(x)$ is $L_{f}$-Lipschitz continuous and $L_{\delta}$-smooth,
where $L_{\delta}={cL_{f}\sqrt{d}}/{\delta}$ and $c$ is a constant.
\item $\nabla f_{\delta}^{i}(\cdot)\in\partial_{\delta}f^i(\cdot)$. 
\item There exists $\Delta$ such that for any $x\in{\mathbb R}^d$, $f_{\delta}(x)-f^*\leq\Delta_{\delta}$, where $\Delta_{\delta}=\Delta+\delta L_{f}$.
\end{enumerate}
\end{prop}

\noindent\textbf{Graph. }
Consider a time-varying network $(\mcal V, \mcal E_k)$ of agents, where $\mcal V$ denotes the set of nodes and $\mcal E_k$ is the set of links connecting nodes at time $k$.  Let $A(k)\triangleq(a_{i,j}(k))_{i,j=1}^{m}$ denote the matrix of weights associated with links in the graph at time $k>0$. In addition, we will use $A^{\tau}(k)\triangleq(a_{i,j}^{\tau}(k))_{i,j=1}^{m}$ to denote the weighted matrix (mixing matrix) for the $\tau$-th repetition of communication in $k$-th iterations.
We make the following standard assumption on the
graph topology.
\begin{assumption}
[Graph property]\label{assu:doubly_sto}The weighted matrix $A(k)$
has a sparsity pattern compliant with~$\mcal G$ that is 
\begin{enumerate}
\item $a_{ii}(k)>0$, for any $i$ and $k$;
\item $a_{i,j}(k)>0$ if $(i,j)\in\mcal E_k$ and $a_{i,j}(k)=0$ otherwise;
\item $A(k)$ is doubly stochastic, which means $\onebf^{\top}A(k)=\onebf^{\top}$
and $A(k)\onebf=\onebf$
\end{enumerate}
Furthermore, all of above properties hold if we replace $A(k)$ with $A^{\tau}(k)$.
\end{assumption}

\begin{rem}
Several rules for selecting weights for local averaging have been proposed in the literature that satisfies
Assumption~\ref{assu:doubly_sto}~\citep{xiao2005scheme}. 
Examples include the Laplacian,
the Metropolis-Hasting, and the maximum-degree weights rules.
\end{rem}

The doubly stochastic matrix has some desirable properties~\citep{tsitsiklis1984problems,sun2022distributed} that will be used in our analysis. 
\begin{prop}\label{prop:mixing_matrix_prop}
Let $\tilde{A}(k)=A(k)\otimes\Ibf_{d},J=\frac{1}{m}\onebf_{m}\onebf_{m}^{\top}\otimes\Ibf_{d}$,
then:
\begin{enumerate}
\item $\tilde{A}(k)J=J=J\tilde{A}(k)$.
\item $\tilde{A}(k)\bar{\zbf}^{k}=\bar{\zbf}^{k}=J\bar{\zbf}^{k},\forall\bar{z}^{k}\in\mbb R^{d}$.
\item There exists $\rho$ such that $\max_{k}\bcbra{\norm{\tilde{A}(k)-J}}\leq\rho<1$.
\item $\norm{\tilde{A}(k)}\leq1$.
\end{enumerate}
\end{prop}
We remark that Proposition \ref{prop:mixing_matrix_prop} also hold by replacing $A(k)$ with $A^{\tau}(k)$.
\begin{algorithm}[htpb]
\caption{\label{alg:dgfmlin}{\dgfmlin} at each node $i$}
    \begin{algorithmic}[1]
        \Require{$x_{i}^{-1}=x_i^0=\bar{x}^0,\forall i \in [m]$, $K,\eta$}
        \State{\textbf{Initialize:} $y_i^0 = g_i(x_i^{-1};S_i^{-1})=\zerobf_{d}$}
        \For{$k=0,\ldots,K-1$} 
        \State{Sample $S_i^{k}=\bcbra{\xi_{i}^{k,1},w_i^{k,1}}$ and calculate $g_i(x_i^{k};S_i^k)$}
        \State{$y_i^{k+1}=\sum_{j=1}^{m}a_{i,j}(k)(y_j^k+g_j(x_j^k;S_j^k)-g_j(x_j^{k-1};S_j^{k-1}))$}
        \State{$x_i^{k+1}=\sum_{j=1}^{m}a_{i,j}(k)(x_j^k - \eta y_j^{k+1})$}
        \EndFor 
    \State{\textbf{Return:} Choose $x_{\text{out}}$ uniformly at random from $\bcbra{{x}^k_i}_{k=1,\cdots,K, i=1,\cdots,m}$}
    \end{algorithmic}
\end{algorithm}

\section{{\dgfmlin}}
In this section, we develop the decentralized Gradient-Free Method ({\dgfmlin}), an extension of the centralized zeroth-order method proposed by \citet{lin2022gradient}. 
In a multi-agent environment, a significant challenge lies in managing the consensus error among agents to match the order of the optimization error. 
To address this issue, \dgfmlin{} integrates a widely-used gradient tracking technique~\citep{di2016next,sun2022distributed,nedic2017achieving,lu2019gnsd} with the gradient-free method.
Specifically, for every node $i$, {\dgfmlin} contains the following three steps. Firstly, it sample a random direction $w_i^k$ and a data point $\xi_i^k$, and calculate the corresponding zeroth-order oracle estimator~$g_i(x_i^k; S_i^k)$, where $S_i^k=(w_i^k,\xi_i^k)$. 
Secondly, the gradient tracking technique is applied to monitor the zeroth-order oracle estimator of the overall function. Lastly, the primal variable is updated using a perturbed and locally weighted average. 
We present the details in Algorithm~\ref{alg:dgfmlin}.

We first establish some basic properties of the ergodic sequences, especially for the average sequences.
\begin{lem}\label{lem:dgfmlin}
Let $\{x_{i}^{k},y_{i}^{k},g_{i}(x_{i}^{k};S_{i}^{k})\}$ be the sequence
generated by {\dgfmlin} and $\{\xbf^{k},\ybf^{k},\gbf^{k}\}$ be
the corresponding stack variables, then we have
\begin{align*}
& \xbf^{k+1} =\tilde{A}(k)(\xbf^{k}-\eta\ybf^{k+1}),\ \ \bar{\xbf}^{k+1} =\bar{\xbf}^{k}-\eta\bar{\ybf}^{k+1},\ \ \bar{\ybf}^{k+1} =\bar{\ybf}^{k}+\bar{\gbf}^{k}-\bar{\gbf}^{k-1},  \\ 
& \bar{\ybf}^{k+1} =\bar{\gbf}^{k} \quad\text{and}\quad
\ybf^{k+1}  =\tilde{A}(k)(\ybf^{k}+\gbf^{k}-\gbf^{k-1}),
\end{align*}
where $\bar{x}^k=\frac{1}{m}\sum_{i=1}^m x_i^k, \bar{\xbf}^k=\onebf_m\otimes \bar{x}^k, \bar{y}^k=\frac{1}{m}\sum_{i=1}^m y_i^k$ and $\bar{\ybf}^k=\onebf_m\otimes \bar{y}^k$.
\end{lem}
Lemma~\ref{lem:dgfmlin} implies that {\dgfmlin} effectively executes an approximate gradient descent on the consensus
sequence, utilizing the variable $y$ to track gradients for approximating the overall gradient. 
Subsequently, we present the main convergence results of {\dgfmlin} and discuss the relevant features. The result of consensus error decay at exponential order is given in Lemma~\ref{lem:lin_consensus}.
\begin{lem}[Consensus error decay]\label{lem:lin_consensus}
Let $\{x_{i}^{k}\allowbreak,y_{i}^{k},\allowbreak g_{i}(x_{i}^{k};\allowbreak S_{i}^{k}\allowbreak)\allowbreak\}$ be the sequence
generated by {\dgfmlin} and $\{\xbf^{k},\ybf^{k},\gbf^{k}\}$ be
the corresponding stack variables, then for $k\geq0$, there exist positive $\alpha_1$ and $\alpha_2$ such that
\vspace{-0.5em}
\begin{align*}
\mbb E\bsbra{\norm{\xbf^{k+1}-\bar{\xbf}^{k+1}}^{2}}\leq  \rho^{2}(1+\alpha_{1})\mbb E\bsbra{\norm{\xbf^{k}-\bar{\xbf}^{k}}^{2}}
+\rho^{2}\eta^{2}(1+  \alpha_{1}^{-1})\mbb E\bsbra{\norm{\ybf^{k+1}-\bar{\ybf}^{k+1}}^{2}},
\end{align*}
and
\begin{align*}
&\mbb E\bsbra{\norm{\ybf^{k+1}-\bar{\ybf}^{k+1}}^{2}\mid\mcal F_{k}}\\
& \leq\rho^{2}(1+\alpha_{2})\mbb E\bsbra{\norm{\ybf^{k}-\bar{\ybf}^{k}}^{2}\mid\mcal F_{k}}+\rho^{2}(1+\alpha_{2}^{-1})  (6+36{\eta^{2}}L_{\delta}^{2})m\sigma^{2}\\
&\ \ +\theta_{1}  \mbb E\bsbra{\norm{\nabla f_{\delta}(\bar{x}^{k-1})}^{2}\mid\mcal F_{k}}+\theta_{2}\mbb E\bsbra{\norm{\xbf^{k}\!-\!\bar{\xbf}^{k}}^{2}\mid\mcal F_{k}}+\theta_{3} \mbb E\bsbra{\norm{\xbf^{k-1}\!-\!\bar{\xbf}^{k-1}}^{2}\mid\mcal F_{k}},
\end{align*}
where $\theta_{1}=18L_{\delta}^{2}\eta^{2}{\rho^{2}}(1+\alpha_{2}^{-1})$, $\theta_{2}=9\rho^{2}(1+\alpha_{2}^{-1})L_{\delta}^{2}$, 
 $\theta_{3}=9\rho^{2}(1+\alpha_{2}^{-1}L_{\delta}^{2}(1+4\eta^{2}L_{\delta}^{2})$ and $L_{\delta}={cL_{f}\sqrt{d}}{\delta}^{-1}$ is the smoothness of $f_{\delta}(\cdot)$.
\end{lem}
Lemma~\ref{lem:lin_consensus} indicates that if we omit some additional error term other than $\mbb E [\norm{\xbf^k - \bar{\xbf}^k}^2]$ and $\mbb E [\norm{\ybf^k - ~\bar{\ybf}^k}^2]$, there is a factor $\rho^2(1+\alpha_{i}),i=1,2,$ between two successive consensus error terms. Since $\rho$ is smaller than 1, we can choose a suitable $\alpha_i$ (such as $(1-\rho^2)/(2\rho^2)$) so that the factor $\rho^2(1+\alpha_{i})<1$, thereby achieving an exponential decrease in the consensus error. 
Therefore, we can show the descent property of  $\bcbra{\bar{x}^k}$ in the following Lemma~\ref{lem:decent_lin} and combine it with Lemma~\ref{lem:lin_consensus} to get the descent property of the overall sequence in Lemma~\ref{lem:linthm}.
\begin{lem}
\label{lem:decent_lin}Let $\{\xbf^{k},\ybf^{k},\bar{\xbf}^{k},\bar{\ybf}^{k}\}$
be the sequence generated by {\dgfmlin}, then we have
\vspace{-0.5em}
\begin{align*}
\mbb E\bsbra{f_{\delta}(\bar{x}^{k+1})}-\mbb E\bsbra{f_{\delta}(\bar{x}^{k})} \leq-\frac{\eta}{2}\mbb E\bsbra{\norm{\nabla f_{\delta}(\bar{x}^{k})}^{2}} 
+\frac{\eta L_{\delta}^{2}}{2m}\mbb E\bsbra{\norm{\xbf^{k}-\bar{\xbf}^{k}}^{2}}+L_{\delta}\eta^{2}(\sigma^{2}+L_{f}^{2}).
\end{align*}
\end{lem}
We obtain Lemma~\ref{lem:linthm} by multiplying the two consensus errors in Lemma~\ref{lem:lin_consensus} by their corresponding factors $\beta_x$ and $\beta_y$ and adding them to the result of Lemma~\ref{lem:decent_lin}.
\begin{lem}[Informal]\label{lem:linthm}
Let $\{\xbf^{k},\ybf^{k},\bar{\xbf}^{k},\bar{\ybf}^{k}\}$
be the sequence generated by {\dgfmlin}, then there exist positive constants  $\beta_{x}$ and $\beta_{y}$ such that 
\begin{align*}
 & \theta_{4}\sum_{k=0}^{K-1}\mbb E\bsbra{\norm{\nabla f_{\delta}(\bar{x}^{k})}^{2}}+\left(\theta_{5}-\theta_{6}\right)\sum_{k=0}^{K-1}\mbb E\bsbra{\norm{\xbf^{k}-\bar{\xbf}^{k}}^{2}}\\
 & \leq \mbb E\bsbra{f_{\delta}(\bar{x}^{0})} -\mbb E[f_{\delta}(\bar{x}^{K})] +(\theta_6-\beta_{x})\mbb E\bsbra{\norm{\xbf^{K}-\bar{\xbf}^{K}}^{2}}\\
 &\ \ -\theta_{7}\sum_{k=1}^{K}\mbb E\bsbra{\norm{\ybf^{k}-\bar{\ybf}^{k}}^{2}}+\theta_{8}-\beta_{y}\mbb E\bsbra{\norm{\ybf^{{K+1}}-\bar{\ybf}^{K+1}}^{2}-\norm{\ybf^{{1}}-\bar{\ybf}^{1}}^{2}},
\end{align*}
where constants $\theta_{4},\theta_{5},\theta_{6},\theta_{7}$ and $\theta_{8}$ depend on $\alpha_{1},\alpha_{2},\beta_{x},\beta_{y},d,\eta,L_{\delta},L_{f},m$ and $\rho$.
\end{lem}
Observe that Lemma~\ref{lem:linthm} can give an upper-bound of 
$\mbb E\bsbra{\norm{\nabla f_{\delta}(\bar{x}^k)}^2}+(\theta_5 - \theta_6){\theta_4}^{-1}\mbb E\bsbra{\norm{x^k-\bar{x}^k}^2}$ when we choose the proper parameters of $\alpha_1,\alpha_2,\beta_x,\beta_y,\eta$ and $\delta$ such that $\theta_i>0$ for $i=4,\cdots,8$ and it holds that $(\theta_5 - \theta_6){\theta_4}^{-1}=\mcal O(\delta^{-2})$. Since $\mbb E\bsbra{\norm{\nabla f_{\delta}
(x_i^k)}^2}\leq 2L_{\delta}^{2}\mbb E\bsbra{\norm{x_i^k - \bar{x}^k}^2} + 2\mbb E\bsbra{\norm{\nabla f_{\delta}(\bar{x}^k)}^2}$, then we can give an upper bound of $\mbb E\bsbra{\norm{\nabla f_{\delta}(x^k_i)}^2}$, which naturally leads to the complexity results of 
finding $(\delta,\vep)$-stationary points with simple calculation.
Next,  we present more specific parameter settings, leading to the optimal convergence rate.


\begin{thm}[Informal]\label{thm:lin_order}
{\dgfmlin} outputs a $(\delta,\vep)$-Goldstein stationary point of $f(\cdot)$ in expectation with the total stochastic zeroth-order complexity and total communication complexity  at most $\mcal O(\Delta_{\delta} \delta^{-1}\vep^{-4}d^{\frac{3}{2}})$ by setting 
    $\alpha_1 = \alpha_2 = {(1-\rho^2)}/{2\rho^2}$, $\delta = \mcal O(\vep)$, $\beta_x = \mcal O(\delta^{-1}),\beta_y = \mcal O(\vep^4\delta)$ and $\eta=\mcal O(\vep^2 \delta)$.
\end{thm}

\section{\label{sec:dgfm_plus}{\dgfm}}

In this section, we consider {\dgfm}.
First, we give some preliminaries in the following section.
\begin{algorithm}[ht]
\caption{\label{alg:dgfm}{\dgfm} at each node $i$}
    \begin{algorithmic}[1]
        \Require{$x_{i}^{-1}=x_i^0=\bar{x}^0,y_i^0 = v_i^{-1}=\zerobf_{d},\forall i \in [m]$, $K,\eta,b,b^{\prime}$}
        \For{$k=0,\ldots,K-1$}
            \If{$k$ mod $T$ = 0}
            \State Sample $S_i^{k\prime}=\bcbra{(\xi_i^{k\prime,j},w_i^{k\prime,j})}_{j=1}^{b^\prime}$
            \State Calculate $y_i^{k+1} = v_i^k = g_i(x_i^k;S_i^k)$
            \For{$\tau = 1,\cdots,\mcal T$}
            \State $y_i^{k+1}=\sum_{j=1}^m a_{i,j}^\tau(k)y_j^{k+1}$ 
        \EndFor
        \Else
        \State Sample $S_i^{k}=\bcbra{(\xi_i^{k,j},w_i^{k,j})}_{j=1}^{b}$
        \State  $v_i^k = v_i^{k-1}+g_i(x_i^k;S_i^k)-g_i(x_i^{k-1};S_i^k)$
        \State $y_i^{k+1}=\sum_{j=1}^{m}a_{i,j}(k)(y_j^k + v_j^k - v_j^{k-1})$ 
        \EndIf
        \State $x_i^{k+1}=\sum_{j=1}^{m}a_{i,j}(k)(x_j^k - \eta y_j^{k+1})$
        \EndFor
    \State{\textbf{Return:} Choose $x_{\text{out}}$ uniformly at random from $\bcbra{{x}^k_i}_{k=1,\cdots,K, i=1,\cdots,m}$}
        \end{algorithmic}
\end{algorithm}

\noindent\textbf{Preliminaries for {\dgfm}. }
Mini-batch zeroth-order oracle estimator plays a key role in {\dgfm}. 
The variance of the gradient estimator can be reduced by increasing the batch size.  Furthermore, the smoothness merit of randomized smoothing for mini-batch zeroth-order oracle estimator is still established. All these properties are stated in the following Proposition~\ref{prop:mini-batch-variance-reduce}.
%
\begin{prop}[Corollary 2.1 and Proposition 2.4~\citep{chen2023faster}]
\label{prop:mini-batch-variance-reduce}
Under Assumption~\ref{assu:2.2}, it holds that $\mbb E_{S}\bsbra{\norm{g_{i}(x;S)-\nabla f_{\delta}^{i}(x)}^{2}}\leq  \sigma^{2}/b$, where $\sigma^2 = 16\sqrt{2\pi}dL_{f}^{2}$. Furthermore, for any $w\in\mbb S^{d-1}$ and $x,y\in\mbb R^{d}$,
it holds that $
\mbb E_{\xi}\bsbra{\norm{g_{i}(x;w,\xi)-g_{i}(y;w,\xi)}^{2}}\leq d^{2}L_{f}^{2}\delta^{-2}\|x-y\|^{2}$.
\end{prop}
%

\noindent\textbf{Algorithm and Convergence Analysis. }In {\dgfm}, we divide all the $K$ iterations into $R$ cycles, each containing $T$ iterations. In
the first iteration of each cycle, we sample a mini-batch of size $b^{\prime}$
to compute the stochastic gradient, and then use batchsize of $b$  for the
rest iterations. In addition, we use the SPIDER~\citep{fang2018spider}
method to track the gradient for variance reduction. For the decentralized setting, we perform gradient
tracking on the obtained gradients to approximate the gradient of
the finite sum function and finally perform gradient descent on
the variable $x$. It is worth noting that we need to restart the gradient tracking at the beginning of each cycle.
To reduce the consensus error, we perform frequent fast communication $\mcal T$ times. {\dgfm} is given in Algorithm~\ref{alg:dgfm}.

\begin{lem}\label{lem:iterative_dgfmplus}
Let $\{x_{i}^{k},y_{i}^{k},g_{i}^{k},v_{i}^{k}\}$ be the sequence
generated by {\dgfm} and $\{\xbf^{k},\ybf^{k},\gbf^{k},\vbf^{k}\}$
be the corresponding stack variables, then for $k\geq0$, we have
\begin{align*}
\xbf^{k+1} =\tilde{A}(k)(\xbf^{k}-\eta\ybf^{k+1}), \quad
\bar{\xbf}^{k+1} =\bar{\xbf}^{k}-\eta\bar{\ybf}^{k+1} \quad \text{and}  \quad
\bar{\ybf}^{k+1}  =\bar{\vbf}^{k}.
\end{align*}
Furthermore, for $rT<k<(r+1)T$, $r=0,\cdots,R-1$, we have 
\vspace{-0.5em}
\begin{align*}
\ybf^{k+1} =\tilde{A}(k)(\ybf^{k}+\vbf^{k}-\vbf^{k-1}) \quad \text{and}  \quad
\bar{\ybf}^{k+1} =\bar{\ybf}^{k}+\bar{\vbf}^{k}-\bar{\vbf}^{k-1}.
\end{align*}
\end{lem}
The purpose of restart gradient tracking is to ensure that $\bar{\ybf}^{k+1}=\bar{\vbf}^k$ holds throughout the entire sequence. Additionally, it helps to truncate the accumulation of the variance bound described in Lemma~\ref{lem:spider_variance_bound}. However, this operation may introduce a consensus error at the beginning of each cycle. Inspired from~\citep{luo2022optimal,chen2022simple}, we perform multiple rounds of communication before the start of each cycle. This will be further explained in detail in Lemma~\ref{lem:consensus-1}.
Next, we give some consensus results for {\dgfm}.
\begin{lem}
\label{lem:consensus-1}For sequence $\{\xbf^{k},\ybf^{k}\}$ generated
by {\dgfm}, we have
\begin{equation*}
    \begin{aligned}
        &\mbb E\bsbra{\norm{\ybf^{k+1}-\bar{\ybf}^{k+1}}^{2}}\\ & \leq  \rho^{2}(1+\alpha_{1})\mbb E\bsbra{\norm{\ybf^{k}-\bar{\ybf}^{k}}^{2}}
+3\eta^{2}\Brbra{\frac{\rho^{2}L_{\delta}^{2}}{\alpha_{1}} +\frac{\rho^{2}d^{2}L_{f}^{2}}{b\delta^{2}}}\mbb E\bsbra{\norm{\bar{\vbf}^{k-1}}^{2}}\\
&\ \  +\theta_{9}\mbb E\bsbra{\norm{\xbf^{k}-\bar{\xbf}^{k}}^{2}+\norm{\bar{\xbf}^{k-1}-\xbf^{k-1}}^{2}},\\ 
&\ \ rT+1\leq k<(r+1)T,\forall r=0,1,\cdots,R-1, \\[0.1cm]
    \end{aligned}
\end{equation*}
and 
\begin{equation*}
\begin{aligned}
    & \mbb E\bsbra{\norm{\xbf^{k+1}-\bar{\xbf}^{k+1}}^{2}}\leq (1+\alpha_{2})\rho^{2}\mbb E\bsbra{\norm{\xbf^{k}-\bar{\xbf}^{k}}^{2}}
+\theta_{10} \mbb E\bsbra{\norm{\ybf^{k+1}-\bar{\ybf}^{k+1}}^{2}},\forall k\geq0,\\[0.1cm]
&\mbb E\bsbra{\norm{\ybf^{rT+1}-\bar{\ybf}^{rT+1}}^{2}} \leq{2\rho^{\Tcal}m(\sigma^{2}+L_{f}^{2})},
\forall r= 0,\cdots,R-1,
\end{aligned}    
\end{equation*}
where $\theta_{9}=3({\rho^{2}L_{\delta}^{2}}/{\alpha_{1}}+{\rho^{2}d^{2}L_{f}^{2}}/{\delta^{2}})$
and $\theta_{10}=(1+\alpha_{2}^{-1})\rho^{2}\eta^{2}$.
\end{lem}

Similar to the proof process of {\dgfmlin}, we will present the descent properties of the average sequence in Lemma~\ref{lem:dgfm_one_step_improve}. Since SPIDER is used here, we also provide an upper bound for variance in Lemma~\ref{lem:spider_variance_bound}.
\begin{lem}\label{lem:dgfm_one_step_improve}
For the sequence
$\{\bar{x}^{k},\bar{v}^{k}\}$ generated by {\dgfm}
and $f_{\delta}(x)=~\frac{1}{m}\sum_{i=1}^{m}f_{\delta}^{i}(x)$,
we have
\vspace{-0.5em}
\begin{align*}
f_{\delta}(\bar{x}^{k+1})\leq f_{\delta}(\bar{x}^{k}) -\frac{\eta}{2}\norm{\nabla f_{\delta}(\bar{x}^{k})}^{2}
-\big(\frac{\eta}{2}-\frac{L_{\delta}\eta^{2}}{2}\big) \norm{\bar{v}^{k}}^{2}+\frac{\eta}{2}\norm{\nabla f_{\delta}(\bar{x}^{k})-\bar{v}^{k}}^{2}.
\end{align*}
\end{lem}

\begin{lem}\label{lem:spider_variance_bound}
Let $\{\bar{x}^{k},\bar{v}^{k}\}$ be the sequence 
generated by {\dgfm}, then for $rT\leq k^{\prime}<k\leq(r+1)T-1$,
we have
\vspace{-0.5em}
\begin{align*}
\mbb E\big[\norm{\bar{v}^{k}-\nabla f_{\delta}(\bar{x}^{k})}^{2}\big]  \leq & \frac{2L_{\delta}^{2}}{m}\mbb E\bsbra{\norm{\xbf^{k}-\bar{\xbf}^{k}}^{2}}
+\frac{6d^{2}L_{f}^{2}}{m^{2}\delta^{2}b}\sum_{j=k^{\prime}}^{k}\mbb E\big[\|\xbf^{j} -\bar{\xbf}^{j}\|^{2}+\norm{\xbf^{j-1}-\bar{\xbf}^{j-1}}^{2}\big]\\
&+2\mbb E\lrsbra{\Bnorm{\bar{v}^{k^{\prime}}  -\frac{1}{m}\sum_{i=1}^{m}\nabla f_{\delta}^{i}(x_{i}^{k^{\prime}})}^{2}}
+\frac{6\eta^{2}d^{2}L_{f}^{2}}{m^{2}\delta^{2}b} \sum_{j=k^{\prime}}^{k}\mbb E\big[\norm{\bar{\vbf}^{j-1}}^{2}\big].
\end{align*}
Moreover, for $k=rT$, $r=0,\cdots,R-1$, we have $
\mbb E\bsbra{\bnorm{\bar{v}^{k}-\frac{1}{m}\sum_{i=1}^{m}\nabla f_{\delta}^{i}(x_{i}^{k})}^{2}}\leq {\sigma^{2}}/{b^{\prime}}.
$
\end{lem}

By multiplying the consensus error descent of $x,y$ shown in Lemma~\ref{lem:consensus-1} with their corresponding coefficients, $\beta_x$ and $\beta_y$, and combining it with the results of mean sequence (Lemma~\ref{lem:dgfm_one_step_improve}) and variance upper bound (Lemma~\ref{lem:spider_variance_bound}), we can obtain the following Lemma~\ref{lem:dgfm_convergence_thm}.
\begin{lem}[Informal]\label{lem:dgfm_convergence_thm}
For the sequence $\{\xbf^{k},\ybf^{k}\}$ generated by {\dgfm},
we have
\begin{align*}
& \frac{\eta}{2}\sum_{k=0}^{RT-1}\mbb E\bsbra{\bnorm{\nabla f_{\delta}(\bar{x}^{k})}^{2}}+\theta_{11}\sum_{k=0}^{RT-1}\mbb E\bsbra{\norm{\xbf^{k}-\bar{\xbf}^{k}}^{2}}\\[0.1cm]
&\leq-\theta_{12}\sum_{k=0}^{RT-1} \mbb E\bsbra{\norm{\bar{v}^{k}}^{2}}-\beta_{x}\mbb E\bsbra{\norm{\xbf^{RT}-\bar{\xbf}^{RT}}^{2}}
-\theta_{13}\sum_{k=0}^{RT-1}\mbb E\bsbra{\norm{\ybf^{k}-\bar{\ybf}^{k}}^{2}}-\theta_{14}\mbb E\bsbra{\norm{\ybf^{RT}-\bar{\ybf}^{RT}}^{2}}\\
&\ \ -\mbb E\bsbra{f_{\delta}(\bar{x}^{RT})} +\mbb E[f_{\delta}(\bar{x}^{0})]+{2\rho^{\Tcal}m(\sigma^{2}+L_{f}^{2})}\cdot\beta_{y}R+\theta_{15},
\end{align*}
with constants    $\left(\theta_{11},\theta_{12},\theta_{13},\theta_{14},\theta_{15}\right)$
depend on
$\left(\alpha_{1},\alpha_{2},b,\beta_{x},\beta_{y},d,\eta,L_{\delta},L_{f},R,m,\rho,T\right)$.
\end{lem}
Similar to Lemma~\ref{lem:linthm}, Lemma~\ref{lem:dgfm_convergence_thm} gives an upper bound of $\mbb  E\bsbra{\norm{\nabla f_{\delta}(\bar{x}^k)}^2}+({\theta_{11}}/{\eta})\cdot \mbb E\bsbra{\norm{\xbf^k-\bar{\xbf}^k}^2},$
which can be used to derive the complexity of Goldstein stationary point $x_i^k$ in expectation. We present specific parameter setting in the following Theorem.

\begin{thm}[Informal]\label{thm:dgfm_cor}
{
    {\dgfm} can output a $(\delta,\vep)$-Goldstein stationary point of $f(\cdot)$ in expectation with total stochastic zeroth-order complexity at most
$\mcal O(\Delta_{\delta}\delta^{-1}\vep^{-2}d^{\frac{1}{2}}m^{\frac{1}{2}}\cdot \max\{1,\frac{d}{m\vep}\})$ and the total communication rounds is at most \
$\mcal O\brbra{(\vep^{-1}+\log (\vep^{-1}d))\cdot \Delta_{\delta}\vep^{-2}d^{\frac{1}{2}}m^{\frac{1}{2}}}$
by setting
    $\alpha_1 = \alpha_2 = (1-\rho^2)/{2\rho^2}$, $\delta=\mcal O(\vep)$, $\beta_x = \Ocal(\delta^{-1})$, $\beta_y = \Ocal(\delta)$, $b^{\prime} = \Ocal(\vep^{-2}), b=\mcal O(\vep^{-1})$, $\eta = \mcal O(\vep)$, $T=\mcal O(\vep^{-1})$, $R=\mcal O(\Delta_{\delta}\vep^{-2})$, $\mcal{T} = \mcal O(\log(\vep^{-1}))$.}
\end{thm}

\begin{rem}
    Compared to {\dgfmlin}, {\dgfm} requires less evaluations of zeroth-order oracles to achieve the same level of accuracy, which is consistent with the findings of~\citep{chen2023faster}. Additionally, {\dgfm} involves significantly fewer communication rounds than {\dgfmlin}, and the number of communication rounds required is of the same order as the number of iterations $K$.
\end{rem}

\begin{rem}
    If we do not use multiple rounds of communication during the restart of gradient tracking and only perform communication once, i.e., $\mcal T = 1$, we can achieve the same complexity result to {\dgfmlin} by setting the parameters appropriately, i.e., $\alpha_1=\alpha_2 = (1-\rho^2)/2\rho^2, \delta = \Ocal(\vep), \beta_x = \Ocal (\vep^{-2}), \beta_y = \mcal{O} (\vep^{2})$, $b^{\prime} = \mcal{O}(\vep^{-2})$, $b = \Ocal(1)$,  $T = \Ocal (\vep^{-2})$, $R=\Ocal (\vep^{-2})$ and $\eta = \Ocal (\vep^{2})$. However,   this parameter setting will increase the number of communication rounds to $\mcal O(\delta^{-1}\vep^{-3})$, which is one order of magnitude higher than the current result in Theorem~\ref{thm:dgfm_cor}.
\end{rem}

\section{Numerical Study}

In this section, we show the outperformance of  {\dgfmlin} and 
{\dgfm}  via some numerical experiments.

\subsection{Nonconvex SVM with Capped-$\ell_1$ Penalty}

The first experiment considers the model of penalized Support Vector Machines (SVM). We aim to find a hyperplane to separate data points into two categories. To enhance the robustness of the classifier, we introduce the non-convex and non-smooth regularizers.

\paragraph{Data:}We evaluate our proposed algorithms using several standard datasets in LIBSVM~\citep{chang2011libsvm}, which are described in Table~\ref{tab:dataset_descript}. The feature vectors of all datasets are normalized before optimization.

\begin{table}[t]
  \centering
    \begin{tabular}{crr|crr}
    \toprule
    Dataset & \multicolumn{1}{c}{$n$}     & \multicolumn{1}{c|}{$d$}     & Dataset & \multicolumn{1}{c}{$n$}     & \multicolumn{1}{c}{$d$} \\
    \midrule
    a9a   & 32,561 & 123   & w8a   & 49,749 & 300 \\
    HIGGS & 11,000,000 & 28    & covtype & 581,012 & 54 \\
    rcv   & 20,242 & 47,236 & SUSY  & 5,000,000 & 18 \\
    ijcnn1 & 49,990 & 22    & skin\_nonskin & 245,057 & 3 \\
    \bottomrule
    \end{tabular}%
    \vspace{-0.5em}  \caption{\label{tab:dataset_descript}Descriptions of datasets used in our experiments.}
  \vspace{-0.5em}
\end{table}%

\paragraph{Model:}We consider the nonconvex penalized SVM with capped-$\ell_1$ regularizer~\citep{zhang2010analysis}. The model aims at training a binary classifier $x\in \mbb R^d$ on the training data $\{a_i,b_i\}_{i=1}^n$, where $a_i\in \mbb R^d$ and $b_i\in \{1,-1\}$ are the feature of the $i$-th sample and its label, respectively. For {\dgfm} and {\dgfmlin}, the objective function can be written as
\vspace{-0.5em}
\begin{equation*}
    \begin{aligned}
        \min_{x\in \mbb R^d} f(x)&=\frac{1}{m}\sum_{i=1}^m f^{i}(x),
    \end{aligned}
\end{equation*}
where $f^{i}(x)=\tfrac{1}{n_i}\tsum_{j=1}^{n_i} \ell(b_i^j (a_i^j)^{\top}x)+\gamma(x)$, $\ell(x)=\max\{1-x,0 \allowbreak \}\allowbreak,\gamma(x)\allowbreak=\lambda \allowbreak \tsum_{j=1}^d \allowbreak \min \allowbreak \{\abs{x_j}\allowbreak,\alpha\}$, $n=\tsum_{i=1}^{m}n_i$ and $\lambda,\alpha>0$. Similar to~\cite{chen2023faster}, we take $\lambda = 10^{-5}/n$ and $\alpha=2$.

\paragraph{Network topology:} We consider a simple ring-based topology of the communication network. We set the number of worker nodes to $m=20$. The setting can be found in~\cite{xian2021faster}.

\paragraph{Performance measures:} We measure the performance of the decentralized algorithms by the decrease of the global cost function value $f(\bar{x})$, to which we refer as loss versus zeroth-order gradient calls.

\paragraph{Comparison:} We compare the proposed {\dgfmlin} and {\dgfm} with GFM~\citep{lin2022gradient} and GFM$^{+}$~\citep{chen2023faster}. Throughout all the experiments, we set $\delta = 0.001$ and tune the stepsize $\eta$ from $\{0.0005, 0.001, 0.005, 0.01\}$ for all four algorithms and $b^{\prime}$ from $\{10,100,500\}$, $T$ from $\{10,50,100\}$ for {\dgfm} and GFM$^{+}$, $\mcal{T}$ from $\{1,5,10\}$ for {\dgfm}, $m=20$ for two decentralized algorithms.  We run all the algorithms with the same number of calls to the zeroth-order oracles. We plot the average of five runs in Figure~\ref{fig:res}.

In most of the experiments we conducted, our decentralized algorithms significantly outperform their serial counterparts. While DGFM sometimes exhibits slow convergence due to high consensus error, {\dgfm} consistently demonstrates more robust performance and often achieves the best results. In larger sample size test cases (SUSY, HIGGS), {\dgfm} often demonstrates notably faster convergence than \dgfmlin{}.  These outcomes further corroborate our theoretical analysis of \dgfmlin{} and \dgfm{}.
Additionally, we note that both {\dgfm} and GFM$^{+}$ have exhibited significantly stable performance, as seen in their smoother convergence curves, while DGFM and GFM show more fluctuation. This observation aligns with the theoretical advantage offered by the variance reduction technique.
\begin{figure}[t]
\begin{centering}
\begin{minipage}[t]{0.25\columnwidth}%
\includegraphics[width=4.15cm]{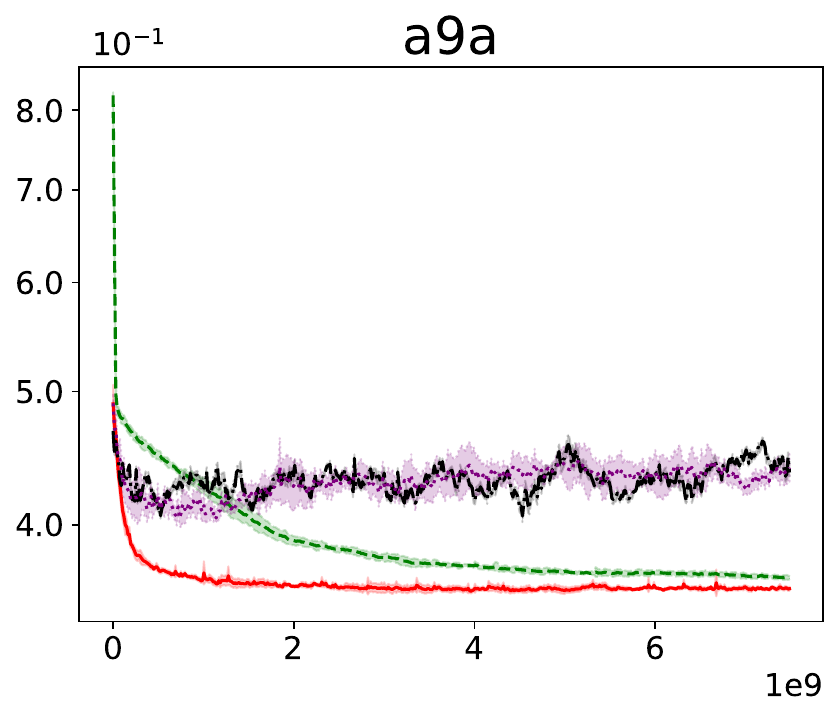}%
\end{minipage}\hfill
\begin{minipage}[t]{0.25\columnwidth}%
\includegraphics[width=4.15cm]{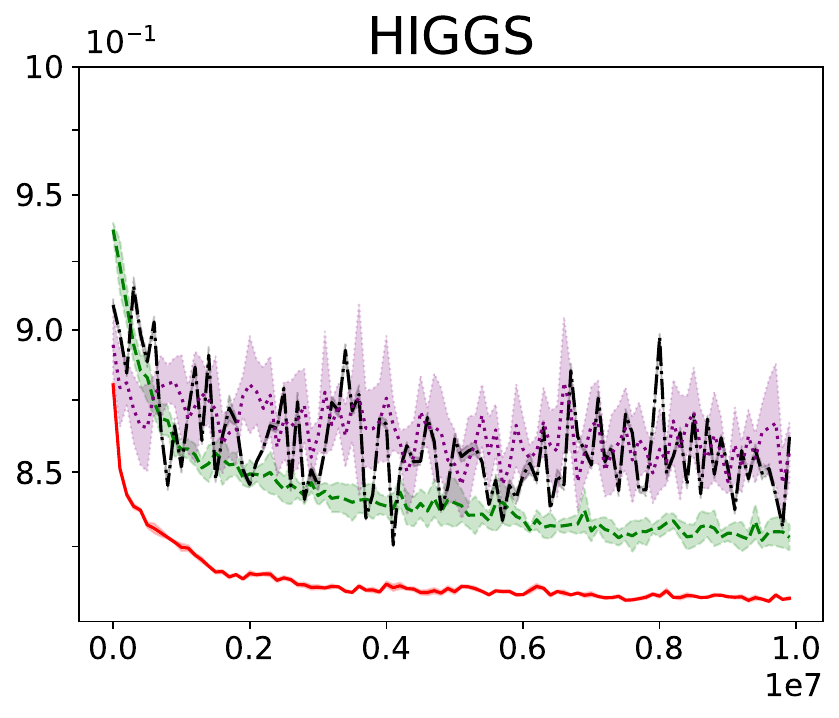}%
\end{minipage}\hfill
\begin{minipage}[t]{0.25\columnwidth}%
\includegraphics[width=4.15cm]{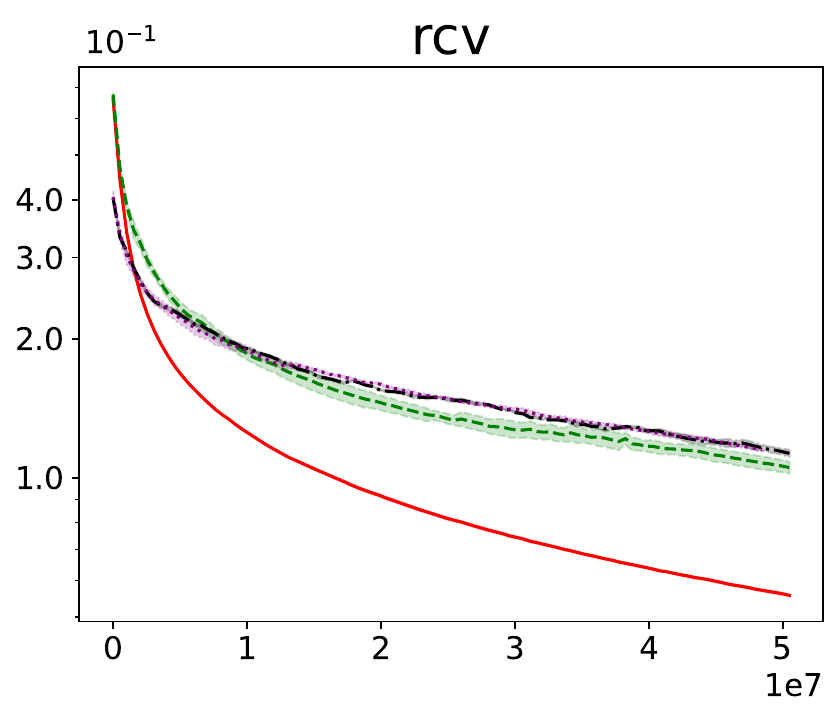}%
\end{minipage}\hfill
\begin{minipage}[t]{0.25\columnwidth}%
\includegraphics[width=4.15cm]{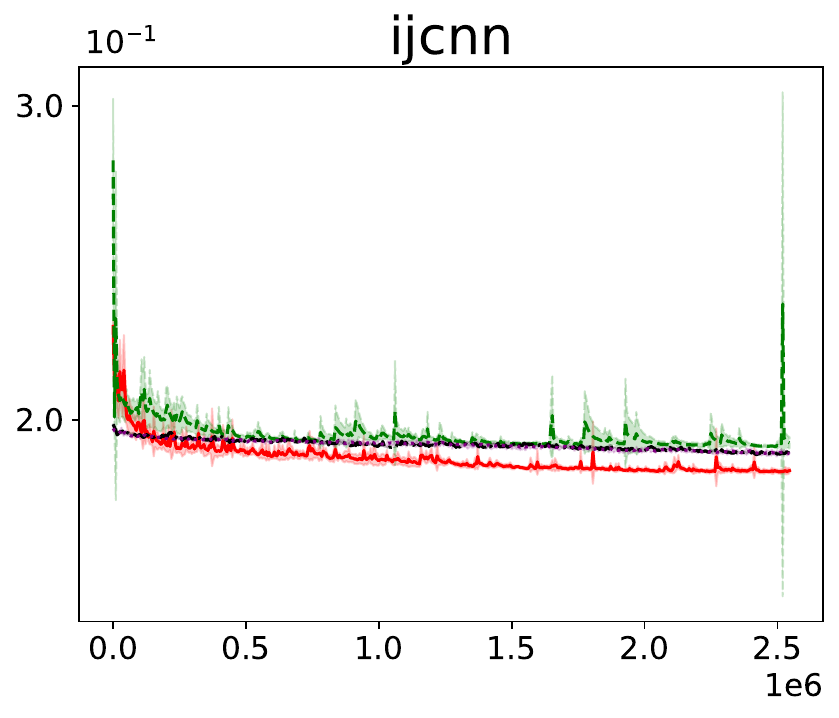}%
\end{minipage}
\\
\begin{minipage}[t]{0.25\columnwidth}%
\includegraphics[width=4.15cm]{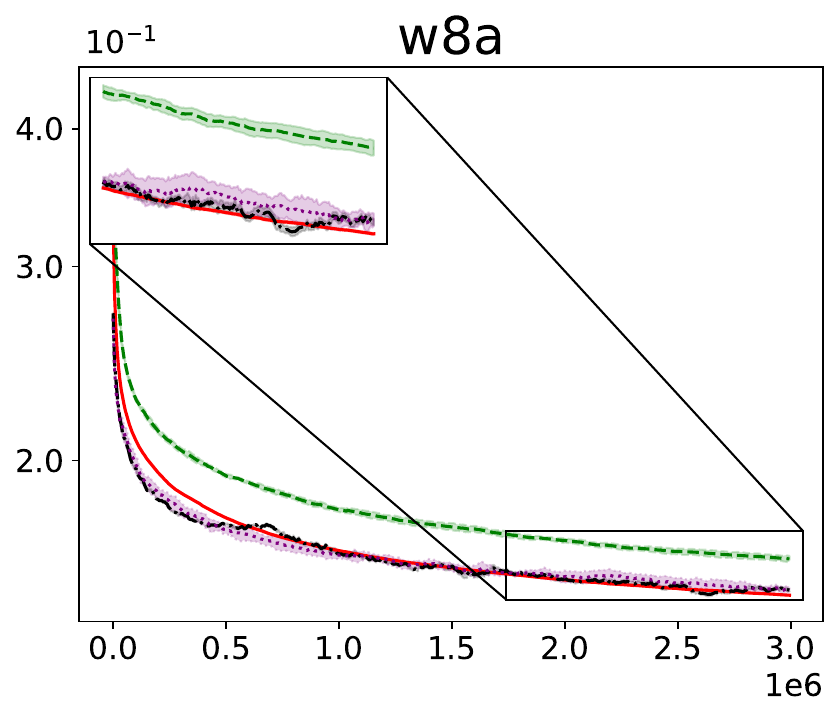}%
\end{minipage}\hfill
\begin{minipage}[t]{0.25\columnwidth}%
\includegraphics[width=4.15cm]{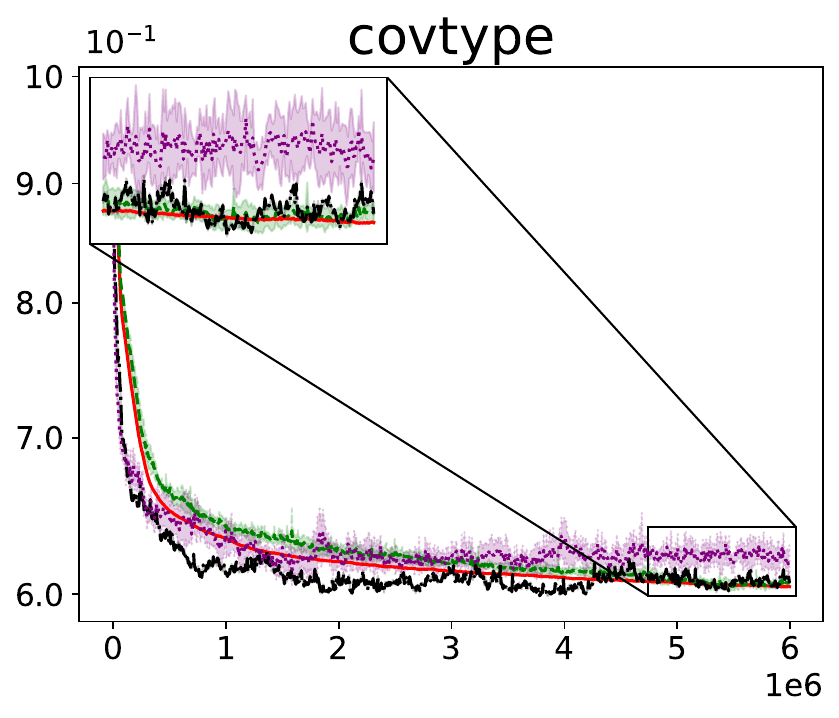}%
\end{minipage}\hfill
\begin{minipage}[t]{0.25\columnwidth}%
\includegraphics[width=4.15cm]{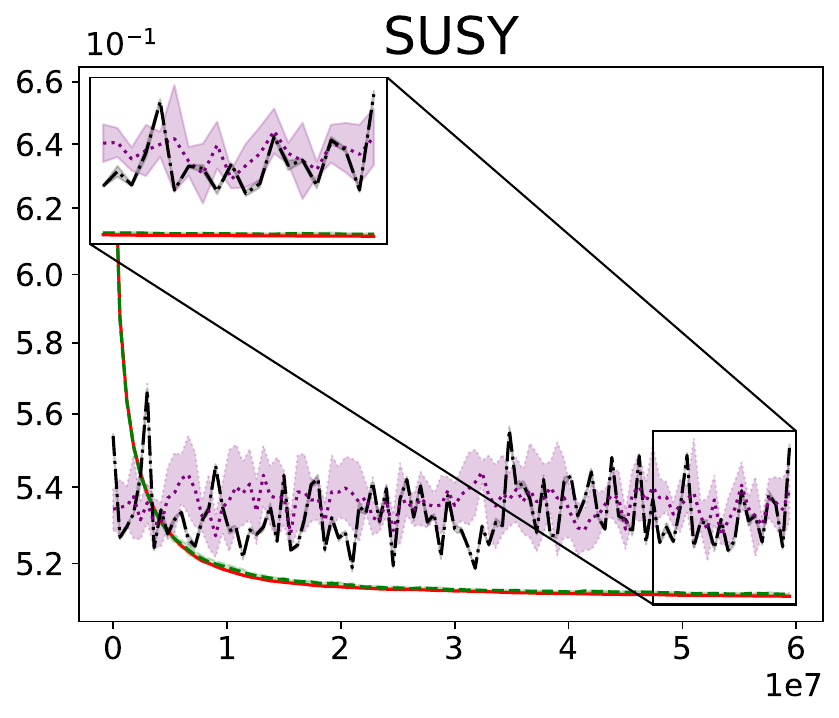}%
\end{minipage}\hfill
\begin{minipage}[t]{0.25\columnwidth}%
\includegraphics[width=4.15cm]{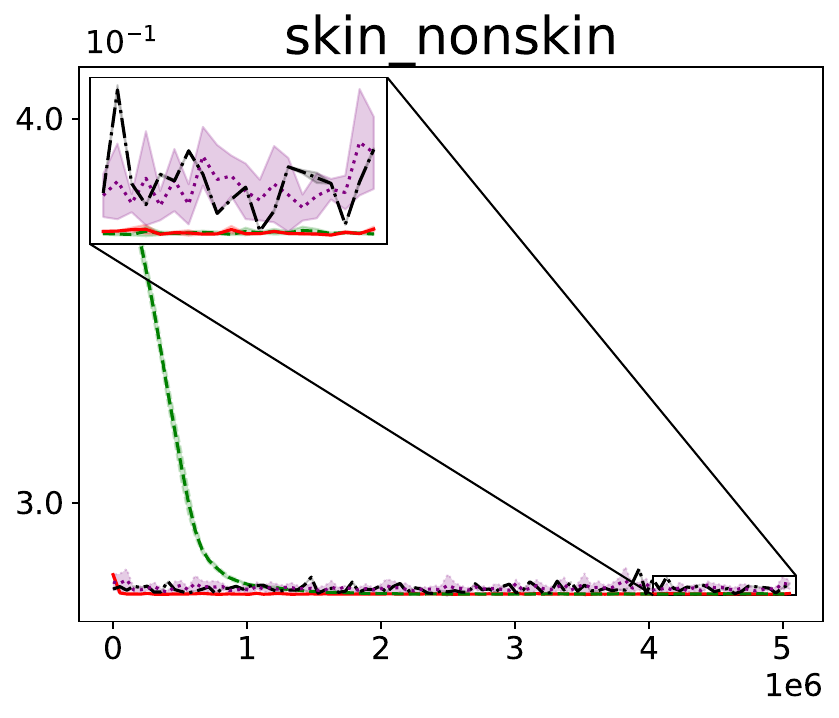}%
\end{minipage}
\\
\vspace{-0.8em}
    \begin{minipage}[t]{1\columnwidth}%
    \centering
    \includegraphics[width=10cm]{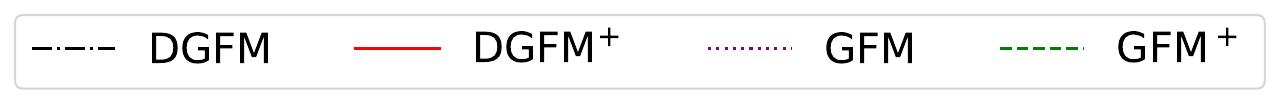}
\end{minipage}
\end{centering}
\vspace{-1.2em}
\caption{\label{fig:res}We assess the convergence performance of four algorithms by plotting the objective function value on the $y$-axis against the number of zeroth-order calls on the $x$-axis.}
\end{figure}


\begin{figure}[t]
\begin{centering}
\begin{minipage}[t]{0.25\columnwidth}%
\includegraphics[width=4.15cm]{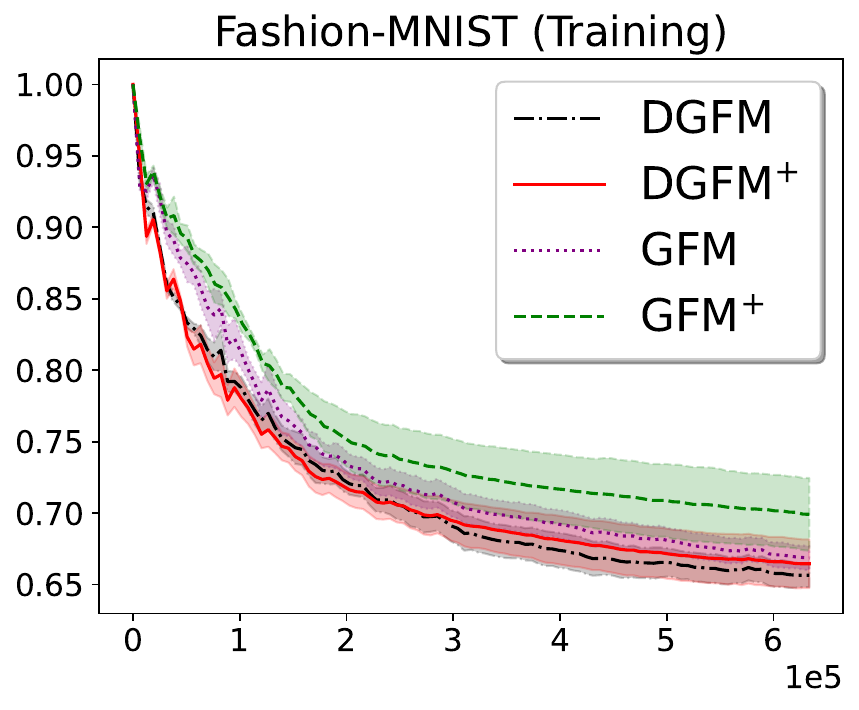}%
\end{minipage}\hfill
\begin{minipage}[t]{0.25\columnwidth}%
\includegraphics[width=4.15cm]{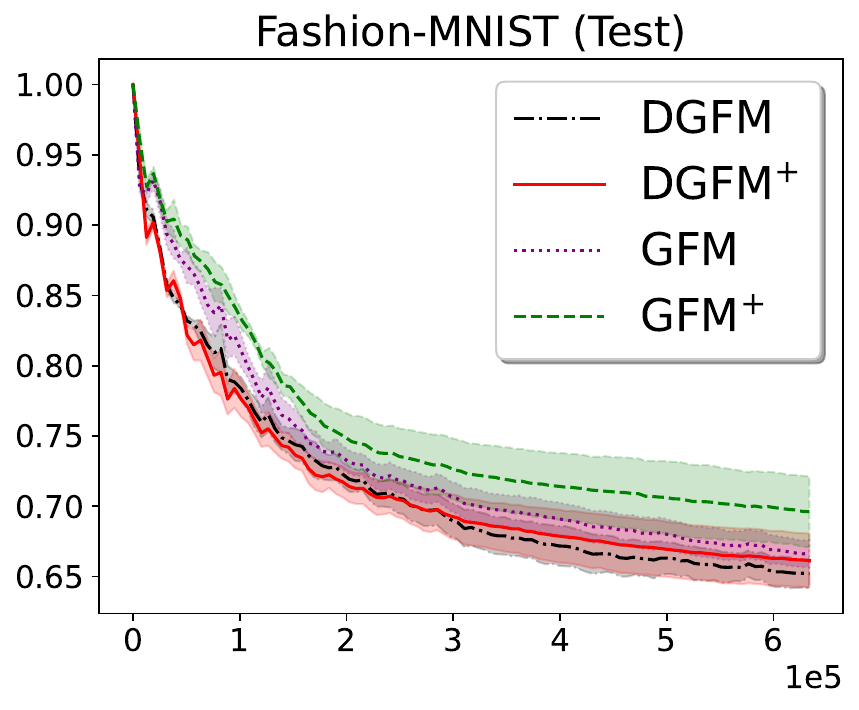}%
\end{minipage}\hfill
\begin{minipage}[t]{0.25\columnwidth}%
\includegraphics[width=4.15cm]{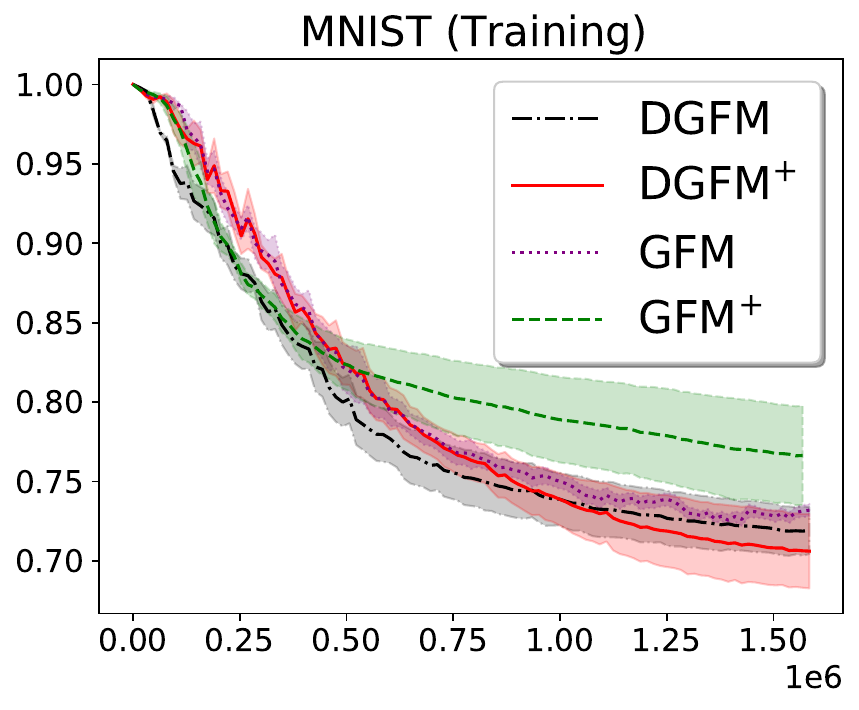}%
\end{minipage}\hfill
\begin{minipage}[t]{0.25\columnwidth}%
\includegraphics[width=4.15cm]{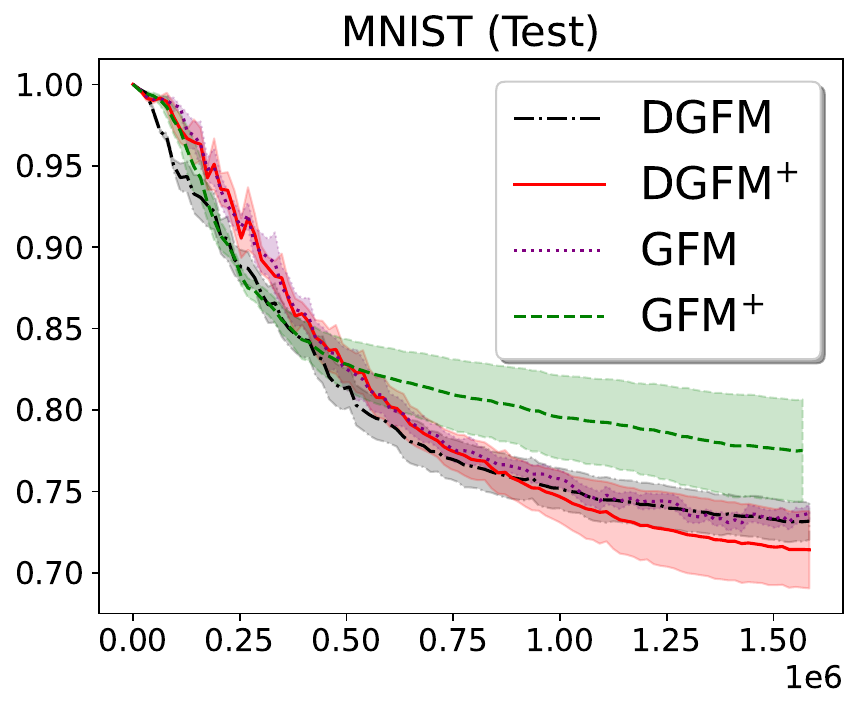}%
\end{minipage}
\end{centering}
\vspace{-1.2em}
\caption{\label{fig:res_attack}We assess the attacking performance of four algorithms by plotting the accuracy after attacking on the $y$-axis against the number of zeroth-order calls on the $x$-axis.}
\end{figure}

\subsection{Universal Attack}
We consider the black-box adversarial attack on image classification with LeNet~\citep{lecun1998gradient}. Our objective is to discover a universal adversarial perturbation~\citep{moosavi2017universal}. When applied to the original image, this perturbation induces misclassification in machine learning models while remaining inconspicuous to human observers. Network topology and Performance measures are the same as in the previous experiment. However, we set $m=8$ for two decentralized algorithms here.
\begin{table}[htbp]
  \centering
    \begin{tabular}{crrrr}
    \toprule
    \multirow{2}[4]{*}{Dataset} & \multicolumn{2}{c}{Fashion-MNIST} & \multicolumn{2}{c}{MNIST} \\
\cmidrule{2-5}          & \multicolumn{1}{c}{Training} & \multicolumn{1}{c}{Test}  & \multicolumn{1}{c}{Training} & \multicolumn{1}{c}{Test} \\
    \midrule
    DGFM  & \textbf{65.66}(\textbf{0.88}) & \textbf{65.22}(1.07) & 71.92(1.56) & 73.18(1.21) \\
    DGFM$^+$ & 66.46(1.80) & 66.14(2.03) & \textbf{70.62}(2.47) & \textbf{71.42}(2.50) \\
    GFM   & 66.86(0.90) & 66.64(\textbf{1.02}) & 73.20(\textbf{0.41}) & 73.66(\textbf{0.55}) \\
    GFM$^{+}$ & 69.92(2.67) & 69.62(2.63) & 76.64(3.29) & 77.52(3.29) \\
    \bottomrule
    \end{tabular}%
  \caption{\label{tab:accuracy_attack}Attacking result (Accuracy/std\%)}
\end{table}
\vspace{-0.5em}
\paragraph{Data:}We evaluate our proposed algorithms using two standard datasets, {\mnist} and {\fashion}.
\paragraph{Model:}The problem can be formulated as the following nonsmooth nonconvex problem:
\vspace{-0.5em}
\begin{equation*}
    \min_{\|\zeta\|_{\infty}\leq \kappa} \frac{1}{m}\sum_{i=1}^{m}\bigg(\frac{1}{|\mcal D_i|}\sum_{j=1}^{|\mcal D_i|}-\ell(a_i^j+\zeta, b_i^j)\bigg),
\end{equation*}
where the dataset $\mcal D_i=\{a_i^j,b_i^j\}$, $a_i^j$ represents the image features, $b_i^j\in \mbb R^{C}$ is the  one-hot encoding label, $C$ is the number of classes, $|\mcal D_i|$ is the cardinality of dataset $\mcal D_i$, $\kappa$ is the constraint level of the distortion, $\zeta$ is the perturbation vector and $\ell(\cdot, \cdot)$ is the cross entropy function. We set $\kappa = 0.25$ for {\fashion} and $\kappa=0.5$ for {\mnist}. Following the setup in the work of \cite{chen2023faster}, we iteratively perform an additional projection step for constraint satisfaction.

\paragraph{Comparison:}We use two pre-trained models with 99.29\% accuracy on {\mnist} and 92.30\% accuracy on {\fashion}, respectively. We use GFM, GFM$^{+}$, {\dgfmlin} and {\dgfm} to attack the pre-trained LeNet on $59577$ images of {\mnist} and $55384$ images of {\fashion}, which are classified correctly on the train set for training LeNet. Furthermore, we evaluate the perturbation on a dataset comprising 9885 MNIST images and 8975 Fashion-MNIST images. These images have been accurately classified during the testing phase of the LetNet training. Throughout all the experiments, we set $\delta=0.01$, $b=\{16, 32, 64\}$. For {\dgfm} and GFM$^{+}$, we tune $b^{\prime}$ from $\{40, 80, 800, 1600\}$, $T$ from $\{2, 5, 10, 20\}$. Additionally, tune $\mcal T$ from $\{1, 10, 20\}$ for {\dgfm}. For all algorithms, we tune the stepsize $\eta$ from $\{0.05, 0.1, 0.5, 1\}$ and multiply a decay factor $0.6$ if no improvement in $300$ iterations. For all experiments, we set the initial perturbation as $\zerobf$. The results of  the average of five runs are shown in Figure~\ref{fig:res_attack} and Table~\ref{tab:accuracy_attack}.  On the Fashion-MNIST dataset, {\dgfmlin} achieves the lowest accuracy after attacking and small variance. For MNIST, {\dgfm} achieves the lowest accuracy after attacking, but its variance is larger. In general, we can observe that our algorithms perform better than the serial counterparts.

\section{Conclusion}

We proposed the first decentralized gradient-free algorithm which has a provable complexity guarantee for non-smooth non-convex optimization over a multi-agent network. We showed that this method obtains an $\mcal O(d^{3/2}\delta^{-1}\vep^{-4})$ complexity for obtaining an $\mcal O(\delta, \vep)$-Goldstein stationary point. Further, we introduced a decentralized variance-reduced method that enhances the complexity to $\mcal O(d^{3/2}\delta^{-1}\vep^{-3})$, thereby achieving the best complexity rate for non-smooth and non-convex optimization in the serial setting. As a future direction, it would be interesting to investigate whether the best complexity bound can be further improved. Additionally, expanding our algorithms to accommodate more challenging scenarios, such as asynchronous distributed optimization, presents an intriguing avenue for further exploration.

\bibliographystyle{plainnat}
\bibliography{ref}

\appendix
\clearpage

\part{Appendix} 
\insertseparator
\startcontents[appendix] 
\printcontents[appendix]{}{1}{}
\insertseparator

In the appendix, we compile a comprehensive list of all Propositions, Lemmas, and Theorems mentioned in the paper, along with their proofs, for improved readability of the main text. We also restate these results in the appendix.

 To further support the proof of Lemma~\ref{lem:main_thm1} and~\ref{lem:main_theorem}, we provide the following auxiliary lemmas which are not mentioned in the main text: Lemma~\ref{lem:aux_lem_1} in Section~\ref{subsec:aux1} and Lemma~\ref{lem:aux_lem_2} in Section~\ref{subsec:aux2}.

To facilitate the understanding of our proofs, we first present two algorithms in Section~\ref{sec:algorithms}. We then introduce the notation used in the proof and a frequently used proposition in Section~\ref{sec:preliminaries}. The proofs of the convergence of {\dgfmlin} and {\dgfm} can be found in Sections~\ref{sec:GDFMlin} and~\ref{sec:GDFM}, respectively. To illustrate the practicality and effectiveness of our algorithms, we also provide specific examples of parameters for both algorithms that can lead to the optimal convergence rate.

\section{Algorithms}\label{sec:algorithms}

We give the two decentralized algorithms here for completeness.

\begin{algorithm}[htpb]
\caption{\label{alg:dgfmlin_appendix}{\dgfmlin} at each node $i$}
    \begin{algorithmic}[1]
        \Require{$x_{i}^{-1}=x_i^0=\bar{x}^0,\forall i \in [m]$, $y_i^0 = g_i(x_i^{-1};S_i^{-1})=\zerobf_{d},K,\eta$}
        \For{$k=0,\ldots,K-1$}
        \State{Sample $S_i^{k}=\bcbra{\xi_{i}^{k,1},w_i^{k,1}}$ and calculate $g_i(x_i^{k};S_i^k)$}
        \State{$y_i^{k+1}=\sum_{j=1}^{m}a_{i,j}(k)(y_j^k+g_j(x_j^k;S_j^k)-g_j(x_j^{k-1};S_j^{k-1}))$}
        \State{$x_i^{k+1}=\tsum_{j=1}^{m}a_{i,j}(k)(x_j^k - \eta y_j^{k+1})$}
        \EndFor
        
    \State{\textbf{Return:} Choose $x_{\text{out}}$ uniformly at random from $\bcbra{{x}^k_i}_{k=1,\cdots,K, i=1,\cdots,m}$}
        \end{algorithmic}
\end{algorithm}

\begin{algorithm}[htbp]
\caption{\label{alg:dgfm_appendix}{\dgfm} at each node $i$}
    \begin{algorithmic}[1]
        \Require{$x_{i}^{-1}=x_i^0=\bar{x}^0,\forall i \in [m]$, $y_i^0 = v_i^{-1}=\zerobf_{d},K,\eta,b,b^{\prime}$}
        \For{$k=0,\ldots,K-1$}
            \If{$k$ mod $T$ = 0}
            \State Sample $S_i^{k\prime}=\bcbra{(\xi_i^{k\prime,j},w_i^{k\prime,j})}_{j=1}^{b^\prime}$
            \State Calculate $y_i^{k+1} = v_i^k = g_i(x_i^k;S_i^k)$
            \For{$\tau = 1,\cdots,\mcal T$}
            \State $y_i^{k+1}=\tsum_{j=1}^m a_{i,j}^\tau(k)y_j^{k+1}$
        \EndFor
        \Else
        \State Sample $S_i^{k}=\bcbra{(\xi_i^{k,j},w_i^{k,j})}_{j=1}^{b}$
        \State  $v_i^k = v_i^{k-1}+g_i(x_i^k;S_i^k)-g_i(x_i^{k-1};S_i^k)$
        \State $y_i^{k+1}=\tsum_{j=1}^{m}a_{i,j}(k)(y_j^k + v_j^k - v_j^{k-1})$
        \EndIf
        \State $x_i^{k+1}=\tsum_{j=1}^{m}a_{i,j}(k)(x_j^k - \eta y_j^{k+1})$
        \EndFor
        
    \State{\textbf{Return: }Choose $x_{\text{out}}$ uniformly at random from $\bcbra{{x}^k_i}_{k=1,\cdots,K, i=1,\cdots,m}$}
        \end{algorithmic}
\end{algorithm}

\section{Preliminaries}\label{sec:preliminaries}

In this section, we introduce the notations listed in Table~\ref{tab:Notations}, which will be used consistently throughout the document. We also provide a commonly used proposition, Proposition~\ref{prop:common_used}, along with its proof for easy reference. 

\begin{table}
\caption{\label{tab:Notations}Notations for {\dgfmlin} and {\dgfm}}
\renewcommand{\arraystretch}{1.3}
\centering{}%
{\small
\begin{tabular}{|c|c|c|c|}
\hline 
Notations & Specific Formulation & Specific Meaning & Dimension\tabularnewline
\hline 
\hline 
$\xbf^{k}$ & $[(x_{1}^{k})^{\top},\cdots,(x_{m}^{k})^{\top}]^{\top}$ & Stack all local variables $x_{i}^{k}$ & $\mbb R^{md}$\tabularnewline
\hline 
$\tilde{\xbf}^{k}$ & $[(\tilde{x}_{1}^{k})^{\top},\cdots,(\tilde{x}_{m}^{k})^{\top}]^{\top}$ & Stack all local variables $\tilde{x}_{i}^{k}$ & $\mbb R^{md}$\tabularnewline
\hline 
$\bar{\xbf}^{k}$ & $\onebf_{m}\otimes\bar{x}^{k}$ & Copy variable $\bar{x}^{k}$ and concatenate & $\mbb R^{md}$\tabularnewline
\hline 
$\ybf^{k}$ & $[(y_{1}^{k})^{\top},\cdots,(y_{m}^{k})^{\top}]^{\top}$ & Stack all local variables $y_{i}^{k}$ & $\Rbb^{md}$\tabularnewline
\hline 
$\bar{\ybf}^{k}$ & $\onebf_{m}\otimes\bar{y}^{k}$ & Copy variable $\bar{y}^{k}$ and concatenate & $\mbb R^{md}$\tabularnewline
\hline 
$\vbf^{k}$ & $[(v_{1}^{k})^{\top},\cdots,(v_{m}^{k})^{\top}]^{\top}$ & Stack all local variables $v_{i}^{k}$ & $\Rbb^{md}$\tabularnewline
\hline 
$\bar{\vbf}^{k}$ & $\onebf_{m}\otimes\bar{v}^{k}$ & Copy variable $\bar{v}^{k}$ and concatenate & $\mbb R^{md}$\tabularnewline
\hline 
$\tilde{A}(k)$ & $A(k)\otimes\Ibf_{d}$ & Doubly Stochastic Matrix for nodes & $\mbb R^{md\times md}$\tabularnewline
\hline 
$J$ & $\frac{1}{m}\onebf_{m}\onebf_{m}^{\top}\otimes\Ibf_{d}$ & Mean Matrix & $\mbb R^{md\times md}$\tabularnewline
\hline 
$\gbf(\xbf^{k};S^{k})$ & $[(g_{1}(x_{1}^{k};S_{1}^{k}))^{\top},\cdots,(g_{m}(x_{m}^{k};S_{m}^{k}))^{\top}]^{\top}$ & Stack all local variables $g(x_{i}^{k};S_{i}^{k})$ & $\Rbb^{md}\to\Rbb^{md}$\tabularnewline
\hline 
$\nabla f_{\delta}^{i}(x)$ & $\mbb E_{\xi}[\nabla f_{\delta}^{i}(x;\xi)]=\mbb E_{S_{i}}[g(x_{i};S_{i})]$ & Expectation of stochastic gradient & $\mbb R^{d}\to\mbb R^{d}$\tabularnewline
\hline 
\end{tabular}
}
\end{table}

\begin{prop}
\label{prop:common_used}Let $\tilde{A}(k)=A(k)\otimes\Ibf_{d}$ and $J=\frac{1}{m}\onebf_{m}\onebf_{m}^{\top}\otimes\Ibf_{d}$,
then we have
\end{prop}
\begin{enumerate}
\item \label{enu:.1property}$\tilde{A}(k)J=J=J\tilde{A}(k)$.
\item \label{enu:.2property}$\tilde{A}(k)\bar{\zbf}^{k}=\bar{\zbf}^{k}=J\bar{\zbf}^{k},\forall\bar{z}^{k}\in\mbb R^{d}$.
\item \label{enu:3property}There exists a constant $\rho>0$ such that
$\max_{k}\{\norm{\tilde{A}(k)-J}\}\leq\rho<1$.
\item \label{enu:.4spectral_norm}$\norm{\tilde{A}(k)}\leq1$.
\end{enumerate}
\begin{proof}
First, we have
\[
\tilde{A}(k)J=(A(k)\otimes\Ibf_{m})\cdot(\frac{1}{m}\onebf_{m}\onebf_{m}^{\top}\otimes\Ibf_{m})=(\frac{1}{m}A(k)\onebf_{m}\onebf_{m}^{\top}\otimes\Ibf_{m})\aeq\frac{1}{m}\onebf_{m}\onebf_{m}^{T}\otimes\Ibf_{m}=J\beq J\tilde{A}(k),
\]
where $(a)$ holds by doubly stochasticity and $(b)$ follows from
the similar argument in above $\tilde{A}(k)J=J$ based on $\onebf_{m}^{\top}\tilde{A}(k)=\onebf_{m}^{\top}$.

Second, it follows from $\bar{\zbf}^{k}=\onebf\otimes\bar{z}^{k}$
that
\[
\tilde{A}(k)\bar{\zbf}^{k}=(A(k)\otimes\Ibf_{d})(\onebf\otimes\bar{z}^{k})=\onebf\otimes\bar{z}^{k}=\bar{\zbf}^{k}=(\frac{1}{m}\onebf_{m}\onebf_{m}^{\top}\otimes\Ibf_{d})\otimes(\onebf\otimes\bar{z}^{k})=J\bar{\zbf}^{k}.
\]

Third, $\norm{\tilde{A}(k)-J}\leq\rho$ is a well known conclusion
(see e.g.,~\cite{tsitsiklis1984problems}).

Last, $\norm{\tilde{A}(k)}\leq1$ is a well known result (see e.g.,~\cite{sun2022distributed}).
\end{proof}

\section{Convergence Analysis for {\dgfmlin}}\label{sec:GDFMlin}
In this section, we present the proofs for the lemmas and theorems that are used in the convergence analysis of the {\dgfmlin} algorithm.
\subsection{Proof of Lemma~\ref{lem:dgfmlin}}
\begin{lem}
Let $\{x_{i}^{k},y_{i}^{k},g_{i}(x_{i}^{k};S_{i}^{k})\}$ be the sequence
generated by {\dgfmlin} and $\{\xbf^{k},\ybf^{k},\gbf^{k}\}$ be
the corresponding stack variables, then for $k\geq0$, we have
\begin{align}
\xbf^{k+1} & =\tilde{A}(k)(\xbf^{k}-\eta\ybf^{k+1}),\label{eq:x_update_dgfm}\\
\ybf^{k+1} & =\tilde{A}(k)(\ybf^{k}+\gbf^{k}-\gbf^{k-1}),\label{eq:y_update_dgfm}\\
\bar{\xbf}^{k+1} & =\bar{\xbf}^{k}-\eta\bar{\ybf}^{k+1},\label{eq:mean_update_x_dgfm}\\
\bar{\ybf}^{k+1} & =\bar{\ybf}^{k}+\bar{\gbf}^{k}-\bar{\gbf}^{k-1},\label{eq:mean_y_update}\\
\bar{\ybf}^{k+1} & =\bar{\gbf}^{k}.\label{eq:y_g_mean}
\end{align}
\end{lem}
\begin{proof}
It follows from $x_{i}^{k+1}=\sum_{j=1}^{m}a_{i,j}(k)(x_{j}^{k}-\eta y_{j}^{k+1})$
and the definition of $\tilde{A}(k)$, $\xbf^{k}$ and $\ybf^{k}$ that~\eqref{eq:x_update_dgfm}.
Similar result~\eqref{eq:y_update_dgfm} holds by $y_{i}^{k+1}=\sum_{j=1}^{m}a_{i,j}(k)(y_{j}^{k}+g_{j}^{k}-g_{j}^{k-1})$.
It follows from $J\til A(k)=J$ in~\ref{enu:.1property} of Proposition~\ref{prop:common_used},~\eqref{eq:x_update_dgfm} and~\eqref{eq:y_update_dgfm} that~\eqref{eq:mean_update_x_dgfm}
and~\eqref{eq:mean_y_update} hold. Furthermore, combining $y_{i}^{0}=g_{i}^{-1}=\zerobf_{d}$
and~\eqref{eq:mean_y_update} yields~\eqref{eq:y_g_mean}.
\end{proof}
\subsection{Proof of Lemma~\ref{lem:lin_consensus}}
\begin{lem}[Consensus error decay]
\label{lem:Let--generated}Let $\{\bar{\xbf}^{k},\bar{\ybf}^{k}\}$
be the sequence generated by~{\dgfmlin}, then for $k\geq0$, we have
\begin{equation}
\begin{aligned} & \mbb E\bsbra{\norm{\xbf^{k+1}-\bar{\xbf}^{k+1}}^{2}}\\
& \leq\  \rho^{2}(1+\alpha_{1})\mbb E\bsbra{\norm{\xbf^{k}-\bar{\xbf}^{k}}^{2}}+\rho^{2}\eta^{2}(1+\alpha_{1}^{-1})\mbb E\bsbra{\norm{\ybf^{k+1}-\bar{\ybf}^{k+1}}^{2}},
\end{aligned}
\label{eq:consensus_x_dgfmlin}
\end{equation}
and 
\begin{equation}
\begin{aligned} & \mbb E\bsbra{\norm{\ybf^{k+1}-\bar{\ybf}^{k+1}}^{2}\mid\mcal F_{k}}\\
& \leq\ \rho^{2}(1+\alpha_{2})\mbb E\bsbra{\norm{\ybf^{k}-\bar{\ybf}^{k}}^{2}\mid\mcal F_{k}}+\rho^{2}(1+\alpha_{2}^{-1})(6+36{ \eta^{2}}L_{\delta}^{2})m\sigma^{2}\\
 &\ \  +9\rho^{2}(1+\alpha_{2}^{-1})L_{\delta}^{2}\mbb E\bsbra{\norm{\xbf^{k}-\bar{\xbf}^{k}}^{2}\mid\mcal F_{k}}\\
 &\ \  +9\rho^{2}(1+\alpha_{2}^{-1}L_{\delta}^{2}(1+4\eta^{2}L_{\delta}^{2})\mbb E\bsbra{\norm{\xbf^{k-1}-\bar{\xbf}^{k-1}}^{2}\mid\mcal F_{k}}\\
 &\ \  +18L_{\delta}^{2}\eta^{2}{ \rho^{2}}(1+\alpha_{2}^{-1})\mbb E\bsbra{\norm{\nabla f_{\delta}(\bar{x}^{k-1})}^{2}\mid\mcal F_{k}}.
\end{aligned}
\label{eq:consensus_y_dgfmlin}
\end{equation}
\end{lem}
\begin{proof}
We consider sequence $\{\xbf^{k}\},\{\ybf^{k}\}$, separately.

\textbf{Sequence $\{\xbf^{k}\}$: }It follows from~\eqref{eq:x_update_dgfm}
that
\[
\xbf^{k+1}-\bar{\xbf}^{k+1}=(\tilde{A}(k)-J)(\xbf^{k}-\eta\ybf^{k+1}).
\]
Combining the above inequality and~\eqref{enu:.2property},~\eqref{enu:3property}
in Proposition~\ref{prop:common_used}, we have
\begin{align*}
\norm{\xbf^{k+1}-\bar{\xbf}^{k+1}} & \leq\rho\norm{\xbf^{k}-\bar{\xbf}^{k}}+\rho\eta\norm{\ybf^{k+1}-\bar{\ybf}^{k+1}}.
\end{align*}
Combining the above inequality and Young's inequality and taking the expectation,
then we complete the proof of~\eqref{eq:consensus_x_dgfmlin}.

\textbf{Sequence $\{\ybf^{k}\}$:} It follows from~\eqref{eq:y_update_dgfm}
that 
$
\norm{\ybf^{k+1}-\bar{\ybf}^{k+1}}\leq\rho\norm{\ybf^{k}-\bar{\ybf}^{k}}+\rho\norm{\gbf^{k}-\gbf^{k-1}}.
$
Combining the above inequality, Young's inequality, and taking the expectation
on natural field $\mcal F_{k}$, then we have
\begin{align*}
\mbb E\bsbra{\norm{\ybf^{k+1}-\bar{\ybf}^{k+1}}^{2}\mid\mcal F_{k}} & \leq\rho^{2}(1+\alpha_{2})\mbb E\bsbra{\norm{\ybf^{k}-\bar{\ybf}^{k}}^{2}\mid\mcal F_{k}}+\rho^{2}(1+\alpha_{2}^{-1})\mbb E\bsbra{\norm{\gbf^{k}-\gbf^{k-1}}^{2}\mid\mcal F_{k}}.
\end{align*}
Note that 
\[
\begin{aligned}\mbb E\bsbra{\norm{\gbf^{k}-\gbf^{k-1}}^{2}\mid\mcal F_{k}}& = \sum_{i=1}^{m}\mbb E\bsbra{\norm{g_{i}(x_{i}^{k};S_{i}^{k})\pm\nabla f_{\delta}^{i}(x_{i}^{k})\pm\nabla f_{\delta}^{i}(x_{i}^{k-1})-g_{i}(x_{i}^{k};S_{i}^{k-1})}^{2}\mid\mcal F_{k}}\\
& \leq 3\sum_{i=1}^{m}\left(2\sigma^{2}+L_{\delta}^{2}\mbb E\bsbra{\norm{x_{i}^{k}-x_{i}^{k-1}}^{2}\mid\mcal F_{k}}\right)\\
& = 6m\sigma^{2}+3L_{\delta}^{2}\mbb E\bsbra{\norm{\xbf^{k}-\xbf^{k-1}}^{2}\mid\mcal F_{k}}\\
& \leq 6m\sigma^{2}+9L_{\delta}^{2}\brbra{\mbb E\bsbra{\norm{\xbf^{k}-\bar{\xbf}^{k}}^{2}+\norm{\bar{\xbf}^{k}-\bar{\xbf}^{k-1}}^{2}+\norm{\xbf^{k-1}-\bar{\xbf}^{k-1}}^{2}\mid\mcal F_{k}}}\\
& = 6m\sigma^{2}+9L_{\delta}^{2}\brbra{\mbb E\bsbra{\norm{\xbf^{k}-\bar{\xbf}^{k}}^{2}+\eta^{2}\norm{\bar{\ybf}^{k}}^{2}+\norm{\xbf^{k-1}-\bar{\xbf}^{k-1}}^{2}\mid\mcal F_{k}}},
\end{aligned}
\]
and
\[
\begin{aligned}
&\mbb E\bsbra{\norm{\bar{\ybf}^{k}}^{2}\mid\mcal F_{k}}\\
&\leq m\cdot\mbb E\bigg[2\Bnorm{\frac{1}{m}\sum_{i=1}^{m}(g_{i}(x_{i}^{k-1};S_{i}^{k-1})\pm\nabla f_{\delta}^{i}(x_{i}^{k-1})-\nabla f_{\delta}^{i}(\bar{x}^{k-1})}^{2}+2\Bnorm{\frac{1}{m}\sum_{i=1}^{m}\nabla f_{\delta}^{i}(\bar{x}^{k-1})}^{2}\mid \mcal F_{k}\bigg]\\
 & \ \ +4m\mbb E\bigg[\Bnorm{\frac{1}{m}\sum_{i=1}^{m}f_{\delta}^{i}(x_{i}^{k-1})-\nabla f_{\delta}^{i}(\bar{x}^{k-1})}^{2}\mid\mcal F_{k}\bigg]+2\mbb E\bsbra{\norm{\nabla f_{\delta}(\bar{x}^{k-1})}^{2}\mid\mcal F_{k}}\\
& \leq  4\mbb E\bigg[\sum_{i=1}^{m}\norm{g_{i}(x_{i}^{k-1};S_{i}^{k-1})-\nabla f_{\delta}^{i}(x_{i}^{k-1})}^{2}\bigg]+4\mbb E\bigg[\sum_{i=1}^{m}\norm{\nabla f_{\delta}^{i}(x_{i}^{k-1})-\nabla f_{\delta}^{i}(\bar{x}^{k-1})}^{2}\mid\mcal F_{k}\bigg]\\
 &\ \  +2\mbb E\bsbra{\norm{\nabla f_{\delta}(\bar{x}^{k-1})}^{2}\mid\mcal F_{k}}\\
&\leq 4m\sigma^{2}+4L_{\delta}^{2}\mbb E\bsbra{\norm{\xbf^{k-1}-\bar{\xbf}^{k-1}}^{2}\mid\mcal F_{k}}+2\mbb E\bsbra{\norm{\nabla f_{\delta}(\bar{x}^{k-1})}^{2}\mid\mcal F_{k}}.
\end{aligned}
\]

Combining the above three inequalities and taking expectation with
all random variables yield
\[
\begin{aligned} &\mbb E\bsbra{\norm{\ybf^{k+1}-\bar{\ybf}^{k+1}}^{2}\mid\mcal F_{k}}\\ & \leq \rho^{2}(1+\alpha_{2})\mbb E\bsbra{\norm{\ybf^{k}-\bar{\ybf}^{k}}^{2}\mid\mcal F_{k}}+\rho^{2}(1+\alpha_{2}^{-1})\mbb E\bsbra{\norm{\gbf^{k}-\gbf^{k-1}}^{2}\mid\mcal F_{k}}\\
& \leq  \rho^{2}(1+\alpha_{2})\mbb E\bsbra{\norm{\ybf^{k}-\bar{\ybf}^{k}}^{2}\mid\mcal F_{k}}\\
 &\ \  +\rho^{2}(1+\alpha_{2}^{-1})\left(6m\sigma^{2}+9L_{\delta}^{2}(\mbb E\bsbra{\norm{\xbf^{k}-\bar{\xbf}^{k}}^{2}+\eta^{2}\norm{\bar{\ybf}^{k}}^{2}+\norm{\xbf^{k-1}-\bar{\xbf}^{k-1}}^{2}\mid\mcal F_{k}})\right)\\
& \leq \rho^{2}(1+\alpha_{2})\mbb E\bsbra{\norm{\ybf^{k}-\bar{\ybf}^{k}}^{2}\mid\mcal F_{k}}\\
 &\ \  +\rho^{2}(1+\alpha_{2}^{-1})\left(6m\sigma^{2}+9L_{\delta}^{2}(\mbb E\bsbra{\norm{\xbf^{k}-\bar{\xbf}^{k}}^{2}+\norm{\xbf^{k-1}-\bar{\xbf}^{k-1}}^{2}\mid\mcal F_{k}})\right.\\
 & \ \  \left.+9L_{\delta}^{2}\eta^{2}(4m\sigma^{2}+4L_{\delta}^{2}\mbb E\bsbra{\norm{\xbf^{k-1}-\bar{\xbf}^{k-1}}^{2}\mid\mcal F_{k}}+2\mbb E\bsbra{\norm{\nabla f_{\delta}(\bar{x}^{k-1})}^{2}\mid\mcal F_{k}}\right)\\
& = \rho^{2}(1+\alpha_{2})\cdot \mbb E\bsbra{\norm{\ybf^{k}-\bar{\ybf}^{k}}^{2}\mid\mcal F_{k}}+\rho^{2}(1+\alpha_{2}^{-1})(6+36{ \eta^{2}}L_{\delta}^{2})m\sigma^{2}\\
 &\ \  +9\rho^{2}(1+\alpha_{2}^{-1})L_{\delta}^{2}\cdot \mbb E\bsbra{\norm{\xbf^{k}-\bar{\xbf}^{k}}^{2}\mid\mcal F_{k}}\\
 & \ \ +9\rho^{2}(1+\alpha_{2}^{-1}L_{\delta}^{2}(1+4\eta^{2}L_{\delta}^{2})\cdot \mbb E\bsbra{\norm{\xbf^{k-1}-\bar{\xbf}^{k-1}}^{2}\mid\mcal F_{k}}\\
 & \ \  +18L_{\delta}^{2}\eta^{2}{ \rho^{2}}(1+\alpha_{2}^{-1})\cdot \mbb E\bsbra{\norm{\nabla f_{\delta}(\bar{x}^{k-1})}^{2}\mid\mcal F_{k}},
\end{aligned}
\]
which complete the proof of~\eqref{eq:consensus_y_dgfmlin}.
\end{proof}

\subsection{Proof of Lemma~\ref{lem:decent_lin}}
\begin{lem}
[Descent Property of {\dgfmlin}]\label{lem:Let--be}Let $\{\xbf^{k},\ybf^{k},\bar{\xbf}^{k},\bar{\ybf}^{k}\}$
be the sequence generated by {\dgfmlin}, then we have
\begin{equation}\label{eq:3.3}
\mbb E\bsbra{f_{\delta}(\bar{x}^{k+1})}-\mbb E\bsbra{f_{\delta}(\bar{x}^{k})}\leq-\frac{\eta}{2}\mbb E\bsbra{\norm{\nabla f_{\delta}(\bar{x}^{k})}^{2}}+\frac{\eta L_{\delta}^{2}}{2m}\mbb E\bsbra{\norm{\xbf^{k}-\bar{\xbf}^{k}}^{2}}+L_{\delta}\eta^{2}(\sigma^{2}+L_{f}^{2}).
\end{equation}
\end{lem}
\begin{proof}
It follows from $f_{\delta}^{i}(x)$ is differentiable and $L_{f}$-Lipschitz
with $L_{\delta}$-Lipschitz gradient that
\begin{align*}
f_{\delta}^{i}(\bar{x}^{k+1}) & \leq f_{\delta}^{i}(\bar{x}^{k})-\binprod{\nabla f_{\delta}^{i}(\bar{x}^{k})}{\bar{x}^{k+1}-\bar{x}^{k}}+\frac{L_{\delta}}{2}\norm{\bar{x}^{k+1}-\bar{x}^{k}}^{2}\\
 & =f_{\delta}^{i}(\bar{x}^{k})-\eta\binprod{\nabla f_{\delta}^{i}(\bar{x}^{k})}{\bar{g}^{k}}+\frac{L_{\delta}\eta^{2}}{2}\norm{\bar{g}^{k}}^{2}.
\end{align*}
Summing the above inequality over $i=1$ to $m$, dividing by $m$
and taking expectation on $\mcal F_{k}$, then we have
\begin{equation}
\mbb E\bsbra{f_{\delta}(\bar{x}^{k+1})\mid\mcal F_{k}}\leq f_{\delta}(\bar{x}^{k})-\eta\binprod{\nabla f_{\delta}(\bar{x}^{k})}{\mbb E\bsbra{\bar{g}^{k}\mid\mcal F_{k}}}+\frac{L_{\delta}\eta^{2}}{2}\mbb E\bsbra{\norm{\bar{g}^{k}}^{2}\mid\mcal F_{k}}.\label{eq:templin01}
\end{equation}
By $\mbb E\left[\bar{g}^{k}\mid\mcal F_{k}\right]=\frac{1}{m}\sum_{i=1}^{m}\nabla f_{\delta}^{i}(x_{i}^{k})$
and the above inequality, we have
\begin{align}
\binprod{\nabla f_{\delta}(\bar{x}^{k})}{\mbb E\bsbra{\bar{g}^{k}\mid\mcal F_{k}}} & =\Binprod{\nabla f_{\delta}(\bar{x}^{k})}{\frac{1}{m}\sum_{i=1}^{m}\nabla f_{\delta}^{i}(x_{i}^{k})-\nabla f_{\delta}(\bar{x}^{k})}+\norm{\nabla f_{\delta}(\bar{x}^{k})}^{2}\label{eq:templin02}\\
\mbb E\bsbra{\norm{\bar{g}^{k}}^{2}\mid\mcal F_{k}} & \leq2\mbb E\bigg[\Bnorm{\bar{g}^{k}-\frac{1}{m}\sum_{i=1}^{m}\nabla f_{\delta}^{i}(x_{i}^{k})}^{2}+2\Bnorm{\frac{1}{m}\sum_{i=1}^{m}\nabla f_{\delta}^{i}(x_{i}^{k})}^{2}\mid\mcal F_{k}\bigg]\label{eq:templin03}\\
 & \leq2(\sigma^{2}+L_{f}^{2}),\nonumber 
\end{align}

Combining~\eqref{eq:templin01},~\eqref{eq:templin02} and~\eqref{eq:templin03}
yields
\begin{equation*}
\begin{aligned}&\mbb E\bsbra{f_{\delta}(\bar{x}^{k+1})\mid\mcal F_{k}}  \\ &\leq\mbb E\Bsbra{f_{\delta}(\bar{x}^{k})-\eta\Brbra{\Binprod{\nabla f_{\delta}(\bar{x}^{k})}{\frac{1}{m}\sum_{i=1}^{m}\nabla f_{\delta}^{i}(x_{i}^{k})-\nabla f_{\delta}(\bar{x}^{k})}+\norm{\nabla f_{\delta}(\bar{x}^{k})}^{2}}\mid\mcal F_{k}}+L_{\delta}\eta^{2}(\sigma^{2}+L_{f}^{2})\\
  & \aleq\mbb E\bigg[f_{\delta}(\bar{x}^{k})-\eta\norm{\nabla f_{\delta}(\bar{x}^{k})}^{2}+\frac{\eta}{2}\norm{\nabla f_{\delta}(\bar{x}^{k})}^{2}+\frac{\eta}{2}\bnorm{\frac{1}{m}\sum_{i=1}^{m}\nabla f_{\delta}^{i}(x_{i}^{k})-\nabla f_{\delta}(\bar{x}^{k})}^{2}\mid\mcal F_{k}\bigg]\\
  &\ \ +L_{\delta}\eta^{2}(\sigma^{2}+L_{f}^{2})\\
  &\leq\mbb E\bigg[f_{\delta}(\bar{x}^{k})-\frac{\eta}{2}\norm{\nabla f_{\delta}(\bar{x}^{k})}^{2}+\frac{\eta L_{\delta}^{2}}{2m}\norm{\xbf^{k}-\bar{\xbf}^{k}}^{2}\mid\mcal F_{k}\bigg]+L_{\delta}\eta^{2}(\sigma^{2}+L_{f}^{2}),
\end{aligned}
\end{equation*}
where $(a)$ holds by $\binprod{\abf}{\bbf}\leq\frac{1}{2}\norm{\abf}^{2}+\frac{1}{2}\norm{\bbf}^{2}$.
Taking the expectation of both sides and rearranging yields
\[
\mbb E\bsbra{f_{\delta}(\bar{x}^{k+1})}-\mbb E\bsbra{f_{\delta}(\bar{x}^{k})}\leq-\frac{\eta}{2}\mbb E\bsbra{\norm{\nabla f_{\delta}(\bar{x}^{k})}^{2}}+\frac{\eta L_{\delta}^{2}}{2m}\mbb E\bsbra{\norm{\xbf^{k}-\bar{\xbf}^{k}}^{2}}+L_{\delta}\eta^{2}(\sigma^{2}+L_{f}^{2}).
\]
\end{proof}
\subsection{Proof of Lemma~\ref{lem:linthm}}\label{subsec:aux1}
\begin{lem}\label{lem:aux_lem_1}
Let $\{\xbf^{k},\ybf^{k},\bar{\xbf}^{k},\bar{\ybf}^{k}\}$
be the sequence generated by~{\dgfmlin}, then we have

\begin{equation}\label{lem:Let--be-1}
\begin{aligned} & \beta_{x}\mbb E\bsbra{\norm{\xbf^{k+1}-\bar{\xbf}^{k+1}}^{2}-\norm{\xbf^{k}-\bar{\xbf}^{k}}^{2}}+\beta_{y}\mbb E\bsbra{\norm{\ybf^{k+1}-\bar{\ybf}^{k+1}}^{2}-\norm{\ybf^{k}-\bar{\ybf}^{k}}^{2}}\\
& \leq -(\beta_{x}(1-\rho^{2}(1+\alpha_{1}))-9\beta_{y}L_{\delta}^{2}\cdot\rho^{2}(1+\alpha_{2}^{-1}))\mbb E\bsbra{\norm{\xbf^{k}-\bar{\xbf}^{k}}^{2}}\\
 &\ \  +9\beta_{y}\rho^{2}(1+\alpha_{2}^{-1})L_{\delta}^{2}(1+4\eta^{2}L_{\delta}^{2})\mbb E\bsbra{\norm{\xbf^{k-1}-\bar{\xbf}^{k-1}}^{2}}\\
 &\ \  -\beta_{y}(1-\rho^{2}(1+\alpha_{2}))\mbb E\bsbra{\norm{\ybf^{k}-\bar{\ybf}^{k}}^{2}}+\beta_{x}\rho^{2}\eta^{2}(1+\alpha_{1}^{-1})\mbb E\bsbra{\norm{\ybf^{k+1}-\bar{\ybf}^{k+1}}^{2}}\\
 &\ \  +18\beta_{y}\eta^{2}{ \rho^{2}}L_{\delta}^{2}(1+\alpha_{2}^{-1})\mbb E\bsbra{\norm{\nabla f_{\delta}(\bar{x}^{k-1})}^{2}}+\beta_{y}\rho^{2}m(1+\alpha_{2}^{-1})(6+36L_{\delta}^{2}{ \eta^{2}})\sigma^{2},
\end{aligned}
\end{equation}
where $\beta_{x}$ and $\beta_{y}>0$.
\end{lem}

\begin{proof}
It follows from Lemma~\ref{lem:Let--generated} that
\begin{equation}
\begin{aligned}
    &\beta_{x}\mbb E\bsbra{\norm{\xbf^{k+1}-\bar{\xbf}^{k+1}}^{2}-\norm{\xbf^{k}-\bar{\xbf}^{k}}^{2}} \\ &\leq\beta_{x}\brbra{\rho^{2}(1+\alpha_{1})-1}\mbb E\bsbra{\norm{\xbf^{k}-\bar{\xbf}^{k}}^{2}}+\beta_{x}\rho^{2}\eta^{2}(1+\alpha_{1}^{-1})\mbb E\bsbra{\norm{\ybf^{k+1}-\bar{\ybf}^{k+1}}^{2}},\label{eq:beta_x}
\end{aligned}
\end{equation}
and
\begin{equation}\label{eq:eq_y}
    \begin{aligned}
        &\beta_{y}\mbb E\bsbra{\norm{\ybf^{k+1}-\bar{\ybf}^{k+1}}^{2}-\norm{\ybf^{k}-\bar{\ybf}^{k}}^{2}}\\ & \leq \beta_{y}\brbra{\rho^{2}(1+\alpha_{2})-1}\mbb E\bsbra{\norm{\ybf^{k}-\bar{\ybf}^{k}}^{2}}+\beta_{y}\rho^{2}(1+\alpha_{2}^{-1})(6m+36m{ \eta^{2}}L_{\delta}^{2})\sigma^{2} \\
 &\ \  +9\beta_{y}\rho^{2}(1+\alpha_{2}^{-1})L_{\delta}^{2}\mbb E\bsbra{\norm{\xbf^{k}-\bar{\xbf}^{k}}^{2}} \\
 &\ \  +9\beta_{y}\rho^{2}(1+\alpha_{2}^{-1})L_{\delta}^{2}(1+4\eta^{2}L_{\delta}^{2})\mbb E\bsbra{\norm{\xbf^{k-1}-\bar{\xbf}^{k-1}}^{2}}\\
 &\ \  +18\beta_{y}\eta^{2}{ \rho^{2}}L_{\delta}^{2}(1+\alpha_{2}^{-1})\mbb E\bsbra{\norm{\nabla f_{\delta}(\bar{x}^{k-1})}^{2}},
    \end{aligned}
\end{equation}
where $\beta_{x}$ and $\beta_{y}>0$. Summing the two inequalities up yields
the result.
\end{proof}

\begin{lem}
[Convergence Result for {\dgfmlin}]\label{lem:main_thm1}Let $\{\xbf^{k},\ybf^{k},\bar{\xbf}^{k},\bar{\ybf}^{k}\}$
be the sequence generated by {\dgfmlin}, then for any $\beta_{x},\beta_{y}>0$
we have

\begin{equation*}
\begin{aligned} & \beta_{x}\mbb E\bsbra{\norm{\xbf^{K}-\bar{\xbf}^{K}}^{2}-\norm{\xbf^{0}-\bar{\xbf}^{0}}^{2}}+\beta_{y}\mbb E\bsbra{\norm{\ybf^{{ K+1}}-\bar{\ybf}^{K+1}}^{2}-\norm{\ybf^{{ 1}}-\bar{\ybf}^{1}}^{2}}\\
 &\ \  +\mbb E\bsbra{f_{\delta}(\bar{x}^{K})}-\mbb E\bsbra{f_{\delta}(\bar{x}^{0})}\\
& \leq -\brbra{\frac{\eta}{2}-18\beta_{y}{ \rho^{2}}\eta^{2}L_{\delta}^{2}(1+\alpha_{2}^{-1})}\sum_{k=0}^{K-1}\mbb E\bsbra{\norm{\nabla f_{\delta}(\bar{x}^{k})}^{2}}\\
 &\ \  -\brbra{\beta_{x}(1-\rho^{2}(1+\alpha_{1}))-9\beta_{y}\rho^{2}(1+\alpha_{2}^{-1})L_{\delta}^{2}(1+4\eta^{2}L_{\delta}^{2})-\frac{\eta L_{\delta}^{2}}{2m}}\sum_{k=0}^{K-1}\mbb E\bsbra{\norm{\xbf^{k}-\bar{\xbf}^{k}}^{2}}\\
 &\ \  +9\beta_{y}\rho^{2}(1+\alpha_{2}^{-1})L_{\delta}^{2}\sum_{k=1}^{K}\mbb E\bsbra{\norm{\xbf^{k}-\bar{\xbf}^{k}}^{2}}\\
 &\ \  -\brbra{\beta_{y}(1-\rho^{2}(1+\alpha_{2}))-\beta_{x}\eta^{2}\cdot\rho^{2}(1+\alpha_{1}^{-1})}\sum_{k=1}^{K}\mbb E\bsbra{\norm{\ybf^{k}-\bar{\ybf}^{k}}^{2}}\\
 &\ \  +\brbra{\beta_{y}\rho^{2}(1+\alpha_{2}^{-1})(6m+36m{ \eta^{2}}L_{\delta}^{2})\sigma^{2}+L_{\delta}\eta^{2}(\sigma^{2}+L_{f}^{2})}\cdot K.
\end{aligned}
\end{equation*}
\end{lem}
\begin{proof}
Summing the combination of~\eqref{lem:Let--be-1} and~\eqref{eq:3.3} from $k=0$ to $K-1$ completes our proof.
\end{proof}

\subsection{Proof of Theorem~\ref{thm:lin_order}}
\begin{thm}
{\dgfmlin} can output a $(\delta,\vep)$-Goldstein stationary point of $f(\cdot)$ in expectation with the total stochastic zeroth-order complexity and total communication complexity  at most $\mcal O(\Delta_{\delta} \delta^{-1}\vep^{-4})$ by setting 
    $\alpha_1 = \alpha_2 = (1-\rho^2)(2\rho^2)^{-1}$, $\delta = \mcal O(\vep)$, $\beta_x = \mcal O(\delta^{-1}),\beta_y = \mcal O(\vep^4\delta)$ and $\eta=\mcal O(\vep^2 \delta)$. Moreover, a specific example of parameters is given as follows:
\begin{align*}
\beta_{y} & =\frac{(1-\rho^{2})\vep^{2}}{384\sigma^{2}\rho^{2}(1+\rho^{2})}\cdot\frac{\eta}{m},\ \ \beta_{x}=\frac{1152\sigma^{2}\rho^{2}(1+\rho^{2})}{(1-\rho^{2})^{2}}\cdot L_{\delta}^{2}\vep^{-2}\beta_{y},\\
\eta& = \min\left\{ \frac{(1-\rho^{2})^{2}}{48\sigma(1+\rho^{2})\rho^{2}}\cdot\vep L_{\delta}^{-1},\frac{1}{32}L_{\delta}^{-1}(\sigma^{2}+L_{f})^{-1}\cdot\vep^{2},\frac{8\sqrt{6m\sigma^{2}}}{\vep L_{\delta}}\right\},
\end{align*}
then {\dgfmlin} can output a $(\delta,\vep)$-Goldstein stationary point of $f(\cdot)$ in expectation with the total stochastic zeroth-order complexity and total communication complexity at most 
\begin{align*}
\mcal O\left(\max\left\{\Delta_{\delta}\vep^{-4}\delta^{-1}d^{\frac{3}{2}},\frac{(1+L_{f}^{2}/\sigma^{2})(1-\rho^{2})}{m(1+\rho^{2})}\right\}\right).    
\end{align*}
\end{thm}
\begin{proof}
    It follows from Lemma~\ref{lem:main_thm1} that
\[
\begin{aligned} & \beta_{x}\mbb E\bsbra{\norm{\xbf^{K}-\bar{\xbf}^{K}}^{2}-\norm{\xbf^{0}-\bar{\xbf}^{0}}^{2}}+\beta_{y}\mbb E\bsbra{\norm{\ybf^{{ K+1}}-\bar{\ybf}^{K+1}}^{2}-\norm{\ybf^{{ 1}}-\bar{\ybf}^{1}}^{2}}\\
 &\ \  +\mbb E\bsbra{f_{\delta}(\bar{x}^{K})}-\mbb E\bsbra{f_{\delta}(\bar{x}^{0})}\\
& \leq -\brbra{\frac{\eta}{2}-18\beta_{y}{ \rho^{2}}\eta^{2}L_{\delta}^{2}(1+\alpha_{2}^{-1})}\sum_{k=0}^{K-1}\mbb E\bsbra{\norm{\nabla f_{\delta}(\bar{x}^{k})}^{2}}\\
 &\ \  -\brbra{\beta_{x}(1-\rho^{2}(1+\alpha_{1}))-18\beta_{y}\rho^{2}(1+\alpha_{2}^{-1})L_{\delta}^{2}({1}+2\eta^{2}L_{\delta}^{2})-\frac{\eta L_{\delta}^{2}}{2m}}\sum_{k=1}^{K-1}\mbb E\bsbra{\norm{\xbf^{k}-\bar{\xbf}^{k}}^{2}}\\
 &\ \  -\brbra{\beta_{x}(1-\rho^{2}(1+\alpha_{1}))-9\beta_{y}\rho^{2}(1+\alpha_{2}^{-1})L_{\delta}^{2}(1+4\eta^{2}L_{\delta}^{2})-\frac{\eta L_{\delta}^{2}}{2m}}\mbb E\bsbra{\norm{\xbf^{0}-\bar{\xbf}^{0}}^{2}}\\
 &\ \  +9\beta_{y}\rho^{2}(1+\alpha_{2}^{-1})L_{\delta}^{2}\mbb E\bsbra{\norm{\xbf^{K}-\bar{\xbf}^{K}}^{2}}\\
 &\ \  -\brbra{\beta_{y}(1-\rho^{2}(1+\alpha_{2}))-\beta_{x}\eta^{2}\cdot\rho^{2}(1+\alpha_{1}^{-1})}\sum_{k=1}^{K}\mbb E\bsbra{\norm{\ybf^{k}-\bar{\ybf}^{k}}^{2}}\\
 &\ \  +\brbra{\beta_{y}\rho^{2}(1+\alpha_{2}^{-1})(6m+36m{ \eta^{2}}L_{\delta}^{2})\sigma^{2}+L_{\delta}\eta^{2}(\sigma^{2}+L_{f}^{2})}\cdot K.
\end{aligned}
\]
which implies
\begin{equation}\label{eq:key_ineq}
\begin{aligned} & \underbrace{\brbra{\frac{\eta}{2}-18\beta_{y}{ \rho^{2}}\eta^{2}L_{\delta}^{2}(1+\alpha_{2}^{-1})}}_{(\rone)}\sum_{k=0}^{K-1}\mbb E\left[\norm{\nabla f_{\delta}(\bar{x}^{k})}^{2}\right]\\
& \aleq  -\left(\mbb E\left[f_{\delta}(\bar{x}^{K})\right]-\mbb E\left[f_{\delta}(\bar{x}^{0})\right]\right)+\beta_{y}\rho^{2}(\sigma^{2}+L_{f}^{2})\\
 &\ \  -\brbra{\underbrace{\beta_{x}(1-\rho^{2}(1+\alpha_{1}))-18\beta_{y}\rho^{2}(1+\alpha_{2}^{-1})L_{\delta}^{2}(1+2\eta^{2}L_{\delta}^{2})-\frac{\eta L_{\delta}^{2}}{2m}}_{(\rtwo)}}\sum_{k=1}^{K-1}\mbb E\left[\norm{\xbf^{k}-\bar{\xbf}^{k}}^{2}\right]\\
 &\ \  -\brbra{\underbrace{\beta_{x}-9\beta_{y}\rho^{2}(1+\alpha_{2}^{-1})L_{\delta}^{2}}_{(\rthree)}}\mbb E\left[\norm{\xbf^{K}-\bar{\xbf}^{K}}^{2}\right]\\
 &\ \  -\brbra{\underbrace{\beta_{y}(1-\rho^{2}(1+\alpha_{2}))-\beta_{x}\eta^{2}\cdot\rho^{2}(1+\alpha_{1}^{-1})}_{(\rfour)}}\sum_{k=1}^{K}\mbb E\left[\norm{\ybf^{k}-\bar{\ybf}^{k}}^{2}\right]\\
 &\ \  +\brbra{\underbrace{\beta_{y}\rho^{2}(1+\alpha_{2}^{-1})(6m+36m{ \eta^{2}}L_{\delta}^{2}\sigma^{2})}_{(\rfive)}+\underbrace{L_{\delta}\eta^{2}(\sigma^{2}+L_{f}^{2})}_{(\rsix)}}\cdot K.
\end{aligned}
\end{equation}
where (a) is by $\mbb E\bsbra{\norm{\ybf^{{ 1}}-\bar{\ybf}^{1}}^{2}}=\mbb E\bsbra{\norm{(\tilde{A}(0)-J)(\ybf^{{ 0}}+\gbf^{0}-\gbf^{-1})}^{2}}
=\mbb E\bsbra{\norm{(\tilde{A}(0)-J)\gbf^{0}}^{2}}
\leq\rho^{2}(\sigma^{2}+L_{f}^{2})$.

By setting $\alpha_1 = \alpha_2 = (1-\rho^2)(2\rho^2)^{-1}$, we have 
\begin{align*}
 1-\rho^{2}(1+\alpha_{1})=1-\rho^{2}(1+\alpha_{2})=\frac{1-\rho^{2}}{2},
\end{align*}
and
\begin{align*}
1+\alpha_{1}^{-1}=1+\alpha_{2}^{-1}=\frac{1+\rho^{2}}{1-\rho^{2}}.
\end{align*}
Combining the parameters setting of $\beta_x,\beta_y$ and $\eta$ yields $(\rone)$-$(\rsix)>0$ and
\begin{equation}\label{eq:order1}
    (\rone) = \Ocal(\vep^2\delta),\ (\rtwo) = \Ocal(\delta^{-1}),\ (\rthree) = \Ocal(\delta^{-1}),\ (\rfour) = \Ocal (\vep^4 \delta), \  (\rfive) = \Ocal(\vep^4\delta),\ (\rsix) = \Ocal(\vep^4\delta).
\end{equation}
Rearranging~\eqref{eq:key_ineq} and combining $(\rthree)>(\rtwo)$ yields
\begin{equation}\label{eq:res1}
\begin{aligned}
    &(\rone)\sum_{k=0}^{K-1}\mbb E\bsbra{\norm{\nabla f_{\delta}(\bar{x}^{k})}^{2}}+(\rtwo)\sum_{k=1}^{K}\mbb E\bsbra{\norm{\xbf^{k}-\bar{\xbf}^{k}}^{2}}+(\rfour)\sum_{k=1}^{K}\mbb E\bsbra{\norm{\ybf^{k}-\bar{\ybf}^{k}}^{2}} \\
&\leq \Delta_{\delta}+\beta_{y}\rho^{2}(\sigma^{2}+L_{f}^{2})+(\rfive)\cdot K+(\rsix)\cdot K.
\end{aligned}
\end{equation}
Divide $(\rone)\cdot K$ on both sides of~\eqref{eq:res1}, we have
\begin{equation}\label{eq:res2}
\begin{aligned}
    &\frac{1}{K}\sum_{k=0}^{K-1}\mbb E\bsbra{\norm{\nabla f_{\delta}(\bar{x}^{k})}^{2}}+\frac{(\rtwo)}{K\cdot(\rone)}\sum_{k=1}^{K}\mbb E\bsbra{\norm{\xbf^{k}-\bar{\xbf}^{k}}^{2}}+\frac{(\rfour)}{K\cdot(\rone)}\sum_{k=1}^{K}\mbb E\bsbra{\norm{\ybf^{k}-\bar{\ybf}^{k}}^{2}} \\ &\leq\frac{\Delta_{\delta}+\beta_{y}\rho^{2}(\sigma^{2}+L_{f}^{2})}{K\cdot(\rone)}+\frac{(\rfive)+(\rsix)}{(\rone)}.
\end{aligned}
\end{equation}
By using~\eqref{eq:order1}, we deduce that the right-hand side of~\eqref{eq:res2} is at order $\mcal O(\vep^{2})$ when $K=~\mathcal{O}(\Delta_{\delta}\delta^{-1}\vep^{-4})$. Since $(\rtwo)/(\rone)=\Ocal(\vep^{-2}\delta^{-2})$, $\nabla f_{\delta}(x)\in \partial_{\delta}f(x)$ and $\norm{\nabla f_{\delta}(x_i^k)}^2\leq 2\norm{\nabla f_{\delta}(\bar{x}^k)}^2 + 2L_{\delta}^{2}\norm{x_i^k-\bar{x}^k}^2$, then we have $\mcal O(\Delta_{\delta}\delta^{-1}\vep^{-4})$.

Now, suppose $\delta\leq\vep$ and $\vep$ is small enough such that
\begin{equation}\label{eq:temp_proof_media}
    {1152\sigma^{2}(1+\rho^{2})^{2}\rho^{4}} + {(1-\rho^{2})^{4}\vep^{2}}\leq  {1152\sigma^{2}(1+\rho^{2})^{2}\rho^{4}} \cdot 16\sigma^{2}\vep^{-2}
\end{equation}
and $${(1-\rho^{2})^{4}\vep^{2}}\leq{384\sigma^{2}(1+\rho^{2})^{2}\rho^{4}}$$
hold. Now, suppose
\begin{equation}
K\geq\max\left\{ \frac{2(\sigma^{2}+L_{f}^{2})(1-\rho^{2})}{3m\sigma^{2}(1+\rho^{2})},32\Delta_{\delta}\vep^{-2}\eta^{-1}\right\} ,\label{eq:K_large}
\end{equation}
we give specific parameter settings to obtain $(\delta,\vep)$-Goldstein
stationary point in expectation, which can be summarized as follows:
\begin{align}
\beta_{y} & =\frac{(1-\rho^{2})\vep^{2}}{384\sigma^{2}\rho^{2}(1+\rho^{2})}\cdot\frac{\eta}{m},\ \ \beta_{x}=\frac{1152\sigma^{2}\rho^{2}(1+\rho^{2})}{(1-\rho^{2})^{2}}\cdot L_{\delta}^{2}\vep^{-2}\beta_{y},\label{eq:beta_xy}\\
\eta& = \min\left\{ \frac{(1-\rho^{2})^{2}}{48\sigma(1+\rho^{2})\rho^{2}}\cdot\vep L_{\delta}^{-1},\frac{1}{32}L_{\delta}^{-1}(\sigma^{2}+L_{f})^{-1}\cdot\vep^{2},\frac{8\sqrt{6m\sigma^{2}}}{\vep L_{\delta}}\right\} .\label{eq:eta_defn}
\end{align}
By $\eta\leq{8\sqrt{6m\sigma^{2}}}/({\vep L_{\delta}})$
and the definition of $\beta_{y}$ in~\eqref{eq:beta_xy}, we
have $72\eta\rho^{2}L_{\delta}^{2}(1+\rho^{2}) \beta_{y}\leq {1-\rho^{2}}$,
which implies $(\rone)\geq {\eta}/{4}$. By the definition of
$\beta_{x}$ in~\eqref{eq:beta_xy}, then we have
\[
\begin{aligned}(\rfour) & =\frac{1-\rho^{2}}{2}\beta_{y}-\frac{\rho^{2}(1+\rho^{2})}{1-\rho^{2}}\eta^{2}\beta_{x}\\
 & =\frac{1-\rho^{2}}{2}\beta_{y}-\frac{\rho^{2}(1+\rho^{2})}{1-\rho^{2}}\eta^{2}\cdot \frac{1152\sigma^{2}\rho^{2}(1+\rho^{2})}{\left(1-\rho^{2}\right)^{2}}L_{\delta}^{2}\vep^{-2}\beta_{y}\\
 & =\Brbra{\frac{1-\rho^{2}}{2}-\frac{1152\sigma^{2}\rho^{4}(1+\rho^{2})^{2}}{\left(1-\rho^{2}\right)^{3}}\eta^{2}L_{\delta}^{2}\vep^{-2}}\beta_{y}.\\
 & \ageq\Brbra{\frac{1-\rho^{2}}{2}-\frac{1-\rho^{2}}{2}}\beta_{y}\\
 & =0,
\end{aligned}
\]
where $(a)$ is by ${2304\sigma^{2}(1+\rho^{2})^{2}\rho^{4}} \cdot \eta^{2}L_{\delta}^{2}\vep^{-2}\leq{(1-\rho^{2})^{4}}$,
which is derived from 
${48\sigma(1+\rho^{2})\rho^{2}}L_{\delta}\eta\leq~{(1-\rho^{2})^{2}\vep}$.
By the definition of $\beta_{y}$ and ${48\sigma(1+\rho^{2})\rho^{2}}L_{\delta}\eta\leq~{(1-\rho^{2})^{2}\vep}$,
we have
\[
\begin{aligned}(\rtwo) & =\frac{1-\rho^{2}}{2}\beta_{x}-\frac{18\rho^{2}(1+\rho^{2})}{1-\rho^{2}}L_{\delta}^{2}(1+2\eta^{2}L_{\delta}^{2})\beta_{y}-\frac{\eta L_{\delta}^{2}}{2m}\\
 & =\Brbra{\frac{576\sigma^{2}\rho^{2}(1+\rho^{2})}{1-\rho^{2}}\cdot L_{\delta}^{2}\vep^{-2}-\frac{18\rho^{2}(1+\rho^{2})}{1-\rho^{2}}L_{\delta}^{2}(1+2\eta^{2}L_{\delta}^{2})}\beta_{y}-\frac{\eta L_{\delta}^{2}}{2m}\\
 & \ageq\frac{288\sigma^{2}\rho^{2}(1+\rho^{2})}{1-\rho^{2}}\cdot L_{\delta}^{2}\vep^{-2}\beta_{y}-\frac{\eta L_{\delta}^{2}}{2m}\\
 & =\frac{288\sigma^{2}\rho^{2}(1+\rho^{2})}{1-\rho^{2}}\cdot\frac{(1-\rho^{2})\cdot\vep^{2}}{384\sigma^{2}\rho^{2}(1+\rho^{2})}\cdot\frac{\eta}{m}\cdot L_{\delta}^{2}\vep^{-2}-\frac{\eta L_{\delta}^{2}}{2m}\\
 & \geq\frac{\eta L_{\delta}^{2}}{4m},
\end{aligned}
\]
where $(a)$ is from $(1+2\eta^{2}L_{\delta}^{2})\leq1+{(1-\rho^{2})^{4}\vep^{2}}{(1152\sigma^{2}(1+\rho^{2})^{2}\rho^{4})^{-1}}\leq16\sigma^{2}\vep^{-2}$
and $$\frac{18\rho^{2}(1+\rho^{2})}{1-\rho^{2}}L_{\delta}^{2}(1+2\eta^{2}L_{\delta}^{2})\leq\frac{288\sigma^{2}\rho^{2}(1+\rho^{2})}{1-\rho^{2}}L_{\delta}^{2}\vep^{-2},$$
which is derived by ${48L_{\delta}\sigma(1+\rho^{2})\rho^{2}} \eta\leq{(1-\rho^{2})^{2}}\vep $
and~\eqref{eq:temp_proof_media}.
Similarly, we have $(\rthree)\geq(\rtwo)\geq~{\eta L_{\delta}^{2}}/{4m}$.
It follows from~\eqref{eq:key_ineq} that
\begin{equation}
\frac{1}{K}\sum_{k=0}^{K-1}\Expe\bsbra{\norm{\nabla f_{\delta}(\bar{x}_{i}^{k})}^{2}}+\frac{L_{\delta}^{2}}{Km}\sum_{k=0}^{K-1}\Expe\bsbra{\norm{\xbf_{i}^{k}-\bar{\xbf}_{i}^{k}}^{2}}\leq\frac{4\brbra{\Delta_{\delta}+\beta_{y}\rho^{2}(\sigma^{2}+L_{f}^{2})}}{K\eta}+\frac{4\brbra{(\rfive)+(\rsix)}}{\eta}.\label{eq:con_3}
\end{equation}
Next, we claim that the RHS of~\eqref{eq:con_3} can be bounded by
${\vep^{2}}/{2}$ if~\eqref{eq:K_large},~\eqref{eq:beta_xy}
and~\eqref{eq:eta_defn} hold. It follows from the definition of
$K$ and $\beta_{y}$ that
\begin{equation*}
    \frac{4\Delta_{\delta}}{K\eta}\aleq\frac{\vep^{2}}{8},
\end{equation*}
and
\begin{equation*}
\frac{4\beta_{y}\rho^{2}(\sigma^{2}+L_{f}^{2})}{K\eta}=\frac{\sigma^{2}+L_{f}^{2}}{K}\cdot\frac{(1-\rho^{2})\vep^{2}}{12m\sigma^{2}(1+\rho^{2})}\leq\frac{3m\sigma^{2}(\sigma^{2}+L_{f}^{2})(1+\rho^{2})}{\brbra{\sigma^{2}+L_{f}^{2}}(1-\rho^{2})}\cdot\frac{(1-\rho^{2})\vep^{2}}{24m\sigma^{2}(1+\rho^{2})}\leq\frac{\vep^{2}}{8},
\end{equation*}
where $(a)$ holds by $K\geq32\Delta_{\delta}\vep^{-2}\eta^{-1}$,
Similarly, by the definition of $\beta_{x}$, $\beta_{y}$ and $\eta,$we
have
\begin{equation*}
\begin{aligned}\frac{4\brbra{(\rfive)+(\rsix)}}{\eta} & =\frac{4}{\eta}\big(\frac{6m\sigma^{2}\rho^{2}(1+\rho^{2})}{1-\rho^{2}}(1+6\eta^{2}L_{\delta}^{2})\cdot\beta_{y}+L_{\delta}\eta^{2}(\sigma^{2}+L_{f}^{2})\big)\\
 & =\frac{4}{\eta}\big(\frac{6m\sigma^{2}\rho^{2}(1+\rho^{2})}{1-\rho^{2}}(1+6\eta^{2}L_{\delta}^{2})\cdot\frac{(1-\rho^{2})\vep^{2}}{384\sigma^{2}\rho^{2}(1+\rho^{2})}\cdot\frac{\eta}{m}+L_{\delta}\eta^{2}(\sigma^{2}+L_{f}^{2})\big)\\
 & =4\big((1+6\eta^{2}L_{\delta}^{2})\cdot\frac{\vep^{2}}{64}+L_{\delta}\eta(\sigma^{2}+L_{f}^{2})\big)\\
 & \leq4\Brbra{\big(1+\frac{(1-\rho^{2})^{4}\vep^{2}}{384\sigma^{2}(1+\rho^{2})^{2}\rho^{4}}\big)\cdot\frac{\vep^{2}}{64}+L_{\delta}\eta(\sigma^{2}+L_{f}^{2})}\aleq4\brbra{\frac{\vep^{2}}{32}+\frac{\vep^{2}}{32}}\leq\frac{\vep^{2}}{4},
\end{aligned}
\end{equation*}
where (a) is from the claim that ${(1-\rho^{2})^{4}\vep^{2}}\leq {384\sigma^{2}(1+\rho^{2})^{2}\rho^{4}}$
and $32\eta\leq {\vep^{2}}L_{\delta}^{-1}(\sigma^{2}+L_{f})^{-1}$.
Note that
\[
\begin{aligned}\frac{1}{mK}\sum_{k=0}^{K-1}\sum_{i=1}^{m}\Expe\bsbra{\norm{\nabla f_{\delta}(x_{i}^{k})}^{2}} & =\frac{1}{mK}\sum_{k=0}^{K-1}\sum_{i=1}^{m}\Expe\bsbra{\norm{\nabla f_{\delta}(x_{i}^{k}-\bar{x}^{k}+\bar{x}^{k})}^{2}}\\
 & \leq\frac{2}{mK}\sum_{k=0}^{K-1}\sum_{i=1}^{m}\Expe\bsbra{\norm{\nabla f_{\delta}(x_{i}^{k}-\bar{x}^{k})}^{2}}+\frac{2}{mK}\sum_{k=0}^{K-1}\sum_{i=1}^{m}\Expe\bsbra{\norm{\nabla f_{\delta}(\bar{x}^{k})}^{2}}\\
 & \leq\frac{2L_{\delta}^{2}}{mK}\sum_{k=0}^{K-1}\sum_{i=1}^{m}\Expe\bsbra{\norm{x_{i}^{k}-\bar{x}^{k}}^{2}}+\frac{2}{K}\sum_{k=0}^{K-1}\Expe\bsbra{\norm{\nabla f_{\delta}(\bar{x}^{k})}^{2}}\\
 & =\frac{2L_{\delta}^{2}}{mK}\sum_{k=0}^{K-1}\Expe\sum_{i=1}^{m}\bsbra{\norm{x_{i}^{k}-\bar{x}^{k}}^{2}}+\frac{2}{K}\sum_{k=0}^{K-1}\Expe\left[\norm{\nabla f_{\delta}(\bar{x}^{k})}^{2}\right]\\
 & =\frac{2L_{\delta}^{2}}{mK}\sum_{k=0}^{K-1}\Expe\bsbra{\norm{\xbf^{k}-\bar{\xbf}^{k}}^{2}}+\frac{2}{K}\sum_{k=0}^{K-1}\Expe\bsbra{\norm{\nabla f_{\delta}(\bar{x}^{k})}^{2}}\leq\vep^{2},
\end{aligned}
\]
and $\sigma^{2}=16\sqrt{\pi}dL_{f}^{2},$ we have total complexity
$\mcal O\brbra{\max\{\Delta_{\delta}\vep^{-4}\delta^{-1}d^{\frac{3}{2}},{m^{-1}(1+L_{f}^{2}/\sigma^{2})(1-\rho^{2})}/{(1+\rho^{2})}\}}$.
\end{proof}

\section{Convergence Analysis for {\dgfm}}\label{sec:GDFM}
In this section, we present the proofs for the lemmas and theorems that are used in the convergence analysis of the {\dgfm} algorithm.
\subsection{Proof of Lemma~\ref{lem:iterative_dgfmplus}}
\begin{lem}
Let $\{x_{i}^{k},y_{i}^{k},g_{i}^{k},v_{i}^{k}\}$ be the sequence
generated by Algorithm~\ref{alg:dgfm_appendix} and $\{\xbf^{k},\ybf^{k},\gbf^{k},\vbf^{k}\}$
be the corresponding stack variables, then for $k\geq0$, we have
\begin{align}
\xbf^{k+1} & =\tilde{A}(k)(\xbf^{k}-\eta\ybf^{k+1}),\label{eq:x_update-1}\\
\bar{\xbf}^{k+1} & =\bar{\xbf}^{k}-\eta\bar{\ybf}^{k+1},\label{eq:x_bar_update-1}\\
\bar{\ybf}^{k+1} & =\bar{\vbf}^{k}.\label{eq:y_bar_v_bar_dgfm}
\end{align}
Furthermore, for $rT<k<(r+1)T$, $r=0,\cdots,R-1$, we have 
\begin{align}
\ybf^{k+1} & =\tilde{A}(k)(\ybf^{k}+\vbf^{k}-\vbf^{k-1}),\label{eq:y_update-1}\\
\bar{\ybf}^{k+1} & =\bar{\ybf}^{k}+\bar{\vbf}^{k}-\bar{\vbf}^{k-1}.\label{eq:y_bar_dgfm}
\end{align}
\end{lem}
\begin{proof}
It follows from $x_{i}^{k+1}=\sum_{j=1}^{m}a_{i,j}(k)(x_{j}^{k}-\eta y_{j}^{k+1})$
and the definitions of $\xbf^{k},\ybf^{k}$ and $\tilde{A}(k)$ that~\eqref{eq:x_update-1}
holds. Similar result~\eqref{eq:y_update-1} holds by $y_{i}^{k+1}=\sum_{j=1}^{m}a_{i,j}(k)(y_{j}^{k}+v_{j}^{k}-v_{j}^{k-1})$
for $rT<k<(r+1)T,\ r=0,\cdots,R-1$. It follows from $J\tilde{A}(k)=J$
and~\eqref{eq:x_update-1},~\eqref{eq:y_update-1} that~\eqref{eq:x_bar_update-1},~\eqref{eq:y_bar_dgfm}
holds. By $y_{i}^{0}=v_{i}^{-1}$, we have $\bar{y}^{0}=\bar{v}^{-1}$.
Similarly, for $k=rT,r=1,\cdots,R-1$, $\bar{y}^{rT+1}=\bar{v}^{rT}$.
Suppose $\bar{y}^{k}=\bar{v}^{k-1}$ holds for any $k\geq0$, then
it follows from~\eqref{eq:y_bar_v_bar_dgfm} holds for any $k\geq0$
that the desired result~\eqref{eq:y_bar_v_bar_dgfm} holds.
\end{proof}
\subsection{Proof of Lemma~\ref{lem:consensus-1}}
\begin{lem}
\label{lem:consensus-dgfm_plus}For sequence $\{\xbf^{k},\ybf^{k}\}$ generated
by Algorithm~\ref{alg:dgfm_appendix}, we have
\begin{equation}
\begin{aligned}\mbb E\bsbra{\norm{\ybf^{k+1}-\bar{\ybf}^{k+1}}^{2}}& \leq  \rho^{2}(1+\alpha_{1})\mbb E\bsbra{\norm{\ybf^{k}-\bar{\ybf}^{k}}^{2}}\\ &\ \  + 3\Brbra{\frac{\rho^{2}L_{\delta}^{2}}{\alpha_{1}}+\frac{\rho^{2}d^{2}L_{f}^{2}}{b\delta^{2}}}\mbb E\bsbra{\norm{\xbf^{k}-\bar{\xbf}^{k}}^{2}+\norm{\bar{\xbf}^{k-1}-\xbf^{k-1}}^{2}}\\
 &\ \  +3\eta^{2}\Brbra{\frac{\rho^{2}L_{\delta}^{2}}{\alpha_{1}}+\frac{\rho^{2}d^{2}L_{f}^{2}}{b\delta^{2}}}\mbb E\bsbra{\norm{\bar{\vbf}^{k-1}}^{2}},\ \ rT+1\leq k<(r+1)T,\\ &\ \  \forall  r=0,1,\cdots,R-1.
\end{aligned}
\label{eq:y_consensus4-1}
\end{equation}
\begin{equation}
\mbb E\bsbra{\norm{\xbf^{k+1}-\bar{\xbf}^{k+1}}^{2}}\leq(1+\alpha_{2})\rho^{2}\mbb E\bsbra{\norm{\xbf^{k}-\bar{\xbf}^{k}}^{2}}+(1+\alpha_{2}^{-1})\rho^{2}\eta^{2}\mbb E\bsbra{\norm{\ybf^{k+1}-\bar{\ybf}^{k+1}}^{2}},\ \ \forall k\geq0.\label{eq:x_consensus}
\end{equation}
Furthermore, we have
\begin{equation}
\mbb E\bsbra{\norm{\ybf^{rT+1}-\bar{\ybf}^{rT+1}}^{2}}\leq{2\rho^{\Tcal}m(\sigma^{2}+L_{f}^{2})},\ \ \forall r=0,\cdots,R-1.\label{eq:cycle_consensus}
\end{equation}
\end{lem}
\begin{proof}
We consider the two sequence, separately.

\textbf{Sequence $\{\xbf^{k}\}$: } It follows from~\eqref{eq:x_update-1}
and $J\tilde{A}(k)=J=J\tilde{A}(k)$, 
and $J\bar{\ybf}^{k+1}=\tilde{A}(k)\bar{\ybf}^{k+1}$ in Proposition~\ref{prop:common_used}
that
\begin{align*}
\xbf^{k+1}-\bar{\xbf}^{k+1} & =(I-J)\tilde{A}(k)(\xbf^{k}-\eta\ybf^{k+1})\\
 & =(\tilde{A}(k)-J)(\xbf^{k}-\bar{\xbf}^{k}-\eta(\ybf^{k+1}-\bar{\ybf}^{k+1})).
\end{align*}
Then we have 
\[
\norm{\xbf^{k+1}-\bar{\xbf}^{k+1}}\leq\rho\norm{\xbf^{k}-\bar{\xbf}^{k}}+\rho\eta\norm{\ybf^{k+1}-\bar{\ybf}^{k+1}},\forall k\geq0,
\]
which implies the result~\eqref{eq:x_consensus} by taking expectation
and combining Young's inequality.

\textbf{Sequence $\{\ybf^{k}\}$: }For $rT+1\leq k<(r+1)T$, we have $\ybf^{k+1}=\tilde{A}(k)(\ybf^{k}+\vbf^{k}-\vbf^{k-1})$,
which implies 
\begin{equation}
\begin{aligned}&\norm{\ybf^{k+1}-\bar{\ybf}^{k+1}}^{2}\\ & = \norm{(I-J)\tilde{A}(k)(\ybf^{k}+\vbf^{k}-\vbf^{k-1})}^{2}=\norm{(\tilde{A}(k)-J)(\ybf^{k}+\vbf^{k}-\vbf^{k-1})}^{2}\\
& = \norm{(\tilde{A}(k)-J)(\ybf^{k}-\bar{\ybf}^{k})}^{2}+2\binprod{(\tilde{A}(k)-J)\ybf^{k}}{(\tilde{A}(k)-J)(\vbf^{k}-\vbf^{k-1})}+\rho^{2}\norm{\vbf^{k}-\vbf^{k-1}}^{2}\\
& \leq \rho^{2}\norm{\ybf^{k}-\bar{\ybf}^{k}}^{2}+2\binprod{(\tilde{A}(k)-J)\ybf^{k}}{(\tilde{A}(k)-J)(\vbf^{k}-\vbf^{k-1})}+\rho^{2}\norm{\vbf^{k}-\vbf^{k-1}}^{2}.
\end{aligned}
\label{eq:y_consensus1}
\end{equation}
In the following, we bound $\binprod{(\tilde{A}(k)-J)\ybf^{k}}{(\tilde{A}(k)-J)(\vbf^{k}-\vbf^{k-1})}$
and $\norm{\vbf^{k}-\vbf^{k-1}}^{2}$, respectively. Since $\vbf^{k}-\vbf^{k-1}=\gbf(\xbf^{k};S^{k})-\gbf(\xbf^{k-1};S^{k})$
for $rT+1\leq k<(r+1)T$, we have
\begin{equation}
\mbb E\bsbra{\vbf^{k}-\vbf^{k-1}\mid\mcal F^{k}}=\nabla f_{\delta}(\xbf^{k})-\nabla f_{\delta}(\xbf^{k-1}).\label{eq:expectation_v}
\end{equation}
Then we have
\begin{equation}
\begin{aligned} & 2\mbb E\left[\binprod{(\tilde{A}(k)-J)\ybf^{k}}{(\tilde{A}(k)-J)(\vbf^{k}-\vbf^{k-1})}\mid\mcal F^{k}\right]\\
& = 2\Binprod{(\tilde{A}(k)-J)\ybf^{k}}{(\tilde{A}(k)-J)\mbb E\bsbra{\vbf^{k}-\vbf^{k-1}\mid\mcal F^{k}}}\\
& \overset{\eqref{eq:expectation_v}}{=}  2\binprod{(\tilde{A}(k)-J)\ybf^{k}}{(\tilde{A}(k)-J)(\nabla f_{\delta}(\xbf^{k})-\nabla f_{\delta}(\xbf^{k-1}))}\\
& \leq 2\rho\norm{\ybf^{k}-\bar{\ybf}^{k}}\cdot\rho\norm{\nabla f_{\delta}(\xbf^{k})-\nabla f_{\delta}(\xbf^{k-1})}\\
& \leq \alpha_{1}(\rho\norm{\ybf^{k}-\bar{\ybf}^{k}})^{2}+\alpha_{1}^{-1}(\rho L_{\delta}\norm{\xbf^{k}-\xbf^{k-1}})^{2}.
\end{aligned}
\label{eq:y_consensus2}
\end{equation}
Moreover, we have $\mbb E\bsbra{\norm{v_{i}^{k}-v_{i}^{k-1}}^{2}}=\mbb E\bsbra{\norm{g_{i}(x_{i}^{k};S_{i}^{k})-g_{i}(x_{i}^{k-1};S_{i}^{k})}^{2}}\leq {d^{2}L_{f}^{2}}/{(b\delta^{2})}\mbb E\bsbra{\norm{x_{i}^{k}-x_{i}^{k-1}}^{2}}$.
Summing it over $1$ to $m$, we have
\begin{equation}
\rho^{2}\mbb E\bsbra{\norm{\vbf^{k}-\vbf^{k-1}}^{2}}\leq\frac{\rho^{2}d^{2}L_{f}^{2}}{b\delta^{2}}\mbb E\bsbra{\norm{\xbf^{k}-\xbf^{k-1}}^{2}}.\label{eq:y_consensus3}
\end{equation}

Taking expectation of~\eqref{eq:y_consensus1} with all random variable
and combining it with~\eqref{eq:y_consensus2} and~\eqref{eq:y_consensus3}
yields
\begin{equation}
\begin{aligned} & \mbb E\bsbra{\norm{\ybf^{k+1}-\bar{\ybf}^{k+1}}^{2}}\\
&\leq  \rho^{2}(1+\alpha_{1})\mbb E\bsbra{\norm{\ybf^{k}-\bar{\ybf}^{k}}^{2}}+\Brbra{\frac{\rho^{2}L_{\delta}^{2}}{\alpha_{1}}+\frac{\rho^{2}d^{2}L_{f}^{2}}{b\delta^{2}}}\mbb E\bsbra{\norm{\xbf^{k}-\bar{\xbf}^{k}+\bar{\xbf}^{k}-\bar{\xbf}^{k-1}+\bar{\xbf}^{k-1}-\xbf^{k-1}}^{2}}\\
& \aleq \rho^{2}(1+\alpha_{1})\mbb E\bsbra{\norm{\ybf^{k}-\bar{\ybf}^{k}}^{2}}+\Brbra{\frac{\rho^{2}L_{\delta}^{2}}{\alpha_{1}}+\frac{\rho^{2}d^{2}L_{f}^{2}}{b\delta^{2}}}\mbb E\bsbra{3\norm{\xbf^{k}-\bar{\xbf}^{k}}^{2}+3\eta^{2}\norm{\bar{\vbf}^{k-1}}^{2}+3\norm{\bar{\xbf}^{k-1}-\xbf^{k-1}}^{2}},
\end{aligned}
\label{eq:y_consensus4}
\end{equation}
where $(a)$ holds by $\bar{\ybf}^{k+1}=\bar{\vbf}^{k},\bar{\xbf}^{k}=\bar{\xbf}^{k-1}-\eta\bar{\ybf}^{k}$.
Furthermore, for $k=rT$, $r=0,\cdots,R-1$, we have
\begin{equation*}
\begin{aligned} & \mbb E\bsbra{\norm{\ybf^{rT+1}-\bar{\ybf}^{rT+1}}^{2}}\\
 & =\mbb E\bsbra{\bnorm{\big(\tilde{A^{1}}(k)\tilde{A}^{2}(k)\dots\tilde{A}^{\Tcal}(k)-J\big)\gbf(\xbf^{rT};S^{rT\prime})}^{2}}\\
 & \aeq\mbb E\Bsbra{\Bnorm{\Big(\big(\tilde{A^{1}}(k)-J\big)\big(\tilde{A^{2}}(k)-J\big)\dots\big(\tilde{A}^{\Tcal}(k)-J\big)\Big)\gbf\brbra{\xbf^{rT};S^{rT\prime}}}^{2}}\\
 & \leq2\rho^{\Tcal}m(\sigma^{2}+L_{f}^{2}),
\end{aligned}
\end{equation*}
where (a) comes from that $J^{n}=J$ for any $n=1,2, \cdots$. 
\end{proof}

\subsection{Proof of Lemma~\ref{lem:dgfm_one_step_improve}}
\begin{lem}
[One Step Improvement]\label{lem:For-the-sequence}For the sequence
$\{\bar{x}^{k},\bar{v}^{k}\}$ generated by Algorithm~\ref{alg:dgfm_appendix}
and function $f_{\delta}(x)=\frac{1}{m}\sum_{i=1}^{m}f_{\delta}^{i}(x)$,
we have
\[
f_{\delta}(\bar{x}^{k+1})\leq f_{\delta}(\bar{x}^{k})-\frac{\eta}{2}\norm{\nabla f_{\delta}(\bar{x}^{k})}^{2}+\frac{\eta}{2}\norm{\nabla f_{\delta}(\bar{x}^{k})-\bar{v}^{k}}^{2}-\big(\frac{\eta}{2}-\frac{L_{\delta}\eta^{2}}{2}\big)\norm{\bar{v}^{k}}^{2}.
\]
\end{lem}
\begin{proof}
Letting $\hat{x}^{k+1}=\bar{x}^{k}-\eta\nabla f_{\delta}(\bar{x}^{k})$
and combining $\bar{x}^{k+1}=\bar{x}^{k}-\eta\bar{y}^{k+1}$, then
we have
\[
\begin{aligned} & f_{\delta}(\bar{x}^{k+1}) \\
&\leq  f_{\delta}(\bar{x}^{k})+\binprod{\nabla f_{\delta}(\bar{x}^{k})}{\bar{x}^{k+1}-\bar{x}^{k}}+\frac{L_{\delta}}{2}\norm{\bar{x}^{k+1}-\bar{x}^{k}}^{2}\\
& = f_{\delta}(\bar{x}^{k})+\binprod{\nabla f_{\delta}(\bar{x}^{k})\pm\bar{y}^{k+1}}{\bar{x}^{k+1}-\bar{x}^{k}}+\frac{L_{\delta}}{2}\norm{\bar{x}^{k+1}-\bar{x}^{k}}^{2}\\
& = f_{\delta}(\bar{x}^{k})+\binprod{\nabla f_{\delta}(\bar{x}^{k})-\bar{y}^{k+1}}{-\eta\bar{y}^{k+1}\pm\eta\nabla f_{\delta}(\bar{x}^{k})}-\brbra{{\eta}^{-1}-\frac{1}{2}{L_{\delta}}}\norm{\bar{x}^{k+1}-\bar{x}^{k}}^{2}\\
& = f_{\delta}(\bar{x}^{k})+\eta\norm{\nabla f_{\delta}(\bar{x}^{k})-\bar{y}^{k+1}}^{2}-\eta\binprod{\nabla f_{\delta}(\bar{x}^{k})-\bar{y}^{k+1}}{\nabla f_{\delta}(\bar{x}^{k})}-\brbra{{\eta}^{-1}-\frac{1}{2}{L_{\delta}}}\norm{\bar{x}^{k+1}-\bar{x}^{k}}^{2}\\
& = f_{\delta}(\bar{x}^{k})+\eta\norm{\nabla f_{\delta}(\bar{x}^{k})-\bar{y}^{k+1}}^{2}-\eta^{-1}\binprod{\bar{x}^{k+1}-\hat{x}^{k+1}}{\bar{x}^{k}-\hat{x}^{k+1}}-\brbra{{\eta}^{-1}-\frac{1}{2}{L_{\delta}}}\norm{\bar{x}^{k+1}-\bar{x}^{k}}^{2}\\
& = f_{\delta}(\bar{x}^{k})+\eta\norm{\nabla f_{\delta}(\bar{x}^{k})-\bar{y}^{k+1}}^{2}-\brbra{{\eta}^{-1}-\frac{1}{2}{L_{\delta}}}\norm{\bar{x}^{k+1}-\bar{x}^{k}}^{2}\\
 &\ \  -\frac{1}{2\eta}(\norm{\bar{x}^{k+1}-\hat{x}^{k+1}}^{2}+\norm{\bar{x}^{k}-\hat{x}^{k+1}}^{2}-\norm{\bar{x}^{k+1}-\bar{x}^{k}}^{2})\\
& = f_{\delta}(\bar{x}^{k})+\eta\norm{\nabla f_{\delta}(\bar{x}^{k})-\bar{y}^{k+1}}^{2}-\brbra{{\eta}^{-1}-\frac{1}{2}{L_{\delta}}}\norm{\bar{x}^{k+1}-\bar{x}^{k}}^{2}\\
 &\ \  -\frac{1}{2\eta}(\eta^{2}\norm{\nabla f_{\delta}(\bar{x}^{k})-\bar{y}^{k+1}}^{2}+\eta^{2}\norm{\nabla f_{\delta}(\bar{x}^{k})}^{2}-\norm{\bar{x}^{k+1}-\bar{x}^{k}}^{2})\\
& = f_{\delta}(\bar{x}^{k})-\frac{\eta}{2}\norm{\nabla f_{\delta}(\bar{x}^{k})}^{2}+\frac{\eta}{2}\norm{\nabla f_{\delta}(\bar{x}^{k})-\bar{y}^{k+1}}^{2}-\frac{1}{2}\brbra{{\eta}^{-1}-{L_{\delta}}}\norm{\bar{x}^{k+1}-\bar{x}^{k}}^{2}.
\end{aligned}
\]

Combining the above result and $\bar{y}^{k+1}=\bar{v}^{k},\bar{x}^{k+1}-\bar{x}^{k}=\eta\bar{y}^{k+1}$
completes our proof.
\end{proof}
\subsection{Proof of Lemma~\ref{lem:spider_variance_bound}}
\begin{lem}\label{lem:For-the-sequence-1}
For the sequence $\{\bar{x}^{k},\bar{v}^{k}\}$
generated by Algorithm~\ref{alg:dgfm_appendix}, for any $rT\leq k^{\prime}<k\leq(r+1)T-1$,
we have
\begin{equation}
\begin{aligned} & \mbb E\bsbra{\norm{\bar{v}^{k}-\nabla f_{\delta}(\bar{x}^{k})}^{2}}\\
& \leq \frac{2L_{\delta}^{2}}{m}\mbb E\bsbra{\norm{\xbf^{k}-\bar{\xbf}^{k}}^{2}}+2\mbb E\lrsbra{\Bnorm{\bar{v}^{k^{\prime}}-\frac{1}{m}\sum_{i=1}^{m}\nabla f_{\delta}^{i}(x_{i}^{k^{\prime}})}^{2}}\\
 &\ \  +\frac{6d^{2}L_{f}^{2}}{m^{2}\delta^{2}b}\sum_{j=k^{\prime}}^{k}{\mbb E\bsbra{\norm{\xbf^{j}-\bar{\xbf}^{j}}^{2}+\norm{\xbf^{j-1}-\bar{\xbf}^{j-1}}^{2}}}+\frac{6\eta^{2}d^{2}L_{f}^{2}}{m^{2}\delta^{2}b}\sum_{j=k^{\prime}}^{k}\mbb E\bsbra{\norm{\bar{\vbf}^{j-1}}^{2}}.
\end{aligned}
\label{eq:variance_bound}
\end{equation}
Moreover, for $k=rT$, $r=0,\cdots,R-1$, we have 
$
\mbb E\left[\Bnorm{\bar{v}^{k}-\frac{1}{m}\sum_{i=1}^{m}\nabla f_{\delta}^{i}(x_{i}^{k})}^{2}\right]\leq{\sigma^{2}}/{b^{\prime}}.
$
\end{lem}
\begin{proof}
For $rT+1\leq k\leq(r+1)T$, the update of $\bar{v}^{k}=\bar{v}^{k-1}+\frac{1}{m}\sum_{i=1}^{m}g_{i}(x_{i}^{k};S_{i}^{k})-g_{i}(x_{i}^{k-1};S_{i}^{k})$
leads to 
\[\begin{aligned} & \mbb E\bigg[\Bnorm{\bar{v}^{k}-\frac{1}{m}\sum_{i=1}^{m}\nabla f_{\delta}^{i}(x_{i}^{k})}^{2}\bigg]\\
& = \mbb E\bigg[\Bnorm{\bar{v}^{k-1}+\frac{1}{m}\sum_{i=1}^{m}g_{i}(x_{i}^{k};S_{i}^{k})-g_{i}(x_{i}^{k-1};S_{i}^{k})-\nabla f_{\delta}^{i}(x_{i}^{k})}^{2}\bigg]\\
& = \mbb E\bigg[\Bnorm{\bar{v}^{k-1}-\frac{1}{m}\sum_{i=1}^{m}\nabla f_{\delta}^{i}(x_{i}^{k-1})}^{2}+\Bnorm{\frac{1}{m}\sum_{i=1}^{m}\nabla f_{\delta}^{i}(x_{i}^{k-1})+g_{i}(x_{i}^{k};S_{i}^{k})-g_{i}(x_{i}^{k-1};S_{i}^{k})-\nabla f_{\delta}^{i}(x_{i}^{k})}^{2}\bigg].
\end{aligned}\]
Rearrange the above inequality, we have
\[
\small\begin{aligned} & \mbb E\bigg[\Bnorm{\bar{v}^{k}-\frac{1}{m}\sum_{i=1}^{m}\nabla f_{\delta}^{i}(x_{i}^{k})}^{2}\bigg]-\mbb E\bigg[\Bnorm{\bar{v}^{k-1}-\frac{1}{m}\sum_{i=1}^{m}\nabla f_{\delta}^{i}(x_{i}^{k-1})}^{2}\bigg]\\
& \leq \mbb E\bigg[\Bnorm{\frac{1}{m}\sum_{i=1}^{m}\nabla f_{\delta}^{i}(x_{i}^{k-1})+g_{i}(x_{i}^{k};S_{i}^{k})-g_{i}(x_{i}^{k-1};S_{i}^{k})-\nabla f_{\delta}^{i}(x_{i}^{k})}^{2}\bigg]\\
& = \frac{1}{m^{2}}\sum_{i=1}^{m}\mbb E\bigg[\bnorm{\nabla f_{\delta}^{i}(x_{i}^{k-1})+g_{i}(x_{i}^{k};S_{i}^{k})-g_{i}(x_{i}^{k-1};S_{i}^{k})-\nabla f_{\delta}^{i}(x_{i}^{k})}^{2}\bigg]\\
& \aleq  \frac{1}{m^{2}}\sum_{i=1}^{m}\mbb E\bsbra{\norm{g_{i}(x_{i}^{k};S_{i}^{k})-g_{i}(x_{i}^{k-1};S_{i}^{k})}^{2}}\\
& = \frac{1}{m^{2}b}\sum_{i=1}^{m}\mbb E\bsbra{\norm{g_{i}(x_{i}^{k};w_{i}^{k,1},\xi_{i}^{k,1})-g_{i}(x_{i}^{k-1};w_{i}^{k,1},\xi_{i}^{k,1})}^{2}}\\
& \leq  {\frac{d^{2}L_{f}^{2}}{m^{2}\delta^{2}b}}\mbb E\bsbra{\norm{\xbf^{k}-\xbf^{k-1}}^{2}}= {\frac{d^{2}L_{f}^{2}}{m^{2}\delta^{2}b}}\mbb E\bsbra{\norm{\xbf^{k}\pm\bar{\xbf}^{k}\pm\bar{\xbf}^{k-1}-\xbf^{k-1}}^{2}}\\
& \leq  {\frac{3d^{2}L_{f}^{2}}{m^{2}\delta^{2}b}}{\mbb E\bsbra{\norm{\xbf^{k}-\bar{\xbf}^{k}}^{2}+\norm{\xbf^{k-1}-\bar{\xbf}^{k-1}}^{2}}}+\frac{3\eta^{2}d^{2}L_{f}^{2}}{m^{2}\delta^{2}b}\mbb E\bsbra{\norm{\bar{\vbf}^{k-1}}^{2}},
\end{aligned}
\]
where $(a)$ holds by $\Var(X)\leq\mbb E[\norm X^{2}]$ for any random
vector $X$. Then, for any $rT\leq k^{\prime}<k<(r+1)T$,
we have
\begin{equation}
\begin{aligned}  &\mbb E\bigg[\Bnorm{\bar{v}^{k}-\frac{1}{m}  \sum_{i=1}^{m}\nabla f_{\delta}^{i}(x_{i}^{k})}^{2}\bigg] \\
& \leq \mbb E\bigg[\Bnorm{\bar{v}^{k^{\prime}}-\frac{1}{m}\sum_{i=1}^{m}\nabla f_{\delta}^{i}(x_{i}^{k^{\prime}})}^{2}\bigg]\\
 &\ \  + {\frac{3d^{2}L_{f}^{2}}{m^{2}\delta^{2}b}}\sum_{j=k^{\prime}}^{k-1} {\mbb E\bsbra{\norm{\xbf^{j+1}-\bar{\xbf}^{j+1}}^{2}+\norm{\xbf^{j}-\bar{\xbf}^{j}}^{2}}}+{ \frac{3\eta^{2}d^{2}L_{f}^{2}}{m^{2}\delta^{2}b}}\sum_{j=k^{\prime}}^{k-1}\mbb E\bsbra{\norm{\bar{\vbf}^{k^{\prime}}}^{2}}.
\end{aligned}
\label{eq:variance_bound01}
\end{equation}
Furthermore, we have
\begin{equation}
\mbb E\bigg[\Bnorm{\frac{1}{m}\sum_{i=1}^{m}\nabla f_{\delta}^{i}(x_{i}^{k})-\nabla f_{\delta}^{i}(\bar{x}^{k})}^{2}\bigg]\aleq\frac{L_{\delta}^{2}}{m}\mbb E\bigg[\sum_{i=1}^{m}\norm{x_{i}^{k}-\bar{x}^{k}}^{2}\bigg]=\frac{L_{\delta}^{2}}{m}\mbb E\bsbra{\norm{\xbf^{k}-\bar{\xbf}^{k}}^{2}},\label{eq:variance_bound02}
\end{equation}
where $(a)$ holds by $\norm{\nabla f_{\delta}^{i}(x_{i}^{k})-\nabla f_{\delta}^{i}(\bar{x}^{k})}\leq L_{\delta}\norm{x_{i}^{k}-\bar{x}^{k}}$
and $\norm x_{1}\leq\sqrt{m}\norm x_{2},\forall x\in\mbb R^{m}$.
Hence, combining~\eqref{eq:variance_bound01}
and~\eqref{eq:variance_bound02} yields
\begin{align*}\small
\begin{split}    
 & \mbb E\bsbra{\norm{\bar{v}^{k}-\nabla f_{\delta}(\bar{x}^{k})}^{2}}=\mbb E\lrsbra{\Bnorm{\bar{v}^{k}-\frac{1}{m}\sum_{i=1}^{m}\nabla f_{\delta}^{i}(\bar{x}^{k})}^{2}}\\
& \leq 2\mbb E\bigg[\Bnorm{\bar{v}^{k}-\frac{1}{m}\sum_{i=1}^{m}\nabla f_{\delta}^{i}(x_{i}^{k})}^{2}\bigg]+2\mbb E\lrsbra{\Bnorm{\frac{1}{m}\sum_{i=1}^{m}\nabla f_{\delta}^{i}(x_{i}^{k})-\nabla f_{\delta}^{i}(\bar{x}^{k})}^{2}}\\
& \leq \frac{2L_{\delta}^{2}}{m}\mbb E\bsbra{\norm{\xbf^{k}-\bar{\xbf}^{k}}^{2}}+2\mbb E\bigg[\Bnorm{\bar{v}^{k^{\prime}}-\frac{1}{m}\sum_{i=1}^{m}\nabla f_{\delta}^{i}(x_{i}^{k^{\prime}})}^{2}\bigg]\\
 &\ \ + {\frac{6d^{2}L_{f}^{2}}{m^{2}\delta^{2}b}}\sum_{j=k^{\prime}}^{k-1}{\mbb E\bsbra{\norm{\xbf^{j+1}-\bar{\xbf}^{j+1}}^{2}+\norm{\xbf^{j}-\bar{\xbf}^{j}}^{2}}}+{ \frac{6\eta^{2}d^{2}L_{f}^{2}}{m^{2}\delta^{2}b}}\sum_{j=k^{\prime}}^{k-1}\mbb E\bsbra{\norm{\bar{\vbf}^{j}}^{2}}
\end{split} 
\end{align*}
for any $rT\leq k^{\prime}<k<(r+1)T$.

Furthermore, for $k=rT$, we have
\[
\mbb E\left[\Bnorm{\bar{v}^{k}-\frac{1}{m}\sum_{i=1}^{m}\nabla f_{\delta}^{i}(x_{i}^{k})}^{2}\right]=\mbb E\left[\Bnorm{\frac{1}{m}\sum_{i=1}^{m}g_{i}(x_{i}^{k};S_{i}^{k\prime})-\nabla f_{\delta}^{i}(x_{i}^{k})}^{2}\right]\leq\frac{\sigma^{2}}{b^{\prime}}.
\]
\end{proof}
\begin{lem}
\label{lem:one_epoch_progress}Suppose
$K=R T$, then 
we
have
\begin{equation}
\begin{aligned} & \sum_{k=0}^{RT-1}\mbb E\bsbra{\norm{\bar{v}^{k}-\nabla f_{\delta}(\bar{x}^{k})}^{2}}\\
& \leq \frac{2L_{\delta}^{2}}{m}\sum_{k=0}^{RT-1}\mbb E\bsbra{\norm{\xbf^{k}-\bar{\xbf}^{k}}^{2}}+2RT\cdot {\frac{\sigma^{2}}{b^{\prime}}}\\
 &\ \  +\frac{6d^{2}L_{f}^{2}T}{m^{2}\delta^{2}b}\sum_{k=1}^{RT-1}{\mbb E\bsbra{\norm{\xbf^{k}-\bar{\xbf}^{k}}^{2}+\norm{\xbf^{k-1}-\bar{\xbf}^{k-1}}^{2}}}+\frac{6\eta^{2}d^{2}L_{f}^{2}T}{m^{2}\delta^{2}b}\sum_{k=0}^{RT-1}\mbb E\bsbra{\norm{\bar{v}^{k}}^{2}}.
\end{aligned}
\label{eq:the_whole_sequence_progress}
\end{equation}
\end{lem}
\begin{proof}
Taking $k^{\prime}=rT$ in~\eqref{eq:variance_bound}, then for any
$k=rT+t$, $t=1,\cdots,T-1$, we have
\begin{equation}
\begin{aligned} & \mbb E\bsbra{\norm{\bar{v}^{rT+t}-\nabla f_{\delta}(\bar{x}^{rT+t})}^{2}}\\
& \leq \frac{2L_{\delta}^{2}}{m}\mbb E\bsbra{\norm{\xbf^{rT+t}-\bar{\xbf}^{rT+t}}^{2}}+2\mbb E\bigg[\Bnorm{\bar{v}^{rT}-\frac{1}{m}\sum_{i=1}^{m}\nabla f_{\delta}^{i}(x_{i}^{rT})}^{2}\bigg]\\
 &\ \  +\frac{6d^{2}L_{f}^{2}}{m^{2}\delta^{2}b}\sum_{j=k^{\prime}}^{k}{\mbb E\bsbra{\norm{\xbf^{j}-\bar{\xbf}^{j}}^{2}+\norm{\xbf^{j-1}-\bar{\xbf}^{j-1}}^{2}}}+\frac{6\eta^{2}d^{2}L_{f}^{2}}{m^{2}\delta^{2}b}\sum_{j=k^{\prime}}^{k}\mbb E\bsbra{\norm{\bar{\vbf}^{j-1}}^{2}}\\
& \aleq  \frac{2L_{\delta}^{2}}{m}\mbb E\bsbra{\norm{\xbf^{rT+t}-\bar{\xbf}^{rT+t}}^{2}}+2\mbb E\bigg[\Bnorm{\bar{v}^{rT}-\frac{1}{m}\sum_{i=1}^{m}\nabla f_{\delta}^{i}(x_{i}^{rT})}^{2}\bigg]\\
 &\ \  +\frac{6d^{2}L_{f}^{2}}{m^{2}\delta^{2}b}\sum_{k=rT}^{(r+1)T-2}{\mbb E\bsbra{\norm{\xbf^{k+1}-\bar{\xbf}^{k+1}}^{2}+\norm{\xbf^{k}-\bar{\xbf}^{k}}^{2}}}+\frac{6\eta^{2}d^{2}L_{f}^{2}}{m^{2}\delta^{2}b}\sum_{k=rT}^{(r+1)T-2}\mbb E\bsbra{\norm{\bar{\vbf}^{k}}^{2}},
\end{aligned}
\label{eq:variance_bound-1}
\end{equation}
where $(a)$ holds by replacing index from $j$ to $k$ and $$\sum_{j=rT}^{rT+t-1}{\mbb E\bsbra{\norm{\xbf^{j+1}-\bar{\xbf}^{j+1}}^{2}+\norm{\xbf^{j}-\bar{\xbf}^{j}}^{2}}}\leq\sum_{k=rT}^{(r+1)T-2}{\mbb E\bsbra{\norm{\xbf^{k+1}-\bar{\xbf}^{k+1}}^{2}+\norm{\xbf^{k}-\bar{\xbf}^{k}}^{2}}}.$$
Summing $\mbb E\bsbra{\norm{\bar{v}^{rT}-\frac{1}{m}\sum_{i=1}^{m}\nabla f_{\delta}^{i}(x_{i}^{rT})}^{2}}$
and the above inequality over $t=1$ to $T-1$, and combining the
fact $\norm{\bar{\vbf}^{j}}^{2}=m\norm{\bar{v}^{j}}^{2}$ yields
\begin{equation*}
    \begin{aligned}
&\sum_{k=rT}^{(r+1)T-1}\mbb E\bsbra{ \norm{\bar{v}^{k}-\nabla f_{\delta}(\bar{x}^{k})}^{2}} \\ & \leq \frac{2L_{\delta}^{2}}{m}\sum_{k=rT+1}^{(r+1)T-1}\mbb E\bsbra{\norm{\xbf^{k}-\bar{\xbf}^{k}}^{2}}+(2(T-1)+1)\cdot\mbb E\bigg[\Bnorm{\bar{v}^{rT}-\frac{1}{m}\sum_{i=1}^{m}\nabla f_{\delta}^{i}(x_{i}^{rT})}^{2}\bigg]\\
 & \ \ +\frac{6d^{2}L_{f}^{2}T}{m^{2}\delta^{2}b}\sum_{k=rT}^{(r+1)T-2}{\mbb E\bsbra{\norm{\xbf^{k+1}-\bar{\xbf}^{k+1}}^{2}+\norm{\xbf^{k}-\bar{\xbf}^{k}}^{2}}}+\frac{6\eta^{2}d^{2}L_{f}^{2}T}{m^{2}\delta^{2}b}\sum_{k=rT}^{(r+1)T-2}\mbb E\bsbra{\norm{\bar{\vbf}^{k}}^{2}}.
    \end{aligned}
\end{equation*}
Summing the above inequality over $r=0$ to $R-1$ and combining $\mbb E\bsbra{\norm{\bar{v}^{rT}-\frac{1}{m}\sum_{i=1}^{m}\nabla f_{\delta}^{i}(x_{i}^{rT})}^{2}}\leq{\sigma^{2}}/{b^{\prime}}$,
we have
\[
\begin{aligned} & \sum_{k=0}^{RT-1}\mbb E\bsbra{\norm{\bar{v}^{k}-\nabla f_{\delta}(\bar{x}^{k})}^{2}}\\
& \leq \frac{2L_{\delta}^{2}}{m}\sum_{r=0}^{R-1}\sum_{k=rT+1}^{(r+1)T-1}\mbb E\bsbra{\norm{\xbf^{k}-\bar{\xbf}^{k}}^{2}}+2RT\cdot\frac{\sigma^{2}}{b^{\prime}}\\
 &\ \  +\frac{6d^{2}L_{f}^{2}T}{m^{2}\delta^{2}b}\sum_{r=0}^{R-1}\sum_{k=rT}^{(r+1)T-2}{\mbb E\bsbra{\norm{\xbf^{k+1}-\bar{\xbf}^{k+1}}^{2}+\norm{\xbf^{k}-\bar{\xbf}^{k}}^{2}}}+\frac{6\eta^{2}d^{2}L_{f}^{2}T}{m^{2}\delta^{2}b}\sum_{r=0}^{R-1}\sum_{k=rT}^{(r+1)T-2}\mbb E\bsbra{\norm{\bar{\vbf}^{k}}^{2}},
\end{aligned}
\]
which implies the desired result.
\end{proof}

\subsection{Proof of Lemma~\ref{lem:dgfm_convergence_thm}}\label{subsec:aux2}
\begin{lem}\label{lem:aux_lem_2}
For sequence $\{\xbf^{k},\ybf^{k},\vbf^{k}\}$
generated by Algorithm~\ref{alg:dgfm_appendix} and any positive $\beta_{x}>0,\beta_{y}>0$,
we have
\begin{equation}
\begin{aligned} & \beta_{y}\sum_{k=0}^{RT-1}\mbb E\bsbra{\norm{\ybf^{k+1}-\bar{\ybf}^{k+1}}^{2}-\norm{\ybf^{k}-\bar{\ybf}^{k}}^{2}}+\beta_{x}\sum_{k=0}^{RT-1}\mbb E\bsbra{\norm{\xbf^{k+1}-\bar{\xbf}^{k+1}}^{2}-\norm{\xbf^{k}-\bar{\xbf}^{k}}^{2}}\\
& \leq \Brbra{\beta_{x}\brbra{(1+\alpha_{2})\rho^{2}-1}+6\beta_{y}\brbra{\frac{\rho^{2}L_{\delta}^{2}}{\alpha_{1}}+\frac{\rho^{2}d^{2}L_{f}^{2}}{b\delta^{2}}}}\sum_{k=0}^{RT-1}\mbb E\bsbra{\norm{\xbf^{k}-\bar{\xbf}^{k}}^{2}}\\
&\ \ +3\beta_{y}\eta^{2}\brbra{\frac{\rho^{2}L_{\delta}^{2}}{\alpha_{1}}+\frac{\rho^{2}d^{2}L_{f}^{2}}{b\delta^{2}}}\sum_{k=0}^{RT-1}\mbb E\bsbra{\norm{\bar{\vbf}^{k-1}}^{2}}+\beta_{x}(1+\alpha_{2}^{-1})\rho^{2}\eta^{2}\sum_{k=0}^{RT-1}\mbb E\bsbra{\norm{\ybf^{k+1}-\bar{\ybf}^{k+1}}^{2}}\\
&\ \  +{2\rho^{\Tcal}m(\sigma^{2}+L_{f}^{2})}\cdot\beta_{y}R + \beta_{y}\brbra{\rho^{2}(1+\alpha_{1})-1}\sum_{k=0}^{RT-1}\mbb E\bsbra{\norm{\ybf^{k}-\bar{\ybf}^{k}}^{2}}.
\end{aligned}
\label{eq:all_consensus}
\end{equation}
\end{lem}
\begin{proof}
By~\eqref{eq:cycle_consensus} in Lemma~\ref{lem:consensus-dgfm_plus},
we have
\[
\mbb E\bsbra{\norm{\ybf^{rT+1}-\bar{\ybf}^{rT+1}}^{2}-\norm{\ybf^{rT}-\bar{\ybf}^{rT}}^{2}}\leq-\mbb E\bsbra{\norm{\ybf^{rT}-\bar{\ybf}^{rT}}^{2}}+{2\rho^{\Tcal}m(\sigma^{2}+L_{f}^{2})},\forall r=0,\cdots,R-1.
\]
By~\eqref{eq:y_consensus4-1} in Lemma~\ref{lem:consensus-dgfm_plus}, we
have
\[
\begin{aligned} & \mbb E\bsbra{\norm{\ybf^{k+1}-\bar{\ybf}^{k+1}}^{2}-\norm{\ybf^{k}-\bar{\ybf}^{k}}^{2}}\\
& \leq \brbra{\rho^{2}(1+\alpha_{1})-1}\mbb E\bsbra{\norm{\ybf^{k}-\bar{\ybf}^{k}}^{2}}+3\Brbra{\frac{\rho^{2}L_{\delta}^{2}}{\alpha_{1}}+\frac{\rho^{2}d^{2}L_{f}^{2}}{b\delta^{2}}}\mbb E\bsbra{\norm{\xbf^{k}-\bar{\xbf}^{k}}^{2}+\norm{\bar{\xbf}^{k-1}-\xbf^{k-1}}^{2}}\\
 &\ \  +3\eta^{2}\Brbra{\frac{\rho^{2}L_{\delta}^{2}}{\alpha_{1}}+\frac{\rho^{2}d^{2}L_{f}^{2}}{b\delta^{2}}}\mbb E\bsbra{\norm{\bar{\vbf}^{k-1}}^{2}},\ \ rT+1\leq k<(r+1)T,\forall r=0,1,\cdots,R-1.
\end{aligned}
\]

The two inequalities above implies that 
\begin{equation}
\small\begin{aligned} & \sum_{k=rT}^{(r+1)T-1}\mbb E\bsbra{\norm{\ybf^{k+1}-\bar{\ybf}^{k+1}}^{2}-\norm{\ybf^{k}-\bar{\ybf}^{k}}^{2}}\\
& \leq \brbra{\rho^{2}(1+\alpha_{1})-1}\sum_{k=rT}^{(r+1)T-1}\mbb E\bsbra{\norm{\ybf^{k}-\bar{\ybf}^{k}}^{2}}+3\Brbra{\frac{\rho^{2}L_{\delta}^{2}}{\alpha_{1}}+\frac{\rho^{2}d^{2}L_{f}^{2}}{b\delta^{2}}}\sum_{k=rT}^{(r+1)T-1}\mbb E\bsbra{\norm{\xbf^{k}-\bar{\xbf}^{k}}^{2}+\norm{\bar{\xbf}^{k-1}-\xbf^{k-1}}^{2}}\\
 &\ \  +3\eta^{2}\Brbra{\frac{\rho^{2}L_{\delta}^{2}}{\alpha_{1}}+\frac{\rho^{2}d^{2}L_{f}^{2}}{b\delta^{2}}}\sum_{k=rT}^{(r+1)T-1}\mbb E\bsbra{\norm{\bar{\vbf}^{k-1}}^{2}}+2\rho^{\Tcal}m(\sigma^{2}+L_{f}^{2}).
\end{aligned}
\label{eq:multi_beta_y}
\end{equation}

Furthermore, by~\eqref{eq:x_consensus} in Lemma~\ref{lem:consensus-dgfm_plus},
we have
\begin{equation}
\begin{aligned} & \sum_{k=rT}^{(r+1)T-1}\mbb E\bsbra{\norm{\xbf^{k+1}-\bar{\xbf}^{k+1}}^{2}-\norm{\xbf^{k}-\bar{\xbf}^{k}}^{2}}\\
& \leq \brbra{(1+\alpha_{2})\rho^{2}-1}\sum_{k=rT}^{(r+1)T-1}\mbb E\bsbra{\norm{\xbf^{k}-\bar{\xbf}^{k}}^{2}}+(1+\alpha_{2}^{-1})\rho^{2}\eta^{2}\sum_{k=rT}^{(r+1)T-1}\mbb E\bsbra{\norm{\ybf^{k+1}-\bar{\ybf}^{k+1}}^{2}}.
\end{aligned}
\label{eq:multi_beta_x}
\end{equation}

Summing~\eqref{eq:multi_beta_x} times $\beta_{x}$ and~\eqref{eq:multi_beta_y}
times $\beta_{y}$, we have 
\begin{equation}
\begin{aligned} & \beta_{y}\sum_{k=rT}^{(r+1)T-1}\mbb E\bsbra{\norm{\ybf^{k+1}-\bar{\ybf}^{k+1}}^{2}-\norm{\ybf^{k}-\bar{\ybf}^{k}}^{2}}+\beta_{x}\sum_{k=rT}^{(r+1)T-1}\mbb E\bsbra{\norm{\xbf^{k+1}-\bar{\xbf}^{k+1}}^{2}-\norm{\xbf^{k}-\bar{\xbf}^{k}}^{2}}\\
& \leq \beta_{x}\brbra{(1+\alpha_{2})\rho^{2}-1}\sum_{k=rT}^{(r+1)T-1}\mbb E\bsbra{\norm{\xbf^{k}-\bar{\xbf}^{k}}^{2}}+\beta_{x}(1+\alpha_{2}^{-1})\rho^{2}\eta^{2}\sum_{k=rT}^{(r+1)T-1}\mbb E\bsbra{\norm{\ybf^{k+1}-\bar{\ybf}^{k+1}}^{2}}\\ &\ \ +3\beta_{y}\Brbra{\frac{\rho^{2}L_{\delta}^{2}}{\alpha_{1}}+\frac{\rho^{2}d^{2}L_{f}^{2}}{b\delta^{2}}}\sum_{k=rT}^{(r+1)T-1}\mbb E\bsbra{\norm{\xbf^{k}-\bar{\xbf}^{k}}^{2}+\norm{\bar{\xbf}^{k-1}-\xbf^{k-1}}^{2}}\\
 &\ \  +3\beta_{y}\eta^{2}\Brbra{\frac{\rho^{2}L_{\delta}^{2}}{\alpha_{1}}+\frac{\rho^{2}d^{2}L_{f}^{2}}{b\delta^{2}}}\sum_{k=rT}^{(r+1)T-1}\mbb E\bsbra{\norm{\bar{\vbf}^{k-1}}^{2}}+2\rho^{\Tcal}m(\sigma^{2}+L_{f}^{2})\cdot\beta_{y}\\
 &\ \  + \beta_{y}\brbra{\rho^{2}(1+\alpha_{1})-1}\sum_{k=rT}^{(r+1)T-1}\mbb E\bsbra{\norm{\ybf^{k}-\bar{\ybf}^{k}}^{2}}.
\end{aligned}
\label{eq:one_epoch_consensus}
\end{equation}

Now, summing it over $r=0$ to $R-1$ and combining $\sum_{k=0}^{RT-1}\norm{\bar{\xbf}^{k-1}-\xbf^{k-1}}^{2}\leq\sum_{k=0}^{RT-1}\norm{\bar{\xbf}^{k}-\xbf^{k}}^{2}$
complete our proof.
\end{proof}

\begin{lem}
\label{lem:main_theorem}For sequence $\{\xbf^{k},\ybf^{k}\}$
generated by {\dgfm}, we have

\begin{equation}
\begin{aligned} &\ \  -\beta_{y}\mbb E\bsbra{\norm{\ybf^{0}-\bar{\ybf}^{0}}^{2}}+\beta_{x}\mbb E\bsbra{\norm{\xbf^{RT}-\bar{\xbf}^{RT}}^{2}-\norm{\xbf^{0}-\bar{\xbf}^{0}}^{2}}+\mbb E\bsbra{f_{\delta}(\bar{x}^{RT})}-\mbb E\bsbra{f_{\delta}(\bar{x}^{0})}\\
& \leq -\frac{\eta}{2}\sum_{k=0}^{RT-1}\mbb E\bsbra{\norm{\nabla f_{\delta}(\bar{x}^{k})}^{2}}+\eta RT\cdot {\frac{\sigma^{2}}{b^{\prime}}}\\
 &\ \  -\lrrbra{\frac{\eta}{2}-\frac{L_{\delta}\eta^{2}}{2}-\frac{3\eta^{3}d^{2}L_{f}^{2}T}{m^{2}\delta^{2}b}-3\beta_{y}m\eta^{2}\Brbra{\frac{\rho^{2}L_{\delta}^{2}}{\alpha_{1}}+\frac{\rho^{2}d^{2}L_{f}^{2}}{b\delta^{2}}}}\sum_{k=0}^{RT-1}\mbb E\bsbra{\norm{\bar{v}^{k}}^{2}}\\
 &\ \  -\lrrbra{\beta_{x}\brbra{1-(1+\alpha_{2})\rho^{2}}-6\beta_{y}\Brbra{\frac{\rho^{2}L_{\delta}^{2}}{\alpha_{1}}+\frac{\rho^{2}d^{2}L_{f}^{2}}{b\delta^{2}}}-\Brbra{\frac{\eta L_{\delta}^{2}}{m}+\frac{6\eta d^{2}L_{f}^{2}T}{m^{2}\delta^{2}b}}}\sum_{k=0}^{RT-1}\mbb E\bsbra{\norm{\xbf^{k}-\bar{\xbf}^{k}}^{2}}\\
 &\ \  -\Brbra{\beta_{y}\brbra{1-\rho^{2}(1+\alpha_{1})}-\beta_{x}(1+\alpha_{2}^{-1})\rho^{2}\eta^{2}}\sum_{k=0}^{RT-1}\mbb E\bsbra{\norm{\ybf^{k}-\bar{\ybf}^{k}}^{2}}\\
 &\ \  -\brbra{\beta_{y}-\beta_{x}(1+\alpha_{2}^{-1})\rho^{2}\eta^{2}}\mbb E\bsbra{\norm{\ybf^{RT}-\bar{\ybf}^{RT}}^{2}}+{ 2\rho^{\Tcal}m(\sigma^{2}+L_{f}^{2})}\cdot\beta_{y}R.\text{}
\end{aligned}
\label{eq:all_consensus-1-1-1}
\end{equation}
\end{lem}
\begin{proof}
By Lemma~\ref{lem:For-the-sequence}, for all $k$, we have
\[
f_{\delta}(\bar{x}^{k+1})\leq f_{\delta}(\bar{x}^{k})-\frac{\eta}{2}\bnorm{\nabla f_{\delta}(\bar{x}^{k})}^{2}+\frac{\eta}{2}\bnorm{\nabla f_{\delta}(\bar{x}^{k})-\bar{v}^{k}}^{2}-\brbra{\frac{\eta}{2}-\frac{L_{\delta}\eta^{2}}{2}}\norm{\bar{v}^{k}}^{2}.
\]

Taking expectations of both sides of the above inequality and summing
it over $k=0$ to $RT-1$, we have 
\[
\begin{aligned} & \mbb E\bsbra{f_{\delta}(\bar{x}^{RT})}-\mbb E\bsbra{f_{\delta}(\bar{x}^{0})}\\
& \leq -\frac{\eta}{2}\sum_{k=0}^{RT-1}\mbb E\bsbra{\norm{\nabla f_{\delta}(\bar{x}^{k})}^{2}}-\Big(\frac{\eta}{2}-\frac{L_{\delta}\eta^{2}}{2}\Big)\sum_{k=0}^{RT-1}\mbb E\bsbra{\norm{\bar{v}^{k}}^{2}}\\
 &\ \  +\frac{\eta}{2}\sum_{k=0}^{RT-1}\mbb E\bsbra{\norm{\nabla f_{\delta}(\bar{x}^{k})-\bar{v}^{k}}^{2}}\\
& \leq -\frac{\eta}{2}\sum_{k=0}^{RT-1}\mbb E\bsbra{\norm{\nabla f_{\delta}(\bar{x}^{k})}^{2}}-\Big(\frac{\eta}{2}-\frac{L_{\delta}\eta^{2}}{2}\Big)\sum_{k=0}^{RT-1}\mbb E\bsbra{\norm{\bar{v}^{k}}^{2}}\\
 &\ \  +\frac{\eta L_{\delta}^{2}}{m}\sum_{k=0}^{RT-1}\mbb E\bsbra{\norm{\xbf^{k}-\bar{\xbf}^{k}}^{2}}+\eta RT\cdot {\frac{\sigma^{2}}{b^{\prime}}}\\
 &\ \  +\frac{3\eta d^{2}L_{f}^{2}}{m^{2}\delta^{2}b}\sum_{k=1}^{RT-1}\brbra{\mbb E\bsbra{\norm{\xbf^{k}-\bar{\xbf}^{k}}^{2}+\norm{\xbf^{k-1}-\bar{\xbf}^{k-1}}^{2}}}+\frac{3\eta^{3}d^{2}L_{f}^{2}}{m^{2}\delta^{2}b}\sum_{k=0}^{RT-1}\mbb E\bsbra{\norm{\bar{v}^{k}}^{2}}.\\
& \aleq  -\frac{\eta}{2}\sum_{k=0}^{RT-1}\mbb E\bsbra{\norm{\nabla f_{\delta}(\bar{x}^{k})}^{2}}-\Big(\frac{\eta}{2}-\frac{L_{\delta}\eta^{2}}{2}-\frac{3\eta^{3}d^{2}L_{f}^{2}T}{m^{2}\delta^{2}b}\Big)\sum_{k=0}^{RT-1}\mbb E\bsbra{\norm{\bar{v}^{k}}^{2}}+\eta RT\cdot {\frac{\sigma^{2}}{b^{\prime}}}\\
 &\ \  +\Big(\frac{\eta L_{\delta}^{2}}{m}+\frac{6\eta d^{2}L_{f}^{2}T}{m^{2}\delta^{2}b}\Big)\sum_{k=0}^{RT-1}\mbb E\bsbra{\norm{\xbf^{k}-\bar{\xbf}^{k}}^{2}},
\end{aligned}
\]
where $(a)$ holds by $\sum_{k=0}^{RT-1}\mbb E\bsbra{\norm{\xbf^{k-1}-\bar{\xbf}^{k-1}}^{2}}\leq\sum_{k=0}^{RT-1}\mbb E\bsbra{\norm{\xbf^{k}-\bar{\xbf}^{k}}^{2}}$.
Summing the above inequality and~\eqref{eq:all_consensus}, we have

\begin{equation}
\begin{aligned} &\ \  -\beta_{y}\mbb E\bsbra{\norm{\ybf^{0}-\bar{\ybf}^{0}}^{2}}+\beta_{x}\mbb E\bsbra{\norm{\xbf^{RT}-\bar{\xbf}^{RT}}^{2}-\norm{\xbf^{0}-\bar{\xbf}^{0}}^{2}}+\mbb E\bsbra{f_{\delta}(\bar{x}^{RT})}-\mbb E\bsbra{f_{\delta}(\bar{x}^{0})}\\
& \aleq  -\frac{\eta}{2}\sum_{k=0}^{RT-1}\mbb E\bsbra{\norm{\nabla f_{\delta}(\bar{x}^{k})}^{2}}-\lrrbra{\frac{\eta}{2}-\frac{L_{\delta}\eta^{2}}{2}-\frac{3\eta^{3}d^{2}L_{f}^{2}T}{m^{2}\delta^{2}b}-3\beta_{y}m\eta^{2}\Brbra{\frac{\rho^{2}L_{\delta}^{2}}{\alpha_{1}}+\frac{\rho^{2}d^{2}L_{f}^{2}}{b\delta^{2}}}}\sum_{k=0}^{RT-1}\mbb E\bsbra{\norm{\bar{v}^{k}}^{2}}\\
 &\ \  -\left(\beta_{x}(1-(1+\alpha_{2})\rho^{2})-6\beta_{y}\Brbra{\frac{\rho^{2}L_{\delta}^{2}}{\alpha_{1}}+\frac{\rho^{2}d^{2}L_{f}^{2}}{b\delta^{2}}}-\Brbra{\frac{\eta L_{\delta}^{2}}{m}+\frac{6\eta d^{2}L_{f}^{2}T}{m^{2}\delta^{2}b}}\right)\sum_{k=0}^{RT-1}\mbb E\bsbra{\norm{\xbf^{k}-\bar{\xbf}^{k}}^{2}}\\
 &\ \  -\brbra{\beta_{y}(1-\rho^{2}(1+\alpha_{1}))-\beta_{x}(1+\alpha_{2}^{-1})\rho^{2}\eta^{2}}\sum_{k=0}^{RT-1}\mbb E\bsbra{\norm{\ybf^{k}-\bar{\ybf}^{k}}^{2}} +\eta RT\cdot {\frac{\sigma^{2}}{b^{\prime}}} \\
 &\ \  -\brbra{\beta_{y}-\beta_{x}(1+\alpha_{2}^{-1})\rho^{2}\eta^{2}}\mbb E\bsbra{\norm{\ybf^{RT}-\bar{\ybf}^{RT}}^{2}}+{ 2\rho^{\Tcal}m(\sigma^{2}+L_{f}^{2})}\cdot\beta_{y}R,
\end{aligned}
\label{eq:all_consensus-1}
\end{equation}
where $(a)$ holds by $ {\sum_{k=0}^{RT-1}}\mbb E[\norm{\bar{\vbf}^{k-1}}^{2}]\leq m {\sum_{k=0}^{RT-1}}\mbb E[\norm{\bar{v}^{k}}^{2}]$.
Then we complete our proof.
\end{proof}

\subsection{Proof of Theorem~\ref{thm:dgfm_cor}}
\begin{thm}
    {\dgfm} can output a $(\delta,\vep)$-Goldstein stationary point of $f(\cdot)$ in expectation with the total stochastic zeroth-order complexity is at most $\mcal O(\Delta_{\delta}\delta^{-1}\vep^{-3})$ and the total number of communication rounds is at most $\Ocal(\Delta_{\delta}\vep^{-2}(\vep^{-1}+\log(\vep^{-1})))$ by setting
    $\alpha_1 = \alpha_2 = ({1-\rho^2})/{2\rho^2}$, $\delta=\mcal O(\vep)$, $\beta_x = \Ocal(\delta^{-1})$, $\beta_y = \Ocal(\delta)$, $b^{\prime} = \Ocal(\vep^{-2}), b=\mcal O(\vep^{-1})$, $\eta = \mcal O(\vep)$, $T=\mcal O(\vep^{-1})$, $R=\mcal O(\Delta_{\delta}\vep^{-2})$, $\mcal{T} = \mcal O(\log(\vep^{-1}))$. Moreover, a specific example of parameters is given as follows. Define $\eta_1 = (1-\rho^2)^{\frac{3}{2}}\delta^{\frac{1}{2}}\rho^{-2}(1+\rho^2)^{-\frac{1}{2}}d^{-\frac{1}{2}}/24^{\frac{1}{2}}$ and set
\begin{equation*}
\begin{aligned}
        &\beta_{y}  =\frac{(1-\rho^{2})\delta\eta}{2\rho^{4}c^{2}m}\cdot\brbra{\frac{c^{2}}{2\delta}+2T},\ \ \beta_{x}=\frac{(1-\rho^{2})^{2}}{2\rho^{2}(1+\rho^{2})\eta^{2}}\cdot\beta_{y},\ \ b=\frac{d}{m\vep},\\
        &b^{\prime}=\frac{\sigma^{2}}{12}\vep^{-2},\ \ \mcal T=\brbra{\log c^{2}\vep - \log 36(\sigma^{2}+L_{f}^{2})(1-\rho^{2})} \cdot ({\log\rho})^{-1}+2,\\
        &\eta = \min \Bcbra{\eta_1, \frac{1}{2\sqrt{3dT}}\Brbra{\frac{L_{f}^{2}}{m\vep}+\frac{3(1-\rho^{2})}{2c^{2}}\cdot\brbra{\frac{c^{2}L_{f}^{2}}{\left(1-\rho^{2}\right)\delta}+2\rho^{2}L_{f}^{2}m}}^{-\frac{1}{2}}, \frac{1}{2}L_{\delta}^{-1}},
\end{aligned}
\end{equation*}
then {\dgfm} output a $(\delta,\vep)$-Goldstein stationary point of~$f(\cdot)$ in expectation with total stochastic zeroth-order complexity at most
$\mcal O\brbra{\max\{\Delta_{\delta}\delta^{-1}\vep^{-3}d^{\frac{3}{2}}m^{-\frac{1}{2}},\Delta_{\delta}\vep^{-\frac{7}{2}}d^{\frac{3}{2}}m^{-\frac{1}{2}}\}}$ and the total communication rounds is at most
$\mcal O\brbra{(\vep^{-1}+\log (\vep^{-1}d))\cdot \max\{\Delta_{\delta}\vep^{-\frac{3}{2}}d^{\frac{1}{2}}m^{\frac{1}{2}},\Delta_{\delta}\vep^{-2}d^{\frac{1}{2}}m^{\frac{1}{2}}\}}.$
\end{thm}

\begin{proof}
It follows from Lemma~\ref{lem:main_theorem} that 
\begin{equation*}
\begin{aligned} & \frac{\eta}{2}\sum_{k=0}^{RT-1}\mbb E\bsbra{\norm{\nabla f_{\delta}(\bar{x}^{k})}^{2}}\\
& \leq \beta_{y}\mbb E\bsbra{\norm{\ybf^{0}-\bar{\ybf}^{0}}^{2}}+\beta_{x}\mbb E[\norm{\xbf^{0}-\bar{\xbf}^{0}}^{2}-\norm{\xbf^{RT}-\bar{\xbf}^{RT}}^{2}]+\mbb E\bsbra{f_{\delta}(\bar{x}^{0})}-\mbb E\bsbra{f_{\delta}(\bar{x}^{RT})}+\eta RT\frac{\sigma^{2}}{b^{\prime}}\\
 &\ \  -\underbrace{\Brbra{\frac{\eta}{2}-\frac{L_{\delta}\eta^{2}}{2}-\frac{3\eta^{3}d^{2}L_{f}^{2}T}{m^{2}\delta^{2}b}-3\beta_{y}m\eta^{2}\brbra{\frac{\rho^{2}L_{\delta}^{2}}{\alpha_{1}}+\frac{\rho^{2}d^{2}L_{f}^{2}}{b\delta^{2}}}}}_{\clubsuit}\sum_{k=0}^{RT-1}\mbb E\bsbra{\norm{\bar{v}^{k}}^{2}}\\
 &\ \  -\underbrace{\Brbra{\beta_{x}\brbra{1-(1+\alpha_{2})\rho^{2}}-6\beta_{y}\brbra{\frac{\rho^{2}L_{\delta}^{2}}{\alpha_{1}}+\frac{\rho^{2}d^{2}L_{f}^{2}}{b\delta^{2}}}-\brbra{\frac{\eta L_{\delta}^{2}}{m}+\frac{6\eta d^{2}L_{f}^{2}T}{m^{2}\delta^{2}b}}}}_{\spadesuit}\sum_{k=0}^{RT-1}\mbb E\bsbra{\norm{\xbf^{k}-\bar{\xbf}^{k}}^{2}}\\
 &\ \  -\underbrace{\brbra{\beta_{y}(1-\rho^{2}(1+\alpha_{1}))-\beta_{x}(1+\alpha_{2}^{-1})\rho^{2}\eta^{2}}}_{\blacklozenge}\sum_{k=0}^{RT-1}\mbb E\bsbra{\norm{\ybf^{k}-\bar{\ybf}^{k}}^{2}}\\
 &\ \  -\underbrace{\brbra{\beta_{y}-\beta_{x}(1+\alpha_{2}^{-1})\rho^{2}\eta^{2}}}_{\varheart}\mbb E\bsbra{\norm{\ybf^{RT}-\bar{\ybf}^{RT}}^{2}}+{ 2\rho^{\Tcal}m(\sigma^{2}+L_{f}^{2})}\cdot\beta_{y}R.
\end{aligned}
\end{equation*}
Since we choose proper initial point $y_i^0$ and $x_i^0$, then  we simplify the above inequality as 
\begin{equation}
    \begin{aligned}
         & \clubsuit\sum_{k=0}^{RT-1}\mbb E\bsbra{\norm{\bar{v}^{k}}^{2}}+\spadesuit\sum_{k=0}^{RT-1}\mbb E\bsbra{\norm{\xbf^{k}-\bar{\xbf}^{k}}^{2}}+\blacklozenge\sum_{k=0}^{RT-1}\mbb E\bsbra{\norm{\ybf^{k}-\bar{\ybf}^{k}}^{2}}\\&+\varheart \cdot \mbb E\bsbra{\norm{\ybf^{RT}-\bar{\ybf}^{RT}}^{2}}+\frac{\eta}{2}\sum_{k=0}^{RT-1}\mbb E\bsbra{\norm{\nabla f_{\delta}(\bar{x}^{k})}^{2}}\\
& \leq \mbb E[f_{\delta}(\bar{x}^{0})]-\mbb E[f_{\delta}(\bar{x}^{RT})]+{ 2\rho^{\Tcal}m(\sigma^{2}+L_{f}^{2})}\cdot\beta_{y}R.\label{eq:dgfmres1}
    \end{aligned}
\end{equation}
Combining the parameters setting of $\alpha_1,R,T,\beta_x$, $\beta_y$ and $\eta$ yields $\clubsuit,\spadesuit,\blacklozenge,\varheart>0$ and 
\begin{equation}\label{eq:dgfmorder}
    \clubsuit = \mcal O(\vep),\ \spadesuit = \Ocal (\delta^{-1}),\ \blacklozenge=\Ocal (\delta),\ \varheart = \Ocal(\delta).
\end{equation}
By~\eqref{eq:dgfmorder}, we can simplify~\eqref{eq:dgfmres1} as 
\begin{equation}\label{eq:dgfmres2}
\spadesuit\sum_{k=0}^{RT-1}\mbb E[\norm{\xbf^{k}-\bar{\xbf}^{k}}^{2}]+\frac{\eta}{2}\sum_{k=0}^{RT-1}\mbb E[\norm{\nabla f_{\delta}(\bar{x}^{k})}^{2}]\leq \Delta_{\delta} + { 2\rho^{\Tcal}m(\sigma^{2}+L_{f}^{2})}\cdot\beta_{y}R.
\end{equation}
Divide ${\eta \cdot RT}/{2}$ on both sides of~\eqref{eq:dgfmres2}, we have
\begin{equation*}
    \frac{2 \cdot \spadesuit}{\eta \cdot RT}\sum_{k=0}^{RT-1}\mbb E[\norm{\xbf^{k}-\bar{\xbf}^{k}}^{2}]+\frac{1}{ RT}\sum_{k=0}^{RT-1}\mbb E[\norm{\nabla f_{\delta}(\bar{x}^{k})}^{2}]\leq \frac{2\Delta_{\delta}}{\eta \cdot RT} + \frac{4\rho^{\Tcal}m(\sigma^{2}+L_{f}^{2})\cdot\beta_{y}}{\eta \cdot T}.
\end{equation*}
Combining the parameter setting of $\mcal T, \eta, R, T$ and $\delta$,  we have ${\spadesuit}/{\eta} = \mcal O(\delta^{-2})$ and
\begin{equation*}
    \frac{2 \cdot \spadesuit}{\eta \cdot RT}\sum_{k=0}^{RT-1}\mbb E[\norm{\xbf^{k}-\bar{\xbf}^{k}}^{2}]+\frac{1}{ RT}\sum_{k=0}^{RT-1}\mbb E[\norm{\nabla f_{\delta}(\bar{x}^{k})}^{2}] = \mcal O(\vep^2).
\end{equation*}
By ${\spadesuit}/{\eta} = \mcal O(\delta^{-2})$, $\nabla f_{\delta}(x)\in \partial_{\delta}f(x)$ and $\norm{\nabla f_{\delta}(x_i^k)}^2\leq 2\norm{\nabla f_{\delta}(\bar{x}^k)}^2 + 2L_{\delta}^{2}\norm{x_i^k-\bar{x}^k}^2$, then we have a complexity $\mcal O(\Delta_{\delta} \delta^{-1}\vep^{-3})$ for a $(\delta,\vep)$-Goldstein stationary point in expectation.

Now, suppose $\delta\leq\vep$, $T\geq{c^{2}}/{2\delta},$ and $R\geq{24\Delta_{\delta}\delta}\vep^{-2}\eta^{-1}{c^{-2}}$
hold, we give specific parameter settings to obtain $(\delta,\vep)$-Goldstein
stationary point in expectation, which can be summarized as follows:
\begin{equation}
\begin{aligned}
        &\beta_{y}  =\frac{(1-\rho^{2})\delta\eta}{2\rho^{4}c^{2}m}\cdot\brbra{\frac{c^{2}}{2\delta}+2T},\ \ \beta_{x}=\frac{(1-\rho^{2})^{2}}{2\rho^{2}(1+\rho^{2})\eta^{2}}\cdot\beta_{y},\ \ \\        
        &b=\max\{1,\lceil\frac{d}{m\vep}\rceil\},\ \ b^{\prime}=\max\{1,\lceil\frac{\sigma^{2}}{12}\vep^{-2}\rceil\},\ \ \mcal T={\big(\log{c^{2}\vep}-\log{36(\sigma^{2}+L_{f}^{2})(1-\rho^{2})}}\big)\big({\log\rho}\big)^{-1}+2,\\
        & R=\mcal O(\Delta_{\delta}\vep^{-2}d^{\frac{1}{2}}m^{\frac{1}{2}})\\
        &\eta =\min\Bcbra{\eta_1,\frac{1}{2}L_{\delta}^{-1},\frac{1}{2\sqrt{3dT}}\Brbra{\frac{L_{f}^{2}}{m\vep}+\frac{3(1-\rho^{2})}{2c^{2}}\cdot\brbra{\frac{c^{2}L_{f}^{2}}{\left(1-\rho^{2}\right)\delta}+2\rho^{2}L_{f}^{2}m}}^{-\tfrac{1}{2}}}.
        \label{eq:betayxbb}
\end{aligned}
\end{equation}
By the definition of $\beta_{x}$ in~\eqref{eq:betayxbb}, we have
\begin{equation}
\begin{aligned}
\varheart \geq\blacklozenge & =\frac{1-\rho^{2}}{2}\beta_{y}-\frac{\rho^{2}(1+\rho^{2})}{1-\rho^{2}}\cdot\eta^{2}\beta_{x}\geq\frac{1-\rho^{2}}{2}\beta_{y}-\frac{\rho^{2}(1+\rho^{2})}{1-\rho^{2}}\cdot\frac{\left(1-\rho^{2}\right)^{2}}{2\rho^{2}(1+\rho^{2})\eta^{2}}\eta^{2}\beta_{y}=0.
\end{aligned}
\label{eq:triangle_ineq}
\end{equation}
Since $L_{\delta}={cL_{f}\sqrt{d}}/{\delta},b={d}/{m\vep},\delta\leq\vep$ and 
\begin{equation}\label{eq:temp_eta_ineq}
    \eta^{2}\leq\frac{(1-\rho^{2})^{3}\delta}{24\rho^{4}(1+\rho^{2})d}\cdot\Brbra{\frac{2\rho^{2}c^{2}L_{f}^{2}}{1-\rho^{2}}+2mL_{f}^{2}}^{-1},
\end{equation}
then we have
\begin{equation}
\begin{aligned}
& \frac{(1-\rho^{2})^{3}}{4\rho^{2}(1+\rho^{2})\eta^{2}}-6\Brbra{\frac{2\rho^{4}L_{\delta}^{2}}{1-\rho^{2}}+\frac{\rho^{2}d^{2}L_{f}^{2}}{\delta^{2}b}}\\
& \geq  \frac{(1-\rho^{2})^{3}}{4\rho^{2}(1+\rho^{2})}\cdot\frac{24\rho^{4}(1+\rho^{2})}{(1-\rho^{2})^{3}}\cdot \Brbra{\frac{3\rho^{2}L_{\delta}^{2}}{1-\rho^{2}}+\frac{mdL_{f}^{2}}{\delta}} 
 -6\Brbra{\frac{2\rho^{4}L_{\delta}^{2}}{1-\rho^{2}}+\frac{\rho^{2}mdL_{f}^{2}}{\delta}} =  \frac{6\rho^{4}L_{\delta}^{2}}{1-\rho^{2}}\geq0.
\end{aligned}
\label{eq:temp01}
\end{equation}
 Based on the above inequality, for term $\spadesuit$, we have
\begin{equation}
\begin{aligned}
\spadesuit & =\frac{1-\rho^{2}}{2}\beta_{x}-6\beta_{y}\Brbra{\frac{2\rho^{4}L_{\delta}^{2}}{1-\rho^{2}}+\frac{\rho^{2}d^{2}L_{f}^{2}}{\delta^{2}b}}-\Brbra{\frac{\eta L_{\delta}^{2}}{m}+\frac{6\eta d^{2}L_{f}^{2}T}{m^{2}\delta^{2}b}}\\
 & =\Brbra{\frac{(1-\rho^{2})^{3}}{4\rho^{2}(1+\rho^{2})\eta^{2}}-6\Brbra{\frac{2\rho^{4}L_{\delta}^{2}}{1-\rho^{2}}+\frac{\rho^{2}d^{2}L_{f}^{2}}{\delta^{2}b}}}\beta_{y}-\Brbra{\frac{\eta L_{\delta}^{2}}{m}+\frac{6\eta d^{2}L_{f}^{2}T}{m^{2}\delta^{2}b}}\\
 & \ageq\frac{6\rho^{2}L_{\delta}^{2}}{1-\rho^{2}}\Brbra{\frac{6\rho^{4}L_{\delta}^{2}}{1-\rho^{2}}}^{-1}\Brbra{\frac{3L_{\delta}^{2}}{2m}+\frac{6dL_{f}^{2}T}{m\delta}}\eta-\Brbra{\frac{L_{\delta}^{2}}{m}+\frac{6dL_{f}^{2}T}{m\delta}}\eta=\frac{\eta c^{2}L_{f}^{2}d}{2m\delta},
\end{aligned}
\label{eq:spade_ineq_appendix}
\end{equation}
where $(a)$ holds by 
$\beta_{y}={(1-\rho^{2})\delta\eta}\cdot({c^{2}}/{2\delta}+2T)/{2\rho^{4}c^{2}m}$
,
$b={d}/{m\vep}$ and~\eqref{eq:temp01}. It follows from~\eqref{eq:temp_eta_ineq}
that 
\begin{equation}
\begin{aligned}\clubsuit & =\frac{\eta}{2}-\frac{L_{\delta}\eta^{2}}{2}-\frac{3\eta^{3}d^{2}L_{f}^{2}T}{m^{2}\delta^{2}b}-3\beta_{y}m\eta^{2}\Brbra{\frac{\rho^{2}L_{\delta}^{2}}{1-\rho^{2}}+\frac{2\rho^{4}d^{2}L_{f}^{2}}{b\delta^{2}}}\\
 & \ageq\frac{\eta}{2}-\frac{L_{\delta}\eta^{2}}{2}-\frac{3\eta^{3}d^{2}L_{f}^{2}T}{m^{2}\delta^{2}b}-\frac{3(1-\rho^{2})}{2c^{2}}\cdot\brbra{\frac{c^{2}}{2\delta}+2T}\Brbra{\frac{c^{2}L_{f}^{2}d}{\left(1-\rho^{2}\right)\delta}+\frac{2\rho^{2}d^{2}L_{f}^{2}}{b\delta}}\eta^{3}\\
 & \bgeq\frac{\eta}{4}-\Brbra{\frac{3dL_{f}^{2}T}{m\vep}+\big(\frac{c^{2}L_{f}^{2}d}{\left(1-\rho^{2}\right)\delta}+2\rho^{2}L_{f}^{2}md\big)\cdot\frac{9\left(1-\rho^{2}\right)T}{2c^{2}}}\eta^{3}\geq0,
\end{aligned}
\label{eq:plum_ineq}
\end{equation}
where $(a)$ is by the definition of $\beta_{y}$ in~\eqref{eq:betayxbb}
and $(b)$ follows from $b={d}/{m\vep},\delta=\vep,2\eta\leq L_{\delta}^{-1},{c^{2}}/{2\delta}\leq T$.
Now, we give a upper bound of $\frac{1}{mRT}\sum_{k=0}^{RT-1}\sum_{i=1}^{m}\mbb E\bsbra{\norm{\nabla f_{\delta}(x_{i}^{k})}^{2}}$:
\begin{equation*}
\begin{aligned}
&\frac{1}{mRT}\sum_{k=0}^{RT-1}\sum_{i=1}^{m}\mbb E\bsbra{\norm{\nabla f_{\delta}(x_{i}^{k})}^{2}} \\
& =  \frac{1}{mRT}\sum_{k=0}^{RT-1}\sum_{i=1}^{m}\mbb E\bsbra{\norm{\nabla f_{\delta}(x_{i}^{k}-\bar{x}^{k}+\bar{x}^{k})}^{2}}\\
& \leq \frac{2}{mRT}\sum_{k=0}^{RT-1}\sum_{i=1}^{m}\mbb E\bsbra{\norm{\nabla f_{\delta}(x_{i}^{k}-\bar{x}^{k})}^{2}}+\frac{2}{mRT}\sum_{k=0}^{RT-1}\sum_{i=1}^{m}\mbb E\bsbra{\norm{\nabla f_{\delta}(\bar{x}^{k})}^{2}}\\
& \leq \frac{2L_{\delta}^{2}}{mRT}\sum_{k=0}^{RT-1}\sum_{i=1}^{m}\mbb E\bsbra{\norm{x_{i}^{k}-\bar{x}^{k}}^{2}}+\frac{2}{RT}\sum_{k=0}^{RT-1}\mbb E\bsbra{\norm{\nabla f_{\delta}(\bar{x}^{k})}^{2}}\\
& =  \frac{2L_{\delta}^{2}}{mRT}\sum_{k=0}^{RT-1}\mbb E\bsbra{\norm{\xbf^{k}-\bar{\xbf}^{k}}^{2}}+\frac{2}{RT}\sum_{k=0}^{RT-1}\mbb E\bsbra{\norm{\nabla f_{\delta}(\bar{x}^{k})}^{2}}\\
& \aleq  \frac{2\Delta_{\delta}}{\eta RT}+\frac{4}{\eta T}\rho^{\mcal T}m(\sigma^{2}+L_{f}^{2})\cdot\beta_{y}+\frac{\sigma^{2}}{2b^{\prime}}\\
& \leq \frac{2\Delta_{\delta}}{\eta RT}+6\rho^{\mcal T-2}(\sigma^{2}+L_{f}^{2})\cdot\frac{(1-\rho^{2})\delta}{c^{2}}+\frac{\sigma^{2}}{2b^{\prime}}\bleq  \vep^{2}
\end{aligned}
\end{equation*}
where $(a)$ holds by combining~\eqref{eq:triangle_ineq},~\eqref{eq:spade_ineq_appendix},~\eqref{eq:plum_ineq}
and $(b)$ follows from the definition of $b^{\prime},R,T$ and $\mcal T$
in~\eqref{eq:betayxbb}. Hence, the total zeroth-order
complexity of {\dgfm} for $(\delta,\vep)$-Goldstein stationary point
in expectation is $RTb+Rb^{\prime}=\mcal O(\Delta_{\delta}\delta^{-1}\vep^{-2}d^{\frac{1}{2}}m^{\frac{1}{2}}\cdot \max\{1,\frac{d}{m\vep}\})$ and the total communication rounds is at most \
$RT+R\mcal T=\mcal O((\vep^{-1}+\log (\vep^{-1}d))\cdot \Delta_{\delta}\vep^{-2}d^{\frac{1}{2}}m^{\frac{1}{2}}).$
\end{proof}
\end{document}


\global\long\def\dgfm{\text{DGFM}$^{+}$}%
\global\long\def\dgfmlin{\text{DGFM}}%
\global\long\def\mnist{\text{MNIST}}%
\global\long\def\fashion{\text{Fashion-MNIST}}%
\onecolumn
\begin{center}
\Huge Supplementary of Decentralized Gradient Free Methods for Stochastic
Non-Smooth Non-Convex Optimization
\end{center}

In this supplementary, we list all the Propositions, Lemmas, and Theorems
mentioned in the paper and proved them.

\section{Additional Literature Review}

\paragraph{Review of non-smooth non-convex optimization}
\citet{zhang2020complexity} introduced $(\delta,\epsilon)$-Goldstein stationarity as a valid criterion for non-smooth non-convex optimization, which makes it possible to analyze the non-asymptotic convergence. They also provided novel algorithms with complexity 
$\Ocal(\Delta L_f^3 \delta^{-1} \epsilon^{-4})$, where $\Delta$ is defined as $f(x_0)-\inf_{x\in \mbb R^d} f(x)\leq \Delta$.
Later, by adding random perturbation, \citet{davis2022gradient} and~\citet{tian2022finite} relaxed the subgradient selection oracle assumption.
More recently,~\citet{cutkosky2023optimal} found a connection between
non-convex stochastic optimization and online learning and established a stochastic first-order oracle
complexity of $\Ocal(\Delta L_f^2 \delta^{-1} \epsilon^{-3})$, which is the optimal rate for all $\epsilon\leq\Ocal(\delta)$. These preliminary attempts are first-order algorithms. 
In light of recent advances in designing zeroth-order algorithms for non-smooth non-convex problems, an effective way is by applying the randomized smoothing technique~\cite{nesterov2017random,shamir2017optimal}. This approach constructs a smooth surrogate function to which algorithms for smooth functions can be applied.
\citet{lin2022gradient} established the relationship between Goldstein stationarity of the original function and $\epsilon$-stationarity of the surrogate function constructed via randomized smoothing and presented an algorithm for computing a $(\delta,\epsilon)$-Goldstein stationary point within at most $\Ocal(L_f^4\epsilon^{-4}+ \Delta L_f^3 \delta^{-1}\epsilon^{-4})$ stochastic zeroth-order oracle calls. Later, \citet{chen2023faster} constructed stochastic
recursive gradient estimators to accelerate and achieve a stochastic zeroth-order oracle complexity of $\Ocal(L_f^3 \epsilon^{-3}+ \Delta L_f^2 \delta^{-1}\epsilon^{-3})$.

\section{Algorithms}

We give the two decentralized algorithms here for completeness.

\begin{algorithm}[htpb]
\caption{\label{alg:dgfmlin_appendix}{\dgfmlin} at each node $i$}
    \begin{algorithmic}[1]
        \Require{$x_{i}^{-1}=x_i^0=\bar{x}^0,\forall i \in [m]$, $y_i^0 = g_i(x_i^{-1};S_i^{-1})=\zerobf_{d},K,\eta$}
        \For{$k=0,\ldots,K-1$}
        \State{Sample $S_i^{k}=\bcbra{\xi_{i}^{k,1},w_i^{k,1}}$ and calculate $g_i(x_i^{k};S_i^k)$}
        \State{$y_i^{k+1}=\sum_{j=1}^{m}a_{i,j}(k)(y_j^k+g_j(x_j^k;S_j^k)-g_j(x_j^{k-1};S_j^{k-1}))$}
        \State{$x_i^{k+1}=\tsum_{j=1}^{m}a_{i,j}(k)(x_j^k - \eta y_j^{k+1})$}
        \EndFor
    \Return{Choose $x_{\text{out}}$ uniformly at random from $\bcbra{{x}^k_i}_{k=1,\cdots,K, i=1,\cdots,m}$}
        \end{algorithmic}
\end{algorithm}

\begin{algorithm}[h]
\caption{\label{alg:dgfm_appendix}{\dgfm} at each node $i$}
    \begin{algorithmic}[1]
        \Require{$x_{i}^{-1}=x_i^0=\bar{x}^0,\forall i \in [m]$, $y_i^0 = v_i^{-1}=\zerobf_{d},K,\eta,b,b^{\prime}$}
        \For{$k=0,\ldots,K-1$}
            \If{$k$ mod $T$ = 0}
            \State Sample $S_i^{k\prime}=\bcbra{(\xi_i^{k\prime,j},w_i^{k\prime,j})}_{j=1}^{b^\prime}$
            \State Calculate $y_i^{k+1} = v_i^k = g_i(x_i^k;S_i^k)$
            \For{$\tau = 1,\cdots,\mcal T$}
            \State $y_i^{k+1}=\tsum_{j=1}^m a_{i,j}^\tau(k)y_j^{k+1}$
        \EndFor
        \Else
        \State Sample $S_i^{k}=\bcbra{(\xi_i^{k,j},w_i^{k,j})}_{j=1}^{b}$
        \State  $v_i^k = v_i^{k-1}+g_i(x_i^k;S_i^k)-g_i(x_i^{k-1};S_i^k)$
        \State $y_i^{k+1}=\tsum_{j=1}^{m}a_{i,j}(k)(y_j^k + v_j^k - v_j^{k-1})$
        \EndIf
        \State $x_i^{k+1}=\tsum_{j=1}^{m}a_{i,j}(k)(x_j^k - \eta y_j^{k+1})$
        \EndFor
    \Return{Choose $x_{\text{out}}$ uniformly at random from $\bcbra{{x}^k_i}_{k=1,\cdots,K, i=1,\cdots,m}$}
        \end{algorithmic}
\end{algorithm}

\section{Notations}

In this section, we list the notations in Table~\ref{tab:Notations},
which will be used repeatedly.

\begin{table}
[htp]\caption{\label{tab:Notations}Notations for {\dgfmlin} and {\dgfm}}
\renewcommand{\arraystretch}{1.3}
\centering{}%
\begin{tabular}{|c|c|c|c|}
\hline 
Notations & Specific Formulation & Specific Meaning & Dimension\tabularnewline
\hline 
\hline 
$\xbf^{k}$ & $[(x_{1}^{k})^{\top},\cdots,(x_{m}^{k})^{\top}]^{\top}$ & Stack all local variables $x_{i}^{k}$ & $\mbb R^{md}$\tabularnewline
\hline 
$\tilde{\xbf}^{k}$ & $[(\tilde{x}_{1}^{k})^{\top},\cdots,(\tilde{x}_{m}^{k})^{\top}]^{\top}$ & Stack all local variables $\tilde{x}_{i}^{k}$ & $\mbb R^{md}$\tabularnewline
\hline 
$\bar{\xbf}^{k}$ & $\onebf_{m}\otimes\bar{x}^{k}$ & Copy variable $\bar{x}^{k}$ and concatenate & $\mbb R^{md}$\tabularnewline
\hline 
$\ybf^{k}$ & $[(y_{1}^{k})^{\top},\cdots,(y_{m}^{k})^{\top}]^{\top}$ & Stack all local variables $y_{i}^{k}$ & $\Rbb^{md}$\tabularnewline
\hline 
$\bar{\ybf}^{k}$ & $\onebf_{m}\otimes\bar{y}^{k}$ & Copy variable $\bar{y}^{k}$ and concatenate & $\mbb R^{md}$\tabularnewline
\hline 
$\vbf^{k}$ & $[(v_{1}^{k})^{\top},\cdots,(v_{m}^{k})^{\top}]^{\top}$ & Stack all local variables $v_{i}^{k}$ & $\Rbb^{md}$\tabularnewline
\hline 
$\bar{\vbf}^{k}$ & $\onebf_{m}\otimes\bar{v}^{k}$ & Copy variable $\bar{v}^{k}$ and concatenate & $\mbb R^{md}$\tabularnewline
\hline 
$\tilde{A}(k)$ & $A(k)\otimes\Ibf_{d}$ & Doubly Stochastic Matrix for nodes & $\mbb R^{md\times md}$\tabularnewline
\hline 
$J$ & $\tfrac{1}{m}\onebf_{m}\onebf_{m}^{\top}\otimes\Ibf_{d}$ & Mean Matrix & $\mbb R^{md\times md}$\tabularnewline
\hline 
$\gbf(\xbf^{k};S^{k})$ & $[(g_{1}(x_{1}^{k};S_{1}^{k}))^{\top},\cdots,(g_{m}(x_{m}^{k};S_{m}^{k}))^{\top}]^{\top}$ & Stack all local variables $g(x_{i}^{k};S_{i}^{k})$ & $\Rbb^{md}\to\Rbb^{md}$\tabularnewline
\hline 
$\nabla f_{\delta}^{i}(x)$ & $\mbb E_{\xi}[\nabla f_{\delta}^{i}(x;\xi)]=\mbb E_{S_{i}}[g(x_{i};S_{i})]$ & Expectation of stochastic gradient & $\mbb R^{d}\to\mbb R^{d}$\tabularnewline
\hline 
$\nabla F_{\delta}(\xbf^{k})$ & $[(\nabla f_{\delta}^{1}(x_{1}^{k}))^{\top},\cdots,(\nabla f_{\delta}^{m}(x_{m}^{k}))^{\top}]^{\top}$ & Stack all expectation of stochastic gradient & $\mbb R^{md}\to\mbb R^{md}$\tabularnewline
\hline 
\end{tabular}
\end{table}

\section{Preliminaries}
\begin{prop}
\label{prop:common_used}Let $\tilde{A}(k)=A(k)\otimes\Ibf_{d},J=\tfrac{1}{m}\onebf_{m}\onebf_{m}^{\top}\otimes\Ibf_{d}$,
then we have
\end{prop}
\begin{enumerate}
\item \label{enu:.1property}$\tilde{A}(k)J=J=J\tilde{A}(k)$.
\item \label{enu:.2property}$\tilde{A}(k)\bar{\zbf}^{k}=\bar{\zbf}^{k}=J\bar{\zbf}^{k},\forall\bar{z}^{k}\in\mbb R^{d}$.
\item \label{enu:3property}There exists a constant $\rho>0$ such that
$\max_{k}\{\norm{\tilde{A}(k)-J}\}\leq\rho<1$.
\item \label{enu:.4spectral_norm}$\norm{\tilde{A}(k)}\leq1$.
\end{enumerate}
\begin{proof}
First, we have
\[
\tilde{A}(k)J=(A(k)\otimes\Ibf_{m})\cdot(\tfrac{1}{m}\onebf_{m}\onebf_{m}^{\top}\otimes\Ibf_{m})=(\tfrac{1}{m}A(k)\onebf_{m}\onebf_{m}^{\top}\otimes\Ibf_{m})\aeq\tfrac{1}{m}\onebf_{m}\onebf_{m}^{T}\otimes\Ibf_{m}=J\beq J\tilde{A}(k),
\]
where $(a)$ holds by doubly stochasticity and $(b)$ follows from
the similar argument in above $\tilde{A}(k)J=J$ based on $\onebf_{m}^{T}\tilde{A}(k)=\onebf_{m}^{T}$.

Second, it follows from $\bar{\zbf}^{k}=\onebf\otimes\bar{z}^{k}$
that
\[
\tilde{A}(k)\bar{\zbf}^{k}=(A(k)\otimes\Ibf_{d})(\onebf\otimes\bar{z}^{k})=\onebf\otimes\bar{z}^{k}=\bar{\zbf}^{k}=(\tfrac{1}{m}\onebf_{m}\onebf_{m}^{T}\otimes\Ibf_{d})\otimes(\onebf\otimes\bar{z}^{k})=J\bar{\zbf}^{k}.
\]

Third, $\norm{\tilde{A}(k)-J}\leq\rho$ is a well known conclusion
(see e.g.,~\cite{tsitsiklis1984problems}).

Last, $\norm{\tilde{A}(k)}\leq1$ is a well known result (see e.g.,~\cite{sun2022distributed}).
\end{proof}

\section{Convergence Analysis for {\dgfmlin}}
\begin{lem}
Let $\{x_{i}^{k},y_{i}^{k},g_{i}(x_{i}^{k};S_{i}^{k})\}$ be the sequence
generated by {\dgfmlin} and $\{\xbf^{k},\ybf^{k},\gbf^{k}\}$ be
the corresponding stack variables, then for $k\geq0$, we have
\begin{align}
\xbf^{k+1} & =\tilde{A}(k)(\xbf^{k}-\eta\ybf^{k+1}),\label{eq:x_update_dgfm}\\
\ybf^{k+1} & =\tilde{A}(k)(\ybf^{k}+\gbf^{k}-\gbf^{k-1}),\label{eq:y_update_dgfm}\\
\bar{\xbf}^{k+1} & =\bar{\xbf}^{k}-\eta\bar{\ybf}^{k+1},\label{eq:mean_update_x_dgfm}\\
\bar{\ybf}^{k+1} & =\bar{\ybf}^{k}+\bar{\gbf}^{k}-\bar{\gbf}^{k-1},\label{eq:mean_y_update}\\
\bar{\ybf}^{k+1} & =\bar{\gbf}^{k}.\label{eq:y_g_mean}
\end{align}
\end{lem}
\begin{proof}
It follows from $x_{i}^{k+1}=\sum_{j=1}^{m}a_{i,j}(k)(x_{j}^{k}-\eta y_{j}^{k+1})$
and the definition of $\tilde{A}(k)$ and $\xbf^{k},\ybf^{k}$ that~\eqref{eq:x_update_dgfm}.
Similar result~\eqref{eq:y_update_dgfm} holds by $y_{i}^{k+1}=\sum_{j=1}^{m}a_{i,j}(k)(y_{j}^{k}+g_{j}^{k}-g_{j}^{k-1})$.
It follows from $J\til A(k)=J$ in~\ref{enu:.1property} of Proposition~\ref{prop:common_used}
andt~\eqref{eq:x_update_dgfm},~\eqref{eq:y_update_dgfm} that~\eqref{eq:mean_update_x_dgfm}
and~\eqref{eq:mean_y_update} hold. Furthermore, combining $y_{i}^{0}=g_{i}^{-1}=\zerobf_{d}$
and~\eqref{eq:mean_y_update} yields~\eqref{eq:y_g_mean}.
\end{proof}
\begin{lem}[Consensus Error Decay]
\label{lem:Let--generated}Let $\{\bar{\xbf}^{k},\bar{\ybf}^{k}\}$
be the sequence generated by~{\dgfmlin}, then for $k\geq0$, we have
\begin{equation}
\begin{split} & \mbb E[\norm{\xbf^{k+1}-\bar{\xbf}^{k+1}}^{2}]\\
\leq &\  \rho^{2}(1+\alpha_{1})\mbb E[\norm{\xbf^{k}-\bar{\xbf}^{k}}^{2}]+\rho^{2}\eta^{2}(1+\alpha_{1}^{-1})\mbb E[\norm{\ybf^{k+1}-\bar{\ybf}^{k+1}}^{2}],
\end{split}
\label{eq:consensus_x_dgfmlin}
\end{equation}
and 
\begin{equation}
\begin{split} & \mbb E[\norm{\ybf^{k+1}-\bar{\ybf}^{k+1}}^{2}\mid\mcal F_{k}]\\
\leq &\ \rho^{2}(1+\alpha_{2})\mbb E[\norm{\ybf^{k}-\bar{\ybf}^{k}}^{2}\mid\mcal F_{k}]+\rho^{2}(1+\alpha_{2}^{-1})(6+36{ \eta^{2}}L(\delta)^{2})m\sigma^{2}\\
 & +9\rho^{2}(1+\alpha_{2}^{-1})L(\delta)^{2}\mbb E[\norm{\xbf^{k}-\bar{\xbf}^{k}}^{2}\mid\mcal F_{k}]\\
 & +9\rho^{2}(1+\alpha_{2}^{-1}L(\delta)^{2}(1+4\eta^{2}L(\delta)^{2})\mbb E[\norm{\xbf^{k-1}-\bar{\xbf}^{k-1}}^{2}\mid\mcal F_{k}]\\
 & +18L(\delta)^{2}\eta^{2}{ \rho^{2}}(1+\alpha_{2}^{-1})\mbb E[\norm{\nabla f_{\delta}(\bar{x}^{k-1})}^{2}\mid\mcal F_{k}].
\end{split}
\label{eq:consensus_y_dgfmlin}
\end{equation}
\end{lem}
\begin{proof}
We consider sequence $\{\xbf^{k}\},\{\ybf^{k}\}$, separately.

\textbf{Sequence $\{\xbf^{k}\}$: }It follows from~\eqref{eq:x_update_dgfm}
that
\[
\xbf^{k+1}-\bar{\xbf}^{k+1}=(\tilde{A}(k)-J)(\xbf^{k}-\eta\ybf^{k+1}).
\]
Combining the above inequality and~\eqref{enu:.2property},~\eqref{enu:3property}
in Proposition~\ref{prop:common_used}, we have
\begin{align*}
\norm{\xbf^{k+1}-\bar{\xbf}^{k+1}} & \leq\rho\norm{\xbf^{k}-\bar{\xbf}^{k}}+\rho\eta\norm{\ybf^{k+1}-\bar{\ybf}^{k+1}}.
\end{align*}
Combining the above inequality and Young's inequality and taking the expectation,
then we complete the proof of~\eqref{eq:consensus_x_dgfmlin}.

\textbf{Sequence $\{\ybf^{k}\}$:} It follows from~\eqref{eq:y_update_dgfm}
that 
\[
\norm{\ybf^{k+1}-\bar{\ybf}^{k+1}}\leq\rho\norm{\ybf^{k}-\bar{\ybf}^{k}}+\rho\norm{\gbf^{k}-\gbf^{k-1}}.
\]
Combining the above inequality, Young's inequality, and taking the expectation
on natural field $\mcal F_{k}$, then we have
\begin{align*}
\mbb E[\norm{\ybf^{k+1}-\bar{\ybf}^{k+1}}^{2}\mid\mcal F_{k}] & \leq\rho^{2}(1+\alpha_{2})\mbb E[\norm{\ybf^{k}-\bar{\ybf}^{k}}^{2}\mid\mcal F_{k}]+\rho^{2}(1+\alpha_{2}^{-1})\mbb E[\norm{\gbf^{k}-\gbf^{k-1}}^{2}\mid\mcal F_{k}].
\end{align*}
Note that 
\[
\begin{split}\mbb E[\norm{\gbf^{k}-\gbf^{k-1}}^{2}\mid\mcal F_{k}]= & \sum_{i=1}^{m}\mbb E[\norm{g_{i}(x_{i}^{k};S_{i}^{k})\pm\nabla f_{\delta}^{i}(x_{i}^{k})\pm\nabla f_{\delta}^{i}(x_{i}^{k-1})-g_{i}(x_{i}^{k};S_{i}^{k-1})}^{2}\mid\mcal F_{k}]\\
\leq & 3\sum_{i=1}^{m}(2\sigma^{2}+L(\delta)^{2}\mbb E[\norm{x_{i}^{k}-x_{i}^{k-1}}^{2}\mid\mcal F_{k}])\\
= & 6m\sigma^{2}+3L(\delta)^{2}\mbb E[\norm{\xbf^{k}-\xbf^{k-1}}^{2}\mid\mcal F_{k}]\\
\leq & 6m\sigma^{2}+9L(\delta)^{2}(\mbb E[\norm{\xbf^{k}-\bar{\xbf}^{k}}^{2}+\norm{\bar{\xbf}^{k}-\bar{\xbf}^{k-1}}^{2}+\norm{\xbf^{k-1}-\bar{\xbf}^{k-1}}^{2}\mid\mcal F_{k}])\\
= & 6m\sigma^{2}+9L(\delta)^{2}(\mbb E[\norm{\xbf^{k}-\bar{\xbf}^{k}}^{2}+\eta^{2}\norm{\bar{\ybf}^{k}}^{2}+\norm{\xbf^{k-1}-\bar{\xbf}^{k-1}}^{2}\mid\mcal F_{k}]),
\end{split}
\]
and 
\[
\begin{split}\mbb E[\norm{\bar{\ybf}^{k}}^{2}\mid\mcal F_{k}]\leq & m\cdot\mbb E[(2\norm{\tfrac{1}{m}\sum_{i=1}^{m}(g_{i}(x_{i}^{k-1};S_{i}^{k-1})\pm\nabla f_{\delta}^{i}(x_{i}^{k-1})-\nabla f_{\delta}^{i}(\bar{x}^{k-1})}^{2}+2\norm{\tfrac{1}{m}\sum_{i=1}^{m}\nabla f_{\delta}^{i}(\bar{x}^{k-1})}^{2})\mid\mcal F_{k}]\\
= & 4m\mbb E[\norm{\tfrac{1}{m}\sum_{i=1}^{m}g_{i}(x_{i}^{k-1};S_{i}^{k-1})-\nabla f_{\delta}^{i}(x_{i}^{k-1})}^{2}\mid\mcal F_{k}]\\
 & +4m\mbb E[\norm{\tfrac{1}{m}\sum_{i=1}^{m}f_{\delta}^{i}(x_{i}^{k-1})-\nabla f_{\delta}^{i}(\bar{x}^{k-1})}^{2}\mid\mcal F_{k}]+2\mbb E[\norm{\nabla f_{\delta}(\bar{x}^{k-1})}^{2}\mid\mcal F_{k}]\\
\leq & 4\mbb E[\sum_{i=1}^{m}\norm{g_{i}(x_{i}^{k-1};S_{i}^{k-1})-\nabla f_{\delta}^{i}(x_{i}^{k-1})}^{2}]+4\mbb E[\sum_{i=1}^{m}\norm{\nabla f_{\delta}^{i}(x_{i}^{k-1})-\nabla f_{\delta}^{i}(\bar{x}^{k-1})}^{2}\mid\mcal F_{k}]\\
 & +2\mbb E[\norm{\nabla f_{\delta}(\bar{x}^{k-1})}^{2}\mid\mcal F_{k}]\\
\leq & 4m\sigma^{2}+4L(\delta)^{2}\mbb E[\norm{\xbf^{k-1}-\bar{\xbf}^{k-1}}^{2}\mid\mcal F_{k}]+2\mbb E[\norm{\nabla f_{\delta}(\bar{x}^{k-1})}^{2}\mid\mcal F_{k}].
\end{split}
\]

Combining the above three inequalities and taking expectation with
all random variables yields
\[
\begin{aligned}\mbb E[\norm{\ybf^{k+1}-\bar{\ybf}^{k+1}}^{2}\mid\mcal F_{k}]\leq & \rho^{2}(1+\alpha_{2})\mbb E[\norm{\ybf^{k}-\bar{\ybf}^{k}}^{2}\mid\mcal F_{k}]+\rho^{2}(1+\alpha_{2}^{-1})\mbb E[\norm{\gbf^{k}-\gbf^{k-1}}^{2}\mid\mcal F_{k}]\\
\leq & \rho^{2}(1+\alpha_{2})\mbb E[\norm{\ybf^{k}-\bar{\ybf}^{k}}^{2}\mid\mcal F_{k}]\\
 & +\rho^{2}(1+\alpha_{2}^{-1})\left(6m\sigma^{2}+9L(\delta)^{2}(\mbb E[\norm{\xbf^{k}-\bar{\xbf}^{k}}^{2}+\eta^{2}\norm{\bar{\ybf}^{k}}^{2}+\norm{\xbf^{k-1}-\bar{\xbf}^{k-1}}^{2}\mid\mcal F_{k}])\right)\\
\leq & \rho^{2}(1+\alpha_{2})\mbb E[\norm{\ybf^{k}-\bar{\ybf}^{k}}^{2}\mid\mcal F_{k}]\\
 & +\rho^{2}(1+\alpha_{2}^{-1})\left(6m\sigma^{2}+9L(\delta)^{2}(\mbb E[\norm{\xbf^{k}-\bar{\xbf}^{k}}^{2}+\norm{\xbf^{k-1}-\bar{\xbf}^{k-1}}^{2}\mid\mcal F_{k}])\right.\\
 & \left.+9L(\delta)^{2}\eta^{2}(4m\sigma^{2}+4L(\delta)^{2}\mbb E[\norm{\xbf^{k-1}-\bar{\xbf}^{k-1}}^{2}\mid\mcal F_{k}]+2\mbb E[\norm{\nabla f_{\delta}(\bar{x}^{k-1})}^{2}\mid\mcal F_{k}]\right)\\
= & \rho^{2}(1+\alpha_{2})\mbb E[\norm{\ybf^{k}-\bar{\ybf}^{k}}^{2}\mid\mcal F_{k}]+\rho^{2}(1+\alpha_{2}^{-1})(6+36{ \eta^{2}}L(\delta)^{2})m\sigma^{2}\\
 & +9\rho^{2}(1+\alpha_{2}^{-1})L(\delta)^{2}\mbb E[\norm{\xbf^{k}-\bar{\xbf}^{k}}^{2}\mid\mcal F_{k}]\\
 & +9\rho^{2}(1+\alpha_{2}^{-1}L(\delta)^{2}(1+4\eta^{2}L(\delta)^{2})\mbb E[\norm{\xbf^{k-1}-\bar{\xbf}^{k-1}}^{2}\mid\mcal F_{k}]\\
 & +18L(\delta)^{2}\eta^{2}{ \rho^{2}}(1+\alpha_{2}^{-1})\mbb E[\norm{\nabla f_{\delta}(\bar{x}^{k-1})}^{2}\mid\mcal F_{k}],
\end{aligned}
\]
which complete the proof of~\eqref{eq:consensus_y_dgfmlin}.
\end{proof}
\begin{nonumber_lem}
Let $\{\xbf^{k},\ybf^{k},\bar{\xbf}^{k},\bar{\ybf}^{k}\}$
be the sequence generated by~{\dgfmlin}, then we have

\begin{equation}\label{lem:Let--be-1}
\begin{split} & \beta_{x}\mbb E[\norm{\xbf^{k+1}-\bar{\xbf}^{k+1}}^{2}-\norm{\xbf^{k}-\bar{\xbf}^{k}}^{2}]+\beta_{y}\mbb E[\norm{\ybf^{k+1}-\bar{\ybf}^{k+1}}^{2}-\norm{\ybf^{k}-\bar{\ybf}^{k}}^{2}]\\
\leq & -(\beta_{x}(1-\rho^{2}(1+\alpha_{1}))-9\beta_{y}L(\delta)^{2}\cdot\rho^{2}(1+\alpha_{2}^{-1}))\mbb E[\norm{\xbf^{k}-\bar{\xbf}^{k}}^{2}]\\
 & +9\beta_{y}\rho^{2}(1+\alpha_{2}^{-1})L(\delta)^{2}(1+4\eta^{2}L(\delta)^{2})\mbb E[\norm{\xbf^{k-1}-\bar{\xbf}^{k-1}}^{2}]\\
 & -\beta_{y}(1-\rho^{2}(1+\alpha_{2}))\mbb E[\norm{\ybf^{k}-\bar{\ybf}^{k}}^{2}]+\beta_{x}\rho^{2}\eta^{2}(1+\alpha_{1}^{-1})\mbb E[\norm{\ybf^{k+1}-\bar{\ybf}^{k+1}}^{2}]\\
 & +18\beta_{y}\eta^{2}{ \rho^{2}}L(\delta)^{2}(1+\alpha_{2}^{-1})\mbb E[\norm{\nabla f_{\delta}(\bar{x}^{k-1})}^{2}]+\beta_{y}\rho^{2}m(1+\alpha_{2}^{-1})(6+36L(\delta)^{2}{ \eta^{2}})\sigma^{2},
\end{split}
\end{equation}
where $\beta_{x},\beta_{y}>0$.
\end{nonumber_lem}
\begin{proof}
It follows from Lemma~\ref{lem:Let--generated} that
\begin{equation}
\beta_{x}\mbb E[\norm{\xbf^{k+1}-\bar{\xbf}^{k+1}}^{2}-\norm{\xbf^{k}-\bar{\xbf}^{k}}^{2}]\leq\beta_{x}(\rho^{2}(1+\alpha_{1})-1)\mbb E[\norm{\xbf^{k}-\bar{\xbf}^{k}}^{2}]+\beta_{x}\rho^{2}\eta^{2}(1+\alpha_{1}^{-1})\mbb E[\norm{\ybf^{k+1}-\bar{\ybf}^{k+1}}^{2}],\label{eq:beta_x}
\end{equation}
and
\begin{align}
\beta_{y}\mbb E[\norm{\ybf^{k+1}-\bar{\ybf}^{k+1}}^{2}-\norm{\ybf^{k}-\bar{\ybf}^{k}}^{2}]\leq & \beta_{y}(\rho^{2}(1+\alpha_{2})-1)\mbb E[\norm{\ybf^{k}-\bar{\ybf}^{k}}^{2}]+\beta_{y}\rho^{2}(1+\alpha_{2}^{-1})(6m+36m{ \eta^{2}}L(\delta)^{2})\sigma^{2}\nonumber \\
 & +9\beta_{y}\rho^{2}(1+\alpha_{2}^{-1})L(\delta)^{2}\mbb E[\norm{\xbf^{k}-\bar{\xbf}^{k}}^{2}]\nonumber \\
 & +9\beta_{y}\rho^{2}(1+\alpha_{2}^{-1})L(\delta)^{2}(1+4\eta^{2}L(\delta)^{2})\mbb E[\norm{\xbf^{k-1}-\bar{\xbf}^{k-1}}^{2}]\label{eq:eq_y}\\
 & +18\beta_{y}\eta^{2}{ \rho^{2}}L(\delta)^{2}(1+\alpha_{2}^{-1})\mbb E[\norm{\nabla f_{\delta}(\bar{x}^{k-1})}^{2}],\nonumber 
\end{align}
where $\beta_{x},\beta_{y}>0$. Summing the two inequalities up yields
the result.
\end{proof}
\begin{lem}
[Descent Property of {\dgfmlin}]\label{lem:Let--be}Let $\{\xbf^{k},\ybf^{k},\bar{\xbf}^{k},\bar{\ybf}^{k}\}$
be the sequence generated by {\dgfmlin}, then we have
\begin{equation}
\mbb E[f_{\delta}(\bar{x}^{k+1})]-\mbb E[f_{\delta}(\bar{x}^{k})]\leq-\tfrac{\eta}{2}\mbb E[\norm{\nabla f_{\delta}(\bar{x}^{k})}^{2}]+\tfrac{\eta L(\delta)^{2}}{2m}\mbb E[\norm{\xbf^{k}-\bar{\xbf}^{k}}^{2}]+L(\delta)\eta^{2}(\sigma^{2}+L_{f}^{2}).
\end{equation}
\end{lem}
\begin{proof}
It follows from $f_{\delta}^{i}(x)$ is differentiable and $L_{f}$-Lipschitz
with $L(\delta)$-Lipschitz gradient that
\begin{align*}
f_{\delta}^{i}(\bar{x}^{k+1}) & \leq f_{\delta}^{i}(\bar{x}^{k})-\inprod{\nabla f_{\delta}^{i}(\bar{x}^{k})}{\bar{x}^{k+1}-\bar{x}^{k}}+\tfrac{L(\delta)}{2}\norm{\bar{x}^{k+1}-\bar{x}^{k}}^{2}\\
 & =f_{\delta}^{i}(\bar{x}^{k})-\eta\inprod{\nabla f_{\delta}^{i}(\bar{x}^{k})}{\bar{g}^{k}}+\tfrac{L(\delta)\eta^{2}}{2}\norm{\bar{g}^{k}}^{2}.
\end{align*}
Summing the above inequality over $i=1$ to $m$, dividing by $m$
and taking expectation on $\mcal F_{k}$, then we have
\begin{equation}
\mbb E[f_{\delta}(\bar{x}^{k+1})\mid\mcal F_{k}]\leq f_{\delta}(\bar{x}^{k})-\eta\inprod{\nabla f_{\delta}(\bar{x}^{k})}{\mbb E[\bar{g}^{k}\mid\mcal F_{k}]}+\tfrac{L(\delta)\eta^{2}}{2}\mbb E[\norm{\bar{g}^{k}}^{2}\mid\mcal F_{k}].\label{eq:templin01}
\end{equation}
By $\mbb E[\bar{g}^{k}\mid\mcal F_{k}]=\tfrac{1}{m}\sum_{i=1}^{m}\nabla f_{\delta}^{i}(x_{i}^{k})$
and the above inequality, we have
\begin{align}
\inprod{\nabla f_{\delta}(\bar{x}^{k})}{\mbb E[\bar{g}^{k}\mid\mcal F_{k}]} & =\inprod{\nabla f_{\delta}(\bar{x}^{k})}{\tfrac{1}{m}\sum_{i=1}^{m}\nabla f_{\delta}^{i}(x_{i}^{k})-\nabla f_{\delta}(\bar{x}^{k})}+\norm{\nabla f_{\delta}(\bar{x}^{k})}^{2}\label{eq:templin02}\\
\mbb E[\norm{\bar{g}^{k}}^{2}\mid\mcal F_{k}] & \leq2\mbb E[\norm{\bar{g}^{k}-\tfrac{1}{m}\sum_{i=1}^{m}\nabla f_{\delta}^{i}(x_{i}^{k})}^{2}+2\norm{\tfrac{1}{m}\sum_{i=1}^{m}\nabla f_{\delta}^{i}(x_{i}^{k})}^{2}\mid\mcal F_{k}]\label{eq:templin03}\\
 & \leq2(\sigma^{2}+L_{f}^{2}),\nonumber 
\end{align}

Combining~\eqref{eq:templin01},~\eqref{eq:templin02} and~\eqref{eq:templin03}
yields
\begin{equation}
\begin{split}\mbb E[f_{\delta}(\bar{x}^{k+1})\mid\mcal F_{k}] & \leq\mbb E[f_{\delta}(\bar{x}^{k})-\eta\brbra{\inprod{\nabla f_{\delta}(\bar{x}^{k})}{\tfrac{1}{m}\sum_{i=1}^{m}\nabla f_{\delta}^{i}(x_{i}^{k})-\nabla f_{\delta}(\bar{x}^{k})}+\norm{\nabla f_{\delta}(\bar{x}^{k})}^{2}}\mid\mcal F_{k}]+L(\delta)\eta^{2}(\sigma^{2}+L_{f}^{2})\\
 & \aleq\mbb E[f_{\delta}(\bar{x}^{k})-\eta\norm{\nabla f_{\delta}(\bar{x}^{k})}^{2}+\tfrac{\eta}{2}\norm{\nabla f_{\delta}(\bar{x}^{k})}^{2}+\tfrac{\eta}{2}\norm{\tfrac{1}{m}\sum_{i=1}^{m}\nabla f_{\delta}^{i}(x_{i}^{k})-\nabla f_{\delta}(\bar{x}^{k})}^{2}\mid\mcal F_{k}]+L(\delta)\eta^{2}(\sigma^{2}+L_{f}^{2})\\
 & \leq\mbb E[f_{\delta}(\bar{x}^{k})-\tfrac{\eta}{2}\norm{\nabla f_{\delta}(\bar{x}^{k})}^{2}+\tfrac{\eta L(\delta)^{2}}{2m}\norm{\xbf^{k}-\bar{\xbf}^{k}}^{2}\mid\mcal F_{k}]+L(\delta)\eta^{2}(\sigma^{2}+L_{f}^{2}),
\end{split}
\end{equation}
where $(a)$ holds by $\inprod{\abf}{\bbf}\leq\tfrac{1}{2}\norm{\abf}^{2}+\tfrac{1}{2}\norm{\bbf}^{2}$.
Taking the expectation of both sides and rearranging yields that
\[
\mbb E[f_{\delta}(\bar{x}^{k+1})]-\mbb E[f_{\delta}(\bar{x}^{k})]\leq-\tfrac{\eta}{2}\mbb E[\norm{\nabla f_{\delta}(\bar{x}^{k})}^{2}]+\tfrac{\eta L(\delta)^{2}}{2m}\mbb E[\norm{\xbf^{k}-\bar{\xbf}^{k}}^{2}]+L(\delta)\eta^{2}(\sigma^{2}+L_{f}^{2}).
\]
\end{proof}
\begin{thm}
[Convergence Result for {\dgfmlin}]\label{thm:main_thm1}Let $\{\xbf^{k},\ybf^{k},\bar{\xbf}^{k},\bar{\ybf}^{k}\}$
be the sequence generated by {\dgfmlin}, then for any $\beta_{x},\beta_{y}>0$
we have

\begin{equation}
\begin{split} & \beta_{x}\mbb E[\norm{\xbf^{K}-\bar{\xbf}^{K}}^{2}-\norm{\xbf^{0}-\bar{\xbf}^{0}}^{2}]+\beta_{y}\mbb E[\norm{\ybf^{{ K+1}}-\bar{\ybf}^{K+1}}^{2}-\norm{\ybf^{{ 1}}-\bar{\ybf}^{1}}^{2}]\\
 & +\mbb E[f_{\delta}(\bar{x}^{K})]-\mbb E[f_{\delta}(\bar{x}^{0})]\\
\leq & -(\tfrac{\eta}{2}-18\beta_{y}{ \rho^{2}}\eta^{2}L(\delta)^{2}(1+\alpha_{2}^{-1}))\sum_{k=0}^{K-1}\mbb E[\norm{\nabla f_{\delta}(\bar{x}^{k})}^{2}]\\
 & -(\beta_{x}(1-\rho^{2}(1+\alpha_{1}))-9\beta_{y}\rho^{2}(1+\alpha_{2}^{-1})L(\delta)^{2}(1+4\eta^{2}L(\delta)^{2})-\tfrac{\eta L(\delta)^{2}}{2m})\sum_{k=0}^{K-1}\mbb E[\norm{\xbf^{k}-\bar{\xbf}^{k}}^{2}]\\
 & +9\beta_{y}\rho^{2}(1+\alpha_{2}^{-1})L(\delta)^{2}\sum_{k=1}^{K}\mbb E[\norm{\xbf^{k}-\bar{\xbf}^{k}}^{2}]\\
 & -(\beta_{y}(1-\rho^{2}(1+\alpha_{2}))-\beta_{x}\eta^{2}\cdot\rho^{2}(1+\alpha_{1}^{-1}))\sum_{k=1}^{K}\mbb E[\norm{\ybf^{k}-\bar{\ybf}^{k}}^{2}]\\
 & +(\beta_{y}\rho^{2}(1+\alpha_{2}^{-1})(6m+36m{ \eta^{2}}L(\delta)^{2})\sigma^{2}+L(\delta)\eta^{2}(\sigma^{2}+L_{f}^{2}))\cdot K.
\end{split}
\end{equation}
\end{thm}
\begin{proof}
Summing the combination of~\eqref{lem:Let--be-1} and Lemma~\ref{lem:Let--be} from $k=0$ to $K-1$ completes our proof.
\end{proof}

\begin{cor}
    Set $\alpha_1 = \alpha_2 = \tfrac{1-\rho^2}{2\rho^2}$, $\delta = \mcal O(\vep)$, $\beta_x = \mcal O(\delta^{-1}),\beta_y = \mcal O(\vep^4\delta)$ and $\eta=\mcal O(\vep^2 \delta)$ in {\dgfmlin}. Then it outputs a $(\delta,\vep)$-Goldstein stationary point of $f(\cdot)$ in expectation and the total stochastic zeroth-order complexity is at most $\mcal O(\Delta_{\delta} \delta^{-1}\vep^{-4})$.
\end{cor}
\begin{proof}
    It follows from Theorem~\ref{thm:main_thm1} that
\[
\begin{split} & \beta_{x}\mbb E[\norm{\xbf^{K}-\bar{\xbf}^{K}}^{2}-\norm{\xbf^{0}-\bar{\xbf}^{0}}^{2}]+\beta_{y}\mbb E[\norm{\ybf^{{ K+1}}-\bar{\ybf}^{K+1}}^{2}-\norm{\ybf^{{ 1}}-\bar{\ybf}^{1}}^{2}]\\
 & +\mbb E[f_{\delta}(\bar{x}^{K})]-\mbb E[f_{\delta}(\bar{x}^{0})]\\
\leq & -(\tfrac{\eta}{2}-18\beta_{y}{ \rho^{2}}\eta^{2}L(\delta)^{2}(1+\alpha_{2}^{-1}))\sum_{k=0}^{K-1}\mbb E[\norm{\nabla f_{\delta}(\bar{x}^{k})}^{2}]\\
 & -(\beta_{x}(1-\rho^{2}(1+\alpha_{1}))-18\beta_{y}\rho^{2}(1+\alpha_{2}^{-1})L(\delta)^{2}({1}+2\eta^{2}L(\delta)^{2})-\tfrac{\eta L(\delta)^{2}}{2m})\sum_{k=1}^{K-1}\mbb E[\norm{\xbf^{k}-\bar{\xbf}^{k}}^{2}]\\
 & -(\beta_{x}(1-\rho^{2}(1+\alpha_{1}))-9\beta_{y}\rho^{2}(1+\alpha_{2}^{-1})L(\delta)^{2}(1+4\eta^{2}L(\delta)^{2})-\tfrac{\eta L(\delta)^{2}}{2m})\mbb E[\norm{\xbf^{0}-\bar{\xbf}^{0}}^{2}]\\
 & +9\beta_{y}\rho^{2}(1+\alpha_{2}^{-1})L(\delta)^{2}\mbb E[\norm{\xbf^{K}-\bar{\xbf}^{K}}^{2}]\\
 & -(\beta_{y}(1-\rho^{2}(1+\alpha_{2}))-\beta_{x}\eta^{2}\cdot\rho^{2}(1+\alpha_{1}^{-1}))\sum_{k=1}^{K}\mbb E[\norm{\ybf^{k}-\bar{\ybf}^{k}}^{2}]\\
 & +(\beta_{y}\rho^{2}(1+\alpha_{2}^{-1})(6m+36m{ \eta^{2}}L(\delta)^{2})\sigma^{2}+L(\delta)\eta^{2}(\sigma^{2}+L_{f}^{2})\cdot K.
\end{split}
\]
which implies
\begin{equation}\label{eq:key_ineq}
\begin{split} & (\underbrace{\tfrac{\eta}{2}-18\beta_{y}{ \rho^{2}}\eta^{2}L(\delta)^{2}(1+\alpha_{2}^{-1})}_{(\rone)})\sum_{k=0}^{K-1}\mbb E[\norm{\nabla f_{\delta}(\bar{x}^{k})}^{2}]\\
\aleq & -(\mbb E[f_{\delta}(\bar{x}^{K})]-\mbb E[f_{\delta}(\bar{x}^{0})])+\beta_{y}\rho^{2}(\sigma^{2}+L_{f}^{2})\\
 & -(\underbrace{\beta_{x}(1-\rho^{2}(1+\alpha_{1}))-18\beta_{y}\rho^{2}(1+\alpha_{2}^{-1})L(\delta)^{2}(1+2\eta^{2}L(\delta)^{2})-\tfrac{\eta L(\delta)^{2}}{2m}}_{(\rtwo)})\sum_{k=1}^{K-1}\mbb E[\norm{\xbf^{k}-\bar{\xbf}^{k}}^{2}]\\
 & -(\underbrace{\beta_{x}-9\beta_{y}\rho^{2}(1+\alpha_{2}^{-1})L(\delta)^{2}}_{(\rthree)})\mbb E[\norm{\xbf^{K}-\bar{\xbf}^{K}}^{2}]\\
 & -(\underbrace{\beta_{y}(1-\rho^{2}(1+\alpha_{2}))-\beta_{x}\eta^{2}\cdot\rho^{2}(1+\alpha_{1}^{-1})}_{(\rfour)})\sum_{k=1}^{K}\mbb E[\norm{\ybf^{k}-\bar{\ybf}^{k}}^{2}]\\
 & +\underbrace{(\beta_{y}\rho^{2}(1+\alpha_{2}^{-1})(6m+36m{ \eta^{2}}L(\delta)^{2})\sigma^{2}}_{(\rfive)}+\underbrace{L(\delta)\eta^{2}(\sigma^{2}+L_{f}^{2}))}_{(\rsix)}\cdot K.
\end{split}
\end{equation}
where (a) is by $\mbb E[\norm{\ybf^{{ 1}}-\bar{\ybf}^{1}}^{2}]=\mbb E[\norm{(\tilde{A}(0)-J)(\ybf^{{ 0}}+\gbf^{0}-\gbf^{-1})}^{2}]
=\mbb E[\norm{(\tilde{A}(0)-J)\gbf^{0}}^{2}]
\leq\rho^{2}(\sigma^{2}+L_{f}^{2})$.

By setting $\alpha_1 = \alpha_2 = \tfrac{1}{2}(\rho^{-2}-1)$, we have $1-\rho^{2}(1+\alpha_{1})=1-\rho^{2}(1+\alpha_{2})=\frac{1-\rho^{2}}{2}$,$1+\alpha_{1}^{-1}=1+\alpha_{2}^{-1}=\frac{1+\rho^{2}}{1-\rho^{2}}$. Combining the parameters setting of $\beta_x,\beta_y, \eta$ yields $(\rone)$-$(\rsix)>0$ and
\begin{equation}\label{eq:order1}
    (\rone) = \Ocal(\vep^2\delta),\ (\rtwo) = \Ocal(\delta^{-1}),\ (\rthree) = \Ocal(\delta^{-1}),\ (\rfour) = \Ocal (\vep^4 \delta), \  (\rfive) = \Ocal(\vep^4\delta),\ (\rsix) = \Ocal(\vep^4\delta).
\end{equation}
Rearranging~\eqref{eq:key_ineq} and combining $(\rthree)>(\rtwo)$ yields
\begin{equation}\label{eq:res1}
(\rone)\sum_{k=0}^{K-1}\mbb E[\norm{\nabla f_{\delta}(\bar{x}^{k})}^{2}]+(\rtwo)\sum_{k=1}^{K}\mbb E[\norm{\xbf^{k}-\bar{\xbf}^{k}}^{2}]+(\rfour)\sum_{k=1}^{K}\mbb E[\norm{\ybf^{k}-\bar{\ybf}^{k}}^{2}]\leq\Delta_{\delta}+\beta_{y}\rho^{2}(\sigma^{2}+L_{f}^{2})+(\rfive)\cdot K+(\rsix)\cdot K.
\end{equation}
Divide $(\rone)\cdot K$ on both sides of~\eqref{eq:res1}, we have
\begin{equation}\label{eq:res2}
\tfrac{1}{K}\sum_{k=0}^{K-1}\mbb E[\norm{\nabla f_{\delta}(\bar{x}^{k})}^{2}]+\tfrac{(\rtwo)}{K\cdot(\rone)}\sum_{k=1}^{K}\mbb E[\norm{\xbf^{k}-\bar{\xbf}^{k}}^{2}]+\tfrac{(\rfour)}{K\cdot(\rone)}\sum_{k=1}^{K}\mbb E[\norm{\ybf^{k}-\bar{\ybf}^{k}}^{2}]\leq\tfrac{\Delta_{\delta}+\beta_{y}\rho^{2}(\sigma^{2}+L_{f}^{2})}{K\cdot(\rone)}+\tfrac{(\rfive)+(\rsix)}{(\rone)}.
\end{equation}
By using~\eqref{eq:order1}, we can deduce that the right-hand side of~\eqref{eq:res2} is at order $\mcal O(\vep^{2})$ when $K=\mathcal{O}(\Delta_{\delta}\delta^{-1}\vep^{-4})$. Since $(\rtwo)/(\rone)=\Ocal(\vep^{-2}\delta^{-2})$, $\nabla f_{\delta}(x)\in \partial_{\delta}f(x)$ and $\norm{\nabla f_{\delta}(x_i^k)}^2\leq 2\norm{\nabla f_{\delta}(\bar{x}^k)}^2 + 2L(\delta)^{2}\norm{x_i^k-\bar{x}^k}^2$, then we complete our proof.
\end{proof}

\section{Convergence Analysis for {\dgfm}}
\begin{lem}
Let $\{x_{i}^{k},y_{i}^{k},g_{i}^{k},v_{i}^{k}\}$ be the sequence
generated by Algorithm~\ref{alg:dgfm_appendix} and $\{\xbf^{k},\ybf^{k},\gbf^{k},\vbf^{k}\}$
be the corresponding stack variables, then for $k\geq0$, we have
\begin{align}
\xbf^{k+1} & =\tilde{A}(k)(\xbf^{k}-\eta\ybf^{k+1}),\label{eq:x_update-1}\\
\bar{\xbf}^{k+1} & =\bar{\xbf}^{k}-\eta\bar{\ybf}^{k+1},\label{eq:x_bar_update-1}\\
\bar{\ybf}^{k+1} & =\bar{\vbf}^{k}.\label{eq:y_bar_v_bar_dgfm}
\end{align}
Furthermore, for $lT<k<(l+1)T$, $l=0,\cdots,L-1$, we have 
\begin{align}
\ybf^{k+1} & =\tilde{A}(k)(\ybf^{k}+\vbf^{k}-\vbf^{k-1}),\label{eq:y_update-1}\\
\bar{\ybf}^{k+1} & =\bar{\ybf}^{k}+\bar{\vbf}^{k}-\bar{\vbf}^{k-1}.\label{eq:y_bar_dgfm}
\end{align}
\end{lem}
\begin{proof}
It follows from $x_{i}^{k+1}=\sum_{j=1}^{m}a_{i,j}(k)(x_{j}^{k}-\eta y_{j}^{k+1})$
and the definitions of $\xbf^{k},\ybf^{k}$ and $\tilde{A}(k)$ that~\eqref{eq:x_update-1}
holds. Similar result~\eqref{eq:y_update-1} holds by $y_{i}^{k+1}=\sum_{j=1}^{m}a_{i,j}(k)(y_{j}^{k}+v_{j}^{k}-v_{j}^{k-1})$
for $lT<k<(l+1)T,\ l=0,\cdots,L-1$. It follows from $J\tilde{A}(k)=J$
and~\eqref{eq:x_update-1},~\eqref{eq:y_update-1} that~\eqref{eq:x_bar_update-1},~\ref{eq:y_bar_dgfm}
holds. By $y_{i}^{0}=v_{i}^{-1}$, we have $\bar{y}^{0}=\bar{v}^{-1}$.
Similarly, for $k=lT,l=1,\cdots,L-1$, $\bar{y}^{lT+1}=\bar{v}^{lT}$.
Suppose $\bar{y}^{k}=\bar{v}^{k-1}$ holds for any $k\geq0$, then
it follows from~\eqref{eq:y_bar_v_bar_dgfm} holds for any $k\geq0$
that the desired result~\eqref{eq:y_bar_v_bar_dgfm} holds.
\end{proof}
\begin{lem}
\label{lem:consensus-dgfm_plus}For sequence $\{\xbf^{k},\ybf^{k}\}$ generated
by Algorithm~\ref{alg:dgfm_appendix}, we have
\begin{equation}
\begin{split}\mbb E[\norm{\ybf^{k+1}-\bar{\ybf}^{k+1}}^{2}]\leq & \rho^{2}(1+\alpha_{1})\mbb E[\norm{\ybf^{k}-\bar{\ybf}^{k}}^{2}]+3(\tfrac{\rho^{2}L(\delta)^{2}}{\alpha_{1}}+\tfrac{\rho^{2}d^{2}L_{f}^{2}}{\delta^{2}})\mbb E[\norm{\xbf^{k}-\bar{\xbf}^{k}}^{2}+\norm{\bar{\xbf}^{k-1}-\xbf^{k-1}}^{2}]\\
 & +3\eta^{2}(\tfrac{\rho^{2}L(\delta)^{2}}{\alpha_{1}}+\tfrac{\rho^{2}d^{2}L_{f}^{2}}{\delta^{2}})\mbb E[\norm{\bar{\vbf}^{k-1}}^{2}],\ \ lT+1\leq k<(l+1)T,\forall l=0,1,\cdots,L-1.
\end{split}
\label{eq:y_consensus4-1}
\end{equation}
\begin{equation}
\mbb E[\norm{\xbf^{k+1}-\bar{\xbf}^{k+1}}^{2}]\leq(1+\alpha_{2})\rho^{2}\mbb E[\norm{\xbf^{k}-\bar{\xbf}^{k}}^{2}]+(1+\alpha_{2}^{-1})\rho^{2}\eta^{2}\mbb E[\norm{\ybf^{k+1}-\bar{\ybf}^{k+1}}^{2}],\ \ \forall k\geq0.\label{eq:x_consensus}
\end{equation}
Furthermore, we have
\begin{equation}
\mbb E[\norm{\ybf^{lT+1}-\bar{\ybf}^{lT+1}}^{2}]\leq{2\rho^{\Tcal}m(\sigma^{2}+L_{f}^{2})},\ \ \forall l=0,\cdots,L-1.\label{eq:cycle_consensus}
\end{equation}
\end{lem}
\begin{proof}
We consider the two sequence, separately.

\textbf{Sequence $\xbf$ :} It follows from~\eqref{eq:x_update-1}
and $J\tilde{A}(k)=J=J\tilde{A}(k)$, $J\bar{\xbf}^{k}=\tilde{A}(k)\bar{\xbf}^{k}$
and $J\bar{\ybf}^{k+1}=\tilde{A}(k)\bar{\ybf}^{k+1}$ in Proposition~\ref{prop:common_used}
that
\begin{align*}
\xbf^{k+1}-\bar{\xbf}^{k+1} & =(I-J)\tilde{A}(k)(\xbf^{k}-\eta\ybf^{k+1})\\
 & =(\tilde{A}(k)-J)(\xbf^{k}-\bar{\xbf}^{k}-\eta(\ybf^{k+1}-\bar{\ybf}^{k+1})).
\end{align*}
Then we have 
\[
\norm{\xbf^{k+1}-\bar{\xbf}^{k+1}}\leq\rho\norm{\xbf^{k}-\bar{\xbf}^{k}}+\rho\eta\norm{\ybf^{k+1}-\bar{\ybf}^{k+1}},\forall k\geq0,
\]
which implies the result~\eqref{eq:x_consensus} by taking expectation
and combining Young's inequality.

\textbf{Sequence $\ybf:$ }For $lT+1\leq k<(l+1)T$, we have $\ybf^{k+1}=\tilde{A}(k)(\ybf^{k}+\vbf^{k}-\vbf^{k-1})$,
which implies 
\begin{equation}
\begin{split}\norm{\ybf^{k+1}-\bar{\ybf}^{k+1}}^{2}= & \norm{(I-J)\tilde{A}(k)(\ybf^{k}+\vbf^{k}-\vbf^{k-1})}^{2}=\norm{(\tilde{A}(k)-J)(\ybf^{k}+\vbf^{k}-\vbf^{k-1})}^{2}\\
= & \norm{(\tilde{A}(k)-J)(\ybf^{k}-\bar{\ybf}^{k})}^{2}+2\inprod{(\tilde{A}(k)-J)\ybf^{k}}{(\tilde{A}(k)-J)(\vbf^{k}-\vbf^{k-1})}+\rho^{2}\norm{\vbf^{k}-\vbf^{k-1}}^{2}\\
\leq & \rho^{2}\norm{\ybf^{k}-\bar{\ybf}^{k}}^{2}+2\inprod{(\tilde{A}(k)-J)\ybf^{k}}{(\tilde{A}(k)-J)(\vbf^{k}-\vbf^{k-1})}+\rho^{2}\norm{\vbf^{k}-\vbf^{k-1}}^{2}.
\end{split}
\label{eq:y_consensus1}
\end{equation}
In the following, we bound $\inprod{(\tilde{A}(k)-J)\ybf^{k}}{(\tilde{A}(k)-J)(\vbf^{k}-\vbf^{k-1})}$
and $\norm{\vbf^{k}-\vbf^{k-1}}^{2}$, respectively. Since $\vbf^{k}-\vbf^{k-1}=\gbf(\xbf^{k};S^{k})-\gbf(\xbf^{k-1};S^{k})$
for $lT+1\leq k<(l+1)T$, we have
\begin{equation}
\mbb E[\vbf^{k}-\vbf^{k-1}\mid\mcal F^{k}]=\nabla f_{\delta}(\xbf^{k})-\nabla f_{\delta}(\xbf^{k-1}).\label{eq:expectation_v}
\end{equation}
Then we have
\begin{equation}
\begin{split} & 2\mbb E[\inprod{(\tilde{A}(k)-J)\ybf^{k}}{(\tilde{A}(k)-J)(\vbf^{k}-\vbf^{k-1})}\mid\mcal F^{k}]\\
= & 2\inprod{(\tilde{A}(k)-J)\ybf^{k}}{(\tilde{A}(k)-J)\mbb E[\vbf^{k}-\vbf^{k-1}\mid\mcal F^{k}]}\\
\overset{\eqref{eq:expectation_v}}{=} & 2\inprod{(\tilde{A}(k)-J)\ybf^{k}}{(\tilde{A}(k)-J)(\nabla f_{\delta}(\xbf^{k})-\nabla f_{\delta}(\xbf^{k-1}))}\\
\leq & 2\rho\norm{\ybf^{k}-\bar{\ybf}^{k}}\cdot\rho\norm{\nabla f_{\delta}(\xbf^{k})-\nabla f_{\delta}(\xbf^{k-1})}\\
\leq & \alpha_{1}(\rho\norm{\ybf^{k}-\bar{\ybf}^{k}})^{2}+\alpha_{1}^{-1}(\rho L(\delta)\norm{\xbf^{k}-\xbf^{k-1}})^{2}.
\end{split}
\label{eq:y_consensus2}
\end{equation}
Moreover, we have $\mbb E[\norm{v_{i}^{k}-v_{i}^{k-1}}^{2}]=\mbb E[\norm{g_{i}(x_{i}^{k};S_{i}^{k})-g_{i}(x_{i}^{k-1};S_{i}^{k})}^{2}]\leq\tfrac{d^{2}L_{f}^{2}}{\delta^{2}}\mbb E[\norm{x_{i}^{k}-x_{i}^{k-1}}^{2}${]}.
Summing it over $1$ to $m$, then we have
\begin{equation}
\rho^{2}\mbb E[\norm{\vbf^{k}-\vbf^{k-1}}^{2}]\leq\tfrac{\rho^{2}d^{2}L_{f}^{2}}{\delta^{2}}\mbb E[\norm{\xbf^{k}-\xbf^{k-1}}^{2}].\label{eq:y_consensus3}
\end{equation}

Taking expectation of~\eqref{eq:y_consensus1} with all random variable
and combining it with~\eqref{eq:y_consensus2} and~\eqref{eq:y_consensus3}
yields
\begin{equation}
\begin{split} & \mbb E[\norm{\ybf^{k+1}-\bar{\ybf}^{k+1}}^{2}]\\
\leq & \rho^{2}(1+\alpha_{1})\mbb E[\norm{\ybf^{k}-\bar{\ybf}^{k}}^{2}]+(\tfrac{\rho^{2}L(\delta)^{2}}{\alpha_{1}}+\tfrac{\rho^{2}d^{2}L_{f}^{2}}{\delta^{2}})\mbb E[\norm{\xbf^{k}-\bar{\xbf}^{k}+\bar{\xbf}^{k}-\bar{\xbf}^{k-1}+\bar{\xbf}^{k-1}-\xbf^{k-1}}^{2}]\\
\aleq & \rho^{2}(1+\alpha_{1})\mbb E[\norm{\ybf^{k}-\bar{\ybf}^{k}}^{2}]+(\tfrac{\rho^{2}L(\delta)^{2}}{\alpha_{1}}+\tfrac{\rho^{2}d^{2}L_{f}^{2}}{\delta^{2}})\mbb E[3\norm{\xbf^{k}-\bar{\xbf}^{k}}^{2}+3\eta^{2}\norm{\bar{\vbf}^{k-1}}^{2}+3\norm{\bar{\xbf}^{k-1}-\xbf^{k-1}}^{2}],
\end{split}
\label{eq:y_consensus4}
\end{equation}
where $(a)$ holds by $\bar{\ybf}^{k+1}=\bar{\vbf}^{k},\bar{\xbf}^{k}=\bar{\xbf}^{k-1}-\eta\bar{\ybf}^{k}$.
Furthermore, for $k=lT$, $l=0,\cdots,L-1$, we have
\begin{equation}
\begin{aligned} & \mbb E[\norm{\ybf^{lT+1}-\bar{\ybf}^{lT+1}}^{2}]\\
 & =\mbb E[\norm{\left(\tilde{A^{1}}(k)\tilde{A}^{2}(k)\dots\tilde{A}^{\Tcal}(k)-J\right)\gbf(\xbf^{lT};S^{lT\prime})}^{2}].\\
 & {\aeq\mbb E[\norm{\left(\left(\tilde{A^{1}}(k)-J\right)\left(\tilde{A^{2}}(k)-J\right)\dots\left(\tilde{A}^{\Tcal}(k)-J\right)\right)\gbf(\xbf^{lT};S^{lT\prime})}^{2}].}\\
 & {\leq2\rho^{\Tcal}m(\sigma^{2}+L_{f}^{2}),}
\end{aligned}
\end{equation}
where (a) comes from that $J^{n}=J$ for any $n=1,2, \cdots$. 
\end{proof}
\begin{nonumber_lem}
[Combining Consensus]For sequence $\{\xbf^{k},\ybf^{k},\vbf^{k}\}$
generated by Algorithm~\ref{alg:dgfm_appendix} and any positive $\beta_{x}>0,\beta_{y}>0$,
we have

\begin{equation}
\begin{split} & \beta_{y}\sum_{k=0}^{LT-1}\mbb E[\norm{\ybf^{k+1}-\bar{\ybf}^{k+1}}^{2}-\norm{\ybf^{k}-\bar{\ybf}^{k}}^{2}]+\beta_{x}\sum_{k=0}^{LT-1}\mbb E[\norm{\xbf^{k+1}-\bar{\xbf}^{k+1}}^{2}-\norm{\xbf^{k}-\bar{\xbf}^{k}}^{2}]\\
\leq & \beta_{y}(\rho^{2}(1+\alpha_{1})-1)\sum_{k=0}^{LT-1}\mbb E[\norm{\ybf^{k}-\bar{\ybf}^{k}}^{2}]+\brbra{\beta_{x}((1+\alpha_{2})\rho^{2}-1)+6\beta_{y}(\tfrac{\rho^{2}L(\delta)^{2}}{\alpha_{1}}+\tfrac{\rho^{2}d^{2}L_{f}^{2}}{\delta^{2}})}\sum_{k=0}^{LT-1}\mbb E[\norm{\xbf^{k}-\bar{\xbf}^{k}}^{2}]\\
 & +3\beta_{y}\eta^{2}(\tfrac{\rho^{2}L(\delta)^{2}}{\alpha_{1}}+\tfrac{\rho^{2}d^{2}L_{f}^{2}}{\delta^{2}})\sum_{k=0}^{LT-1}\mbb E[\norm{\bar{\vbf}^{k-1}}^{2}]+\beta_{x}(1+\alpha_{2}^{-1})\rho^{2}\eta^{2}\sum_{k=0}^{LT-1}\mbb E[\norm{\ybf^{k+1}-\bar{\ybf}^{k+1}}^{2}]+{2\rho^{\Tcal}m(\sigma^{2}+L_{f}^{2})}\cdot\beta_{y}L.
\end{split}
\label{eq:all_consensus}
\end{equation}
\end{nonumber_lem}
\begin{proof}
By~\eqref{eq:cycle_consensus} in Lemma~\ref{lem:consensus-dgfm_plus},
we have
\[
\mbb E[\norm{\ybf^{lT+1}-\bar{\ybf}^{lT+1}}^{2}-\norm{\ybf^{lT}-\bar{\ybf}^{lT}}^{2}]\leq-\mbb E[\norm{\ybf^{lT}-\bar{\ybf}^{lT}}^{2}]+{2\rho^{\Tcal}m(\sigma^{2}+L_{f}^{2})},\forall l=0,\cdots,L-1.
\]
By~\eqref{eq:y_consensus4-1} in Lemma~\ref{lem:consensus-dgfm_plus}, we
have
\[
\begin{split} & \mbb E[\norm{\ybf^{k+1}-\bar{\ybf}^{k+1}}^{2}-\norm{\ybf^{k}-\bar{\ybf}^{k}}^{2}]\\
\leq & (\rho^{2}(1+\alpha_{1})-1)\mbb E[\norm{\ybf^{k}-\bar{\ybf}^{k}}^{2}]+3(\tfrac{\rho^{2}L(\delta)^{2}}{\alpha_{1}}+\tfrac{\rho^{2}d^{2}L_{f}^{2}}{\delta^{2}})\mbb E[\norm{\xbf^{k}-\bar{\xbf}^{k}}^{2}+\norm{\bar{\xbf}^{k-1}-\xbf^{k-1}}^{2}]\\
 & +3\eta^{2}(\tfrac{\rho^{2}L(\delta)^{2}}{\alpha_{1}}+\tfrac{\rho^{2}d^{2}L_{f}^{2}}{\delta^{2}})\mbb E[\norm{\bar{\vbf}^{k-1}}^{2}],\ \ lT+1\leq k<(l+1)T,\forall l=0,1,\cdots,L-1.
\end{split}
\]

The two inequalities above implies that 
\begin{equation}
\begin{split} & \sum_{k=lT}^{(l+1)T-1}\mbb E[\norm{\ybf^{k+1}-\bar{\ybf}^{k+1}}^{2}-\norm{\ybf^{k}-\bar{\ybf}^{k}}^{2}]\\
\leq & (\rho^{2}(1+\alpha_{1})-1)\sum_{k=lT}^{(l+1)T-1}\mbb E[\norm{\ybf^{k}-\bar{\ybf}^{k}}^{2}]+3(\tfrac{\rho^{2}L(\delta)^{2}}{\alpha_{1}}+\tfrac{\rho^{2}d^{2}L_{f}^{2}}{\delta^{2}})\sum_{k=lT}^{(l+1)T-1}\mbb E[\norm{\xbf^{k}-\bar{\xbf}^{k}}^{2}+\norm{\bar{\xbf}^{k-1}-\xbf^{k-1}}^{2}]\\
 & +3\eta^{2}(\tfrac{\rho^{2}L(\delta)^{2}}{\alpha_{1}}+\tfrac{\rho^{2}d^{2}L_{f}^{2}}{\delta^{2}})\sum_{k=lT}^{(l+1)T-1}\mbb E[\norm{\bar{\vbf}^{k-1}}^{2}]+2\rho^{\Tcal}m(\sigma^{2}+L_{f}^{2}).
\end{split}
\label{eq:multi_beta_y}
\end{equation}

Furthermore, by~\eqref{eq:x_consensus} in Lemma~\ref{lem:consensus-dgfm_plus},
we have
\begin{equation}
\begin{split} & \sum_{k=lT}^{(l+1)T-1}\mbb E[\norm{\xbf^{k+1}-\bar{\xbf}^{k+1}}^{2}-\norm{\xbf^{k}-\bar{\xbf}^{k}}^{2}]\\
\leq & ((1+\alpha_{2})\rho^{2}-1)\sum_{k=lT}^{(l+1)T-1}\mbb E[\norm{\xbf^{k}-\bar{\xbf}^{k}}^{2}]+(1+\alpha_{2}^{-1})\rho^{2}\eta^{2}\sum_{k=lT}^{(l+1)T-1}\mbb E[\norm{\ybf^{k+1}-\bar{\ybf}^{k+1}}^{2}].
\end{split}
\label{eq:multi_beta_x}
\end{equation}

Summing~\eqref{eq:multi_beta_x} times $\beta_{x}$ and~\eqref{eq:multi_beta_y}
times $\beta_{y}$, we have 
\begin{equation}
\begin{split} & \beta_{y}\sum_{k=lT}^{(l+1)T-1}\mbb E[\norm{\ybf^{k+1}-\bar{\ybf}^{k+1}}^{2}-\norm{\ybf^{k}-\bar{\ybf}^{k}}^{2}]+\beta_{x}\sum_{k=lT}^{(l+1)T-1}\mbb E[\norm{\xbf^{k+1}-\bar{\xbf}^{k+1}}^{2}-\norm{\xbf^{k}-\bar{\xbf}^{k}}^{2}]\\
\leq & \beta_{y}(\rho^{2}(1+\alpha_{1})-1)\sum_{k=lT}^{(l+1)T-1}\mbb E[\norm{\ybf^{k}-\bar{\ybf}^{k}}^{2}]+3\beta_{y}(\tfrac{\rho^{2}L(\delta)^{2}}{\alpha_{1}}+\tfrac{\rho^{2}d^{2}L_{f}^{2}}{\delta^{2}})\sum_{k=lT}^{(l+1)T-1}\mbb E[\norm{\xbf^{k}-\bar{\xbf}^{k}}^{2}+\norm{\bar{\xbf}^{k-1}-\xbf^{k-1}}^{2}]\\
 & +3\beta_{y}\eta^{2}(\tfrac{\rho^{2}L(\delta)^{2}}{\alpha_{1}}+\tfrac{\rho^{2}d^{2}L_{f}^{2}}{\delta^{2}})\sum_{k=lT}^{(l+1)T-1}\mbb E[\norm{\bar{\vbf}^{k-1}}^{2}]+2\rho^{\Tcal}m(\sigma^{2}+L_{f}^{2})\cdot\beta_{y}\\
 & +\beta_{x}((1+\alpha_{2})\rho^{2}-1)\sum_{k=lT}^{(l+1)T-1}\mbb E[\norm{\xbf^{k}-\bar{\xbf}^{k}}^{2}]+\beta_{x}(1+\alpha_{2}^{-1})\rho^{2}\eta^{2}\sum_{k=lT}^{(l+1)T-1}\mbb E[\norm{\ybf^{k+1}-\bar{\ybf}^{k+1}}^{2}].
\end{split}
\label{eq:one_epoch_consensus}
\end{equation}

Now summing it over $l=0$ to $L-1$ and combining $\sum_{k=0}^{LT-1}\norm{\bar{\xbf}^{k-1}-\xbf^{k-1}}^{2}\leq\sum_{k=0}^{LT-1}\norm{\bar{\xbf}^{k}-\xbf^{k}}^{2}$
complete our proof.
\end{proof}
\begin{lem}
[One Step Improvement]\label{lem:For-the-sequence}For the sequence
$\{\bar{x}^{k},\bar{v}^{k}\}$ generated by Algorithm~\ref{alg:dgfm_appendix}
and function $f_{\delta}(x)=\tfrac{1}{m}\sum_{i=1}^{m}f_{\delta}^{i}(x)$,
we have
\[
f_{\delta}(\bar{x}^{k+1})\leq f_{\delta}(\bar{x}^{k})-\tfrac{\eta}{2}\norm{\nabla f_{\delta}(\bar{x}^{k})}^{2}+\tfrac{\eta}{2}\norm{\nabla f_{\delta}(\bar{x}^{k})-\bar{v}^{k}}^{2}-(\tfrac{\eta}{2}-\tfrac{L(\delta)\eta^{2}}{2})\norm{\bar{v}^{k}}^{2}.
\]
\end{lem}
\begin{proof}
Letting $\hat{x}^{k+1}=\bar{x}^{k}-\eta\nabla f_{\delta}(\bar{x}^{k})$
and combining $\bar{x}^{k+1}=\bar{x}^{k}-\eta\bar{y}^{k+1}$, then
we have
\[
\begin{split}f_{\delta}(\bar{x}^{k+1})\leq & f_{\delta}(\bar{x}^{k})+\inprod{\nabla f_{\delta}(\bar{x}^{k})}{\bar{x}^{k+1}-\bar{x}^{k}}+\tfrac{L(\delta)}{2}\norm{\bar{x}^{k+1}-\bar{x}^{k}}^{2}\\
= & f_{\delta}(\bar{x}^{k})+\inprod{\nabla f_{\delta}(\bar{x}^{k})\pm\bar{y}^{k+1}}{\bar{x}^{k+1}-\bar{x}^{k}}+\tfrac{L(\delta)}{2}\norm{\bar{x}^{k+1}-\bar{x}^{k}}^{2}\\
= & f_{\delta}(\bar{x}^{k})+\inprod{\nabla f_{\delta}(\bar{x}^{k})-\bar{y}^{k+1}}{-\eta\bar{y}^{k+1}\pm\eta\nabla f_{\delta}(\bar{x}^{k})}-(\tfrac{1}{\eta}-\tfrac{L(\delta)}{2})\norm{\bar{x}^{k+1}-\bar{x}^{k}}^{2}\\
= & f_{\delta}(\bar{x}^{k})+\eta\norm{\nabla f_{\delta}(\bar{x}^{k})-\bar{y}^{k+1}}^{2}-\eta\inprod{\nabla f_{\delta}(\bar{x}^{k})-\bar{y}^{k+1}}{\nabla f_{\delta}(\bar{x}^{k})}-(\tfrac{1}{\eta}-\tfrac{L(\delta)}{2})\norm{\bar{x}^{k+1}-\bar{x}^{k}}^{2}\\
= & f_{\delta}(\bar{x}^{k})+\eta\norm{\nabla f_{\delta}(\bar{x}^{k})-\bar{y}^{k+1}}^{2}-\eta^{-1}\inprod{\bar{x}^{k+1}-\hat{x}^{k+1}}{\bar{x}^{k}-\hat{x}^{k+1}}-(\tfrac{1}{\eta}-\tfrac{L(\delta)}{2})\norm{\bar{x}^{k+1}-\bar{x}^{k}}^{2}\\
= & f_{\delta}(\bar{x}^{k})+\eta\norm{\nabla f_{\delta}(\bar{x}^{k})-\bar{y}^{k+1}}^{2}-(\tfrac{1}{\eta}-\tfrac{L(\delta)}{2})\norm{\bar{x}^{k+1}-\bar{x}^{k}}^{2}\\
 & -\tfrac{1}{2\eta}(\norm{\bar{x}^{k+1}-\hat{x}^{k+1}}^{2}+\norm{\bar{x}^{k}-\hat{x}^{k+1}}^{2}-\norm{\bar{x}^{k+1}-\bar{x}^{k}}^{2})\\
= & f_{\delta}(\bar{x}^{k})+\eta\norm{\nabla f_{\delta}(\bar{x}^{k})-\bar{y}^{k+1}}^{2}-(\tfrac{1}{\eta}-\tfrac{L(\delta)}{2})\norm{\bar{x}^{k+1}-\bar{x}^{k}}^{2}\\
 & -\tfrac{1}{2\eta}(\eta^{2}\norm{\nabla f_{\delta}(\bar{x}^{k})-\bar{y}^{k+1}}^{2}+\eta^{2}\norm{\nabla f_{\delta}(\bar{x}^{k})}^{2}-\norm{\bar{x}^{k+1}-\bar{x}^{k}}^{2})\\
= & f_{\delta}(\bar{x}^{k})-\tfrac{\eta}{2}\norm{\nabla f_{\delta}(\bar{x}^{k})}^{2}+\tfrac{\eta}{2}\norm{\nabla f_{\delta}(\bar{x}^{k})-\bar{y}^{k+1}}^{2}-(\tfrac{1}{2\eta}-\tfrac{L(\delta)}{2})\norm{\bar{x}^{k+1}-\bar{x}^{k}}^{2}.
\end{split}
\]

Combining the above result and $\bar{y}^{k+1}=\bar{v}^{k},\bar{x}^{k+1}-\bar{x}^{k}=\eta\bar{y}^{k+1}$
completes our proof.
\end{proof}
\begin{lem}
[Variance Bound]\label{lem:For-the-sequence-1}For the sequence $\{\bar{x}^{k},\bar{v}^{k}\}$
generated by Algorithm~\ref{alg:dgfm_appendix}, for any $lT\leq k^{\prime}<k\leq(l+1)T-1$,
we have
\begin{equation}
\begin{split} & \mbb E[\norm{\bar{v}^{k}-\nabla f_{\delta}(\bar{x}^{k})}^{2}]\\
\leq & \tfrac{2L(\delta)^{2}}{m}\mbb E[\norm{\xbf^{k}-\bar{\xbf}^{k}}^{2}]+2\mbb E[\norm{\bar{v}^{k^{\prime}}-\tfrac{1}{m}\sum_{i=1}^{m}\nabla f_{\delta}^{i}(x_{i}^{k^{\prime}})}^{2}]\\
 & +\tfrac{6d^{2}L_{f}^{2}}{m^{2}\delta^{2}b}\sum_{j=k^{\prime}}^{k}\brbra{\mbb E[\norm{\xbf^{j}-\bar{\xbf}^{j}}^{2}+\norm{\xbf^{j-1}-\bar{\xbf}^{j-1}}^{2}]}+\tfrac{6\eta^{2}d^{2}L_{f}^{2}}{m^{2}\delta^{2}b}\sum_{j=k^{\prime}}^{k}\mbb E[\norm{\bar{\vbf}^{j-1}}^{2}].
\end{split}
\label{eq:variance_bound}
\end{equation}
Moreover, for $k=lT$, $l=0,\cdots,L-1$, we have 
\begin{equation}
\mbb E[\norm{\bar{v}^{k}-\tfrac{1}{m}\sum_{i=1}^{m}\nabla f_{\delta}^{i}(x_{i}^{k})}^{2}]\leq\tfrac{\sigma^{2}}{b^{\prime}}.
\end{equation}
\end{lem}
\begin{proof}
For $lT+1\leq k\leq(l+1)T$, the update of $\bar{v}^{k}=\bar{v}^{k-1}+\tfrac{1}{m}\sum_{i=1}^{m}g_{i}(x_{i}^{k};S_{i}^{k})-g_{i}(x_{i}^{k-1};S_{i}^{k})$
leads to 
\[\begin{split} & \mbb E[\norm{\bar{v}^{k}-\tfrac{1}{m}\sum_{i=1}^{m}\nabla f_{\delta}^{i}(x_{i}^{k})}^{2}]\\
= & \mbb E[\norm{\bar{v}^{k-1}+\tfrac{1}{m}\sum_{i=1}^{m}g_{i}(x_{i}^{k};S_{i}^{k})-g_{i}(x_{i}^{k-1};S_{i}^{k})-\nabla f_{\delta}^{i}(x_{i}^{k})}^{2}]\\
= & \mbb E[\norm{\bar{v}^{k-1}-\tfrac{1}{m}\sum_{i=1}^{m}\nabla f_{\delta}^{i}(x_{i}^{k-1})}^{2}+\norm{\tfrac{1}{m}\sum_{i=1}^{m}\nabla f_{\delta}^{i}(x_{i}^{k-1})+g_{i}(x_{i}^{k};S_{i}^{k})-g_{i}(x_{i}^{k-1};S_{i}^{k})-\nabla f_{\delta}^{i}(x_{i}^{k})}^{2}].
\end{split}\]
Rearrange the above inequality, we have
\[
\begin{split} & \mbb E[\norm{\bar{v}^{k}-\tfrac{1}{m}\sum_{i=1}^{m}\nabla f_{\delta}^{i}(x_{i}^{k})}^{2}]-\mbb E[\norm{\bar{v}^{k-1}-\tfrac{1}{m}\sum_{i=1}^{m}\nabla f_{\delta}^{i}(x_{i}^{k-1})}^{2}]\\
\leq & \mbb E[\norm{\tfrac{1}{m}\sum_{i=1}^{m}\nabla f_{\delta}^{i}(x_{i}^{k-1})+g_{i}(x_{i}^{k};S_{i}^{k})-g_{i}(x_{i}^{k-1};S_{i}^{k})-\nabla f_{\delta}^{i}(x_{i}^{k})}^{2}]\\
= & \tfrac{1}{m^{2}}\sum_{i=1}^{m}\mbb E[\norm{\nabla f_{\delta}^{i}(x_{i}^{k-1})+g_{i}(x_{i}^{k};S_{i}^{k})-g_{i}(x_{i}^{k-1};S_{i}^{k})-\nabla f_{\delta}^{i}(x_{i}^{k})}^{2}]\\
\aleq & \tfrac{1}{m^{2}}\sum_{i=1}^{m}\mbb E[\norm{g_{i}(x_{i}^{k};S_{i}^{k})-g_{i}(x_{i}^{k-1};S_{i}^{k})}^{2}]\\
= & \tfrac{1}{m^{2}b}\sum_{i=1}^{m}\mbb E[\norm{g_{i}(x_{i}^{k};w_{i}^{k,1},\xi_{i}^{k,1})-g_{i}(x_{i}^{k-1};w_{i}^{k,1},\xi_{i}^{k,1})}^{2}]\\
\leq &  {\tfrac{d^{2}L_{f}^{2}}{m^{2}\delta^{2}b}}\mbb E[\norm{\xbf^{k}-\xbf^{k-1}}^{2}]= {\tfrac{d^{2}L_{f}^{2}}{m^{2}\delta^{2}b}}\mbb E[\norm{\xbf^{k}\pm\bar{\xbf}^{k}\pm\bar{\xbf}^{k-1}-\xbf^{k-1}}^{2}]\\
\leq &  {\tfrac{3d^{2}L_{f}^{2}}{m^{2}\delta^{2}b}}\brbra{\mbb E[\norm{\xbf^{k}-\bar{\xbf}^{k}}^{2}+\norm{\xbf^{k-1}-\bar{\xbf}^{k-1}}^{2}]}+\tfrac{3\eta^{2}d^{2}L_{f}^{2}}{m^{2}\delta^{2}b}\mbb E[\norm{\bar{\vbf}^{k-1}}^{2}],
\end{split}
\]
where $(a)$ holds by $\Var(X)\leq\mbb E[\norm X^{2}]$ for any random
vector $X$, which implies that for any $lT\leq k^{\prime}<k<(l+1)T$,
we have
\begin{equation}
\begin{split} & \mbb E[\norm{\bar{v}^{k}-\tfrac{1}{m}\sum_{i=1}^{m}\nabla f_{\delta}^{i}(x_{i}^{k})}^{2}]\\
\leq & \mbb E[\norm{\bar{v}^{k^{\prime}}-\tfrac{1}{m}\sum_{i=1}^{m}\nabla f_{\delta}^{i}(x_{i}^{k^{\prime}})}^{2}]\\
 & + {\tfrac{3d^{2}L_{f}^{2}}{m^{2}\delta^{2}b}}\sum_{j=k^{\prime}}^{k-1}\brbra{\mbb E[\norm{\xbf^{j+1}-\bar{\xbf}^{j+1}}^{2}+\norm{\xbf^{j}-\bar{\xbf}^{j}}^{2}]}+{ \tfrac{3\eta^{2}d^{2}L_{f}^{2}}{m^{2}\delta^{2}b}}\sum_{j=k^{\prime}}^{k-1}\mbb E[\norm{\bar{\vbf}^{k^{\prime}}}^{2}].
\end{split}
\label{eq:variance_bound01}
\end{equation}
Furthermore, we have
\begin{equation}
\mbb E[\norm{\tfrac{1}{m}\sum_{i=1}^{m}\nabla f_{\delta}^{i}(x_{i}^{k})-\nabla f_{\delta}^{i}(\bar{x}^{k})}^{2}]\aleq\tfrac{L(\delta)^{2}}{m}\mbb E[\sum_{i=1}^{m}\norm{x_{i}^{k}-\bar{x}^{k}}^{2}]=\tfrac{L(\delta)^{2}}{m}\mbb E[\norm{\xbf^{k}-\bar{\xbf}^{k}}^{2}],\label{eq:variance_bound02}
\end{equation}
where $(a)$ holds by $\norm{\nabla f_{\delta}^{i}(x_{i}^{k})-\nabla f_{\delta}^{i}(\bar{x}^{k})}\leq L(\delta)\norm{x_{i}^{k}-\bar{x}^{k}}$
and $\norm x_{1}\leq\sqrt{m}\norm x_{2},\forall x\in\mbb R^{m}$.
Hence, for any $lT\leq k^{\prime}<k<(l+1)T$, combining~\eqref{eq:variance_bound01}
and~\eqref{eq:variance_bound02} yields
\[
\begin{split} & \mbb E[\norm{\bar{v}^{k}-\nabla f_{\delta}(\bar{x}^{k})}^{2}]=\mbb E[\norm{\bar{v}^{k}-\tfrac{1}{m}\sum_{i=1}^{m}\nabla f_{\delta}^{i}(\bar{x}^{k})}^{2}]\\
\leq & \mbb E[2\norm{\bar{v}^{k}-\tfrac{1}{m}\sum_{i=1}^{m}\nabla f_{\delta}^{i}(x_{i}^{k})}^{2}]+\mbb E[2\norm{\tfrac{1}{m}\sum_{i=1}^{m}\nabla f_{\delta}^{i}(x_{i}^{k})-\nabla f_{\delta}^{i}(\bar{x}^{k})}^{2}]\\
\leq & \tfrac{2L(\delta)^{2}}{m}\mbb E[\norm{\xbf^{k}-\bar{\xbf}^{k}}^{2}]+2\mbb E[\norm{\bar{v}^{k^{\prime}}-\tfrac{1}{m}\sum_{i=1}^{m}\nabla f_{\delta}^{i}(x_{i}^{k^{\prime}})}^{2}]\\
 & + {\tfrac{6d^{2}L_{f}^{2}}{m^{2}\delta^{2}b}}\sum_{j=k^{\prime}}^{k-1}\brbra{\mbb E[\norm{\xbf^{j+1}-\bar{\xbf}^{j+1}}^{2}+\norm{\xbf^{j}-\bar{\xbf}^{j}}^{2}]}+{ \tfrac{6\eta^{2}d^{2}L_{f}^{2}}{m^{2}\delta^{2}b}}\sum_{j=k^{\prime}}^{k-1}\mbb E[\norm{\bar{\vbf}^{j}}^{2}].
\end{split}
\]

Furthermore, for $k=lT$, we have
\[
\mbb E[\norm{\bar{v}^{k}-\tfrac{1}{m}\sum_{i=1}^{m}\nabla f_{\delta}^{i}(x_{i}^{k})}^{2}]=\mbb E[\norm{\tfrac{1}{m}\sum_{i=1}^{m}g_{i}(x_{i}^{k};S_{i}^{k\prime})-\nabla f_{\delta}^{i}(x_{i}^{k})}^{2}]\leq\tfrac{\sigma^{2}}{b^{\prime}}.
\]
\end{proof}
\begin{nonumber_lem}
[The Whole Sequence Progress]\label{lem:one_epoch_progress}Suppose
$K=L\cdot T$, then for any $l=0,\cdots,L-1$, $t=0,\cdots,T$, we
have
\begin{equation}
\begin{split} & \sum_{k=0}^{LT-1}\mbb E[\norm{\bar{v}^{k}-\nabla f_{\delta}(\bar{x}^{k})}^{2}\\
\leq & \tfrac{2L(\delta)^{2}}{m}\sum_{k=0}^{LT-1}\mbb E[\norm{\xbf^{k}-\bar{\xbf}^{k}}^{2}]+2LT\cdot {\tfrac{\sigma^{2}}{b^{\prime}}}\\
 & +\tfrac{6d^{2}L_{f}^{2}T}{m^{2}\delta^{2}b}\sum_{k=1}^{LT-1}\brbra{\mbb E[\norm{\xbf^{k}-\bar{\xbf}^{k}}^{2}+\norm{\xbf^{k-1}-\bar{\xbf}^{k-1}}^{2}]}+\tfrac{6\eta^{2}d^{2}L_{f}^{2}T}{m^{2}\delta^{2}b}\sum_{k=0}^{LT-1}\mbb E[\norm{\bar{v}^{k}}^{2}].
\end{split}
\label{eq:the_whole_sequence_progress}
\end{equation}
\end{nonumber_lem}
\begin{proof}
Taking $k^{\prime}=lT$ in~\eqref{eq:variance_bound}, then for any
$k=lT+t$, $t=1,\cdots,T-1$, we have
\begin{equation}
\begin{split} & \mbb E[\norm{\bar{v}^{lT+t}-\nabla f_{\delta}(\bar{x}^{lT+t})}^{2}]\\
\leq & \tfrac{2L(\delta)^{2}}{m}\mbb E[\norm{\xbf^{lT+t}-\bar{\xbf}^{lT+t}}^{2}]+2\mbb E[\norm{\bar{v}^{lT}-\tfrac{1}{m}\sum_{i=1}^{m}\nabla f_{\delta}^{i}(x_{i}^{lT})}^{2}]\\
 & +\tfrac{6d^{2}L_{f}^{2}}{m^{2}\delta^{2}b}\sum_{j=k^{\prime}}^{k}\brbra{\mbb E[\norm{\xbf^{j}-\bar{\xbf}^{j}}^{2}+\norm{\xbf^{j-1}-\bar{\xbf}^{j-1}}^{2}]}+\tfrac{6\eta^{2}d^{2}L_{f}^{2}}{m^{2}\delta^{2}b}\sum_{j=k^{\prime}}^{k}\mbb E[\norm{\bar{\vbf}^{j-1}}^{2}]\\
\aleq & \tfrac{2L(\delta)^{2}}{m}\mbb E[\norm{\xbf^{lT+t}-\bar{\xbf}^{lT+t}}^{2}]+2\mbb E[\norm{\bar{v}^{lT}-\tfrac{1}{m}\sum_{i=1}^{m}\nabla f_{\delta}^{i}(x_{i}^{lT})}^{2}]\\
 & +\tfrac{6d^{2}L_{f}^{2}}{m^{2}\delta^{2}b}\sum_{k=lT}^{(l+1)T-2}\brbra{\mbb E[\norm{\xbf^{k+1}-\bar{\xbf}^{k+1}}^{2}+\norm{\xbf^{k}-\bar{\xbf}^{k}}^{2}]}+\tfrac{6\eta^{2}d^{2}L_{f}^{2}}{m^{2}\delta^{2}b}\sum_{k=lT}^{(l+1)T-2}\mbb E[\norm{\bar{\vbf}^{k}}^{2}],
\end{split}
\label{eq:variance_bound-1}
\end{equation}
where $(a)$ holds by replacing index from $j$ to $k$ and $\sum_{j=lT}^{lT+t-1}\brbra{\mbb E[\norm{\xbf^{j+1}-\bar{\xbf}^{j+1}}^{2}+\norm{\xbf^{j}-\bar{\xbf}^{j}}^{2}]}\leq\sum_{k=lT}^{(l+1)T-2}\brbra{\mbb E[\norm{\xbf^{k+1}-\bar{\xbf}^{k+1}}^{2}+\norm{\xbf^{k}-\bar{\xbf}^{k}}^{2}]}$.
Summing $\mbb E[\norm{\bar{v}^{lT}-\tfrac{1}{m}\sum_{i=1}^{m}\nabla f_{\delta}^{i}(x_{i}^{lT})}^{2}]$
and the above inequality over $t=1$ to $T-1$, and combining the
fact $\norm{\bar{\vbf}^{j}}^{2}=m\norm{\bar{v}^{j}}^{2}$ yields
\begin{align*}
\sum_{k=lT}^{(l+1)T-1}\mbb E[\norm{\bar{v}^{k}-\nabla f_{\delta}(\bar{x}^{k})}^{2}]\leq & \tfrac{2L(\delta)^{2}}{m}\sum_{k=lT+1}^{(l+1)T-1}\mbb E[\norm{\xbf^{k}-\bar{\xbf}^{k}}^{2}]+(2(T-1)+1)\cdot\mbb E[\norm{\bar{v}^{lT}-\tfrac{1}{m}\sum_{i=1}^{m}\nabla f_{\delta}^{i}(x_{i}^{lT})}^{2}]\\
 & +\tfrac{6d^{2}L_{f}^{2}T}{m^{2}\delta^{2}b}\sum_{k=lT}^{(l+1)T-2}\brbra{\mbb E[\norm{\xbf^{k+1}-\bar{\xbf}^{k+1}}^{2}+\norm{\xbf^{k}-\bar{\xbf}^{k}}^{2}]}+\tfrac{6\eta^{2}d^{2}L_{f}^{2}T}{m^{2}\delta^{2}b}\sum_{k=lT}^{(l+1)T-2}\mbb E[\norm{\bar{\vbf}^{k}}^{2}].
\end{align*}

Summing the above inequality over $l=0$ to $L-1$ and combining $\mbb E[\norm{\bar{v}^{lT}-\tfrac{1}{m}\sum_{i=1}^{m}\nabla f_{\delta}^{i}(x_{i}^{lT})}^{2}]\leq\tfrac{\sigma^{2}}{b^{\prime}}$,
we have
\[
\begin{split} & \sum_{k=0}^{LT-1}\mbb E[\norm{\bar{v}^{k}-\nabla f_{\delta}(\bar{x}^{k})}^{2}]\\
\leq & \tfrac{2L(\delta)^{2}}{m}\sum_{l=0}^{L-1}\sum_{k=lT+1}^{(l+1)T-1}\mbb E[\norm{\xbf^{k}-\bar{\xbf}^{k}}^{2}]+2LT\cdot\tfrac{\sigma^{2}}{b^{\prime}}\\
 & +\tfrac{6d^{2}L_{f}^{2}T}{m^{2}\delta^{2}b}\sum_{l=0}^{L-1}\sum_{k=lT}^{(l+1)T-2}\brbra{\mbb E[\norm{\xbf^{k+1}-\bar{\xbf}^{k+1}}^{2}+\norm{\xbf^{k}-\bar{\xbf}^{k}}^{2}]}+\tfrac{6\eta^{2}d^{2}L_{f}^{2}T}{m^{2}\delta^{2}b}\sum_{l=0}^{L-1}\sum_{k=lT}^{(l+1)T-2}\mbb E[\norm{\bar{\vbf}^{k}}^{2}],
\end{split}
\]
which implies the desired result.
\end{proof}
\begin{thm}
\label{thm: main_theorem}For sequence $\{\xbf^{k},\ybf^{k}\}$
generated by {\dgfm}, we have

\begin{equation}
\begin{split} & -\beta_{y}\mbb E[\norm{\ybf^{0}-\bar{\ybf}^{0}}^{2}]+\beta_{x}\mbb E[\norm{\xbf^{LT}-\bar{\xbf}^{LT}}^{2}-\norm{\xbf^{0}-\bar{\xbf}^{0}}^{2}]+\mbb E[f_{\delta}(\bar{x}^{LT})]-\mbb E[f_{\delta}(\bar{x}^{0})]\\
\leq & -\tfrac{\eta}{2}\sum_{k=0}^{LT-1}\mbb E[\norm{\nabla f_{\delta}(\bar{x}^{k})}^{2}]+\eta LT\cdot {\tfrac{\sigma^{2}}{b^{\prime}}}\\
 & -(\tfrac{\eta}{2}-\tfrac{L(\delta)\eta^{2}}{2}-\tfrac{3\eta^{3}d^{2}L_{f}^{2}T}{m^{2}\delta^{2}b}-3\beta_{y}m\eta^{2}(\tfrac{\rho^{2}L(\delta)^{2}}{\alpha_{1}}+\tfrac{\rho^{2}d^{2}L_{f}^{2}}{\delta^{2}}))\sum_{k=0}^{LT-1}\mbb E[\norm{\bar{v}^{k}}^{2}]\\
 & -\brbra{\beta_{x}(1-(1+\alpha_{2})\rho^{2})-6\beta_{y}(\tfrac{\rho^{2}L(\delta)^{2}}{\alpha_{1}}+\tfrac{\rho^{2}d^{2}L_{f}^{2}}{\delta^{2}})-(\tfrac{\eta L(\delta)^{2}}{m}+\tfrac{6\eta d^{2}L_{f}^{2}T}{m^{2}\delta^{2}b})}\sum_{k=0}^{LT-1}\mbb E[\norm{\xbf^{k}-\bar{\xbf}^{k}}^{2}]\\
 & -(\beta_{y}(1-\rho^{2}(1+\alpha_{1}))-\beta_{x}(1+\alpha_{2}^{-1})\rho^{2}\eta^{2})\sum_{k=0}^{LT-1}\mbb E[\norm{\ybf^{k}-\bar{\ybf}^{k}}^{2}]\\
 & -(\beta_{y}-\beta_{x}(1+\alpha_{2}^{-1})\rho^{2}\eta^{2})\mbb E[\norm{\ybf^{LT}-\bar{\ybf}^{LT}}^{2}]+{ 2\rho^{\Tcal}m(\sigma^{2}+L_{f}^{2})}\cdot\beta_{y}L.\text{}
\end{split}
\label{eq:all_consensus-1-1-1}
\end{equation}
\end{thm}
\begin{proof}
By Lemma~\ref{lem:For-the-sequence}, for all $k$, we have
\[
f_{\delta}(\bar{x}^{k+1})\leq f_{\delta}(\bar{x}^{k})-\tfrac{\eta}{2}\norm{\nabla f_{\delta}(\bar{x}^{k})}^{2}+\tfrac{\eta}{2}\norm{\nabla f_{\delta}(\bar{x}^{k})-\bar{v}^{k}}^{2}-(\tfrac{\eta}{2}-\tfrac{L(\delta)\eta^{2}}{2})\norm{\bar{v}^{k}}^{2}.
\]

Taking expectations of both sides of the above inequality and summing
it over $k=0$ to $LT-1$, we have 
\[
\begin{split} & \mbb E[f_{\delta}(\bar{x}^{LT})]-\mbb E[f_{\delta}(\bar{x}^{0})]\\
\leq & -\tfrac{\eta}{2}\sum_{k=0}^{LT-1}\mbb E[\norm{\nabla f_{\delta}(\bar{x}^{k})}^{2}]-(\tfrac{\eta}{2}-\tfrac{L(\delta)\eta^{2}}{2})\sum_{k=0}^{LT-1}\mbb E[\norm{\bar{v}^{k}}^{2}]\\
 & +\tfrac{\eta}{2}\sum_{k=0}^{LT-1}\mbb E[\norm{\nabla f_{\delta}(\bar{x}^{k})-\bar{v}^{k}}^{2}]\\
\leq & -\tfrac{\eta}{2}\sum_{k=0}^{LT-1}\mbb E[\norm{\nabla f_{\delta}(\bar{x}^{k})}^{2}]-(\tfrac{\eta}{2}-\tfrac{L(\delta)\eta^{2}}{2})\sum_{k=0}^{LT-1}\mbb E[\norm{\bar{v}^{k}}^{2}]\\
 & +\tfrac{\eta L(\delta)^{2}}{m}\sum_{k=0}^{LT-1}\mbb E[\norm{\xbf^{k}-\bar{\xbf}^{k}}^{2}]+\eta LT\cdot {\tfrac{\sigma^{2}}{b^{\prime}}}\\
 & +\tfrac{3\eta d^{2}L_{f}^{2}}{m^{2}\delta^{2}b}\sum_{k=1}^{LT-1}\brbra{\mbb E[\norm{\xbf^{k}-\bar{\xbf}^{k}}^{2}+\norm{\xbf^{k-1}-\bar{\xbf}^{k-1}}^{2}]}+\tfrac{3\eta^{3}d^{2}L_{f}^{2}}{m^{2}\delta^{2}b}\sum_{k=0}^{LT-1}\mbb E[\norm{\bar{v}^{k}}^{2}].\\
\aleq & -\tfrac{\eta}{2}\sum_{k=0}^{LT-1}\mbb E[\norm{\nabla f_{\delta}(\bar{x}^{k})}^{2}]-(\tfrac{\eta}{2}-\tfrac{L(\delta)\eta^{2}}{2}-\tfrac{3\eta^{3}d^{2}L_{f}^{2}T}{m^{2}\delta^{2}b})\sum_{k=0}^{LT-1}\mbb E[\norm{\bar{v}^{k}}^{2}]+\eta LT\cdot {\tfrac{\sigma^{2}}{b^{\prime}}}\\
 & +(\tfrac{\eta L(\delta)^{2}}{m}+\tfrac{6\eta d^{2}L_{f}^{2}T}{m^{2}\delta^{2}b})\sum_{k=0}^{LT-1}\mbb E[\norm{\xbf^{k}-\bar{\xbf}^{k}}^{2}],
\end{split}
\]
where $(a)$ holds by $\sum_{k=0}^{LT-1}\mbb E[\norm{\xbf^{k-1}-\bar{\xbf}^{k-1}}^{2}]\leq\sum_{k=0}^{LT-1}\mbb E[\norm{\xbf^{k}-\bar{\xbf}^{k}}^{2}]$.
Summing the above inequality and~\eqref{eq:all_consensus}, we have

\begin{equation}
\begin{split} & -\beta_{y}\mbb E[\norm{\ybf^{0}-\bar{\ybf}^{0}}^{2}]+\beta_{x}\mbb E[\norm{\xbf^{LT}-\bar{\xbf}^{LT}}^{2}-\norm{\xbf^{0}-\bar{\xbf}^{0}}^{2}]+\mbb E[f_{\delta}(\bar{x}^{LT})]-\mbb E[f_{\delta}(\bar{x}^{0})]\\
\aleq & -\tfrac{\eta}{2}\sum_{k=0}^{LT-1}\mbb E[\norm{\nabla f_{\delta}(\bar{x}^{k})}^{2}]-(\tfrac{\eta}{2}-\tfrac{L(\delta)\eta^{2}}{2}-\tfrac{3\eta^{3}d^{2}L_{f}^{2}T}{m^{2}\delta^{2}b}-3\beta_{y}m\eta^{2}(\tfrac{\rho^{2}L(\delta)^{2}}{\alpha_{1}}+\tfrac{\rho^{2}d^{2}L_{f}^{2}}{\delta^{2}}))\sum_{k=0}^{LT-1}\mbb E[\norm{\bar{v}^{k}}^{2}]+\eta LT\cdot {\tfrac{\sigma^{2}}{b^{\prime}}}\\
 & -\brbra{\beta_{x}(1-(1+\alpha_{2})\rho^{2})-6\beta_{y}(\tfrac{\rho^{2}L(\delta)^{2}}{\alpha_{1}}+\tfrac{\rho^{2}d^{2}L_{f}^{2}}{\delta^{2}})-(\tfrac{\eta L(\delta)^{2}}{m}+\tfrac{6\eta d^{2}L_{f}^{2}T}{m^{2}\delta^{2}b})}\sum_{k=0}^{LT-1}\mbb E[\norm{\xbf^{k}-\bar{\xbf}^{k}}^{2}]\\
 & -(\beta_{y}(1-\rho^{2}(1+\alpha_{1}))-\beta_{x}(1+\alpha_{2}^{-1})\rho^{2}\eta^{2})\sum_{k=0}^{LT-1}\mbb E[\norm{\ybf^{k}-\bar{\ybf}^{k}}^{2}]\\
 & -(\beta_{y}-\beta_{x}(1+\alpha_{2}^{-1})\rho^{2}\eta^{2})\mbb E[\norm{\ybf^{LT}-\bar{\ybf}^{LT}}^{2}]+{ 2\rho^{\Tcal}m(\sigma^{2}+L_{f}^{2})}\cdot\beta_{y}L.
\end{split}
\label{eq:all_consensus-1}
\end{equation}
where $(a)$ holds by $ {\sum_{k=0}^{LT-1}}\mbb E[\norm{\bar{\vbf}^{k-1}}^{2}]\leq m {\sum_{k=0}^{LT-1}}\mbb E[\norm{\bar{v}^{k}}^{2}]$.
Then we complete our proof.
\end{proof}

\begin{cor}
    Set $\alpha_1 = \alpha_2 = \tfrac{1-\rho^2}{2\rho^2}$, $\delta=\mcal O(\vep)$, $\beta_x = \Ocal(\delta^{-1})$, $\beta_y = \Ocal(\delta)$, $b^{\prime} = \Ocal(\vep^{-2}), b=\mcal O(\vep^{-1})$, $\eta = \mcal O(\vep)$, $T=\mcal O(\vep^{-1})$, $L=\mcal O(\Delta_{\delta}\vep^{-2})$, $\mcal{T} = \mcal O(\log(\vep^{-1}))$ in {\dgfm}.  Then, it outputs a $(\delta,\vep)$-Goldstein stationary point of $f(\cdot)$ in expectation. The total stochastic zeroth-order complexity is at most $\mcal O(\Delta_{\delta}\delta^{-1}\vep^{-3})$ and the total number of communication rounds is at most $\Ocal(\Delta_{\delta}\vep^{-2}(\vep^{-1}+\log(\vep^{-1})))$.
\end{cor}

\begin{proof}
It follows from Theorem~\ref{thm: main_theorem} that 
\begin{equation}
\begin{split} & \tfrac{\eta}{2}\sum_{k=0}^{LT-1}\mbb E[\norm{\nabla f_{\delta}(\bar{x}^{k})}^{2}]\\
\leq & \beta_{y}\mbb E[\norm{\ybf^{0}-\bar{\ybf}^{0}}^{2}]+\beta_{x}\mbb E[\norm{\xbf^{0}-\bar{\xbf}^{0}}^{2}-\norm{\xbf^{LT}-\bar{\xbf}^{LT}}^{2}]\\
 & +\mbb E[f_{\delta}(\bar{x}^{0})]-\mbb E[f_{\delta}(\bar{x}^{LT})]\\
 & -\underbrace{(\tfrac{\eta}{2}-\tfrac{L(\delta)\eta^{2}}{2}-\tfrac{3\eta^{3}d^{2}L_{f}^{2}T}{m^{2}\delta^{2}b}-3\beta_{y}m\eta^{2}(\tfrac{\rho^{2}L(\delta)^{2}}{\alpha_{1}}+\tfrac{\rho^{2}d^{2}L_{f}^{2}}{\delta^{2}}))}_{(\clubsuit)}\sum_{k=0}^{LT-1}\mbb E[\norm{\bar{v}^{k}}^{2}]\\
 & -\underbrace{\brbra{\beta_{x}(1-(1+\alpha_{2})\rho^{2})-6\beta_{y}(\tfrac{\rho^{2}L(\delta)^{2}}{\alpha_{1}}+\tfrac{\rho^{2}d^{2}L_{f}^{2}}{\delta^{2}})-(\tfrac{\eta L(\delta)^{2}}{m}+\tfrac{6\eta d^{2}L_{f}^{2}T}{m^{2}\delta^{2}b})}}_{(\spadesuit)}\sum_{k=0}^{LT-1}\mbb E[\norm{\xbf^{k}-\bar{\xbf}^{k}}^{2}]\\
 & -\underbrace{(\beta_{y}(1-\rho^{2}(1+\alpha_{1}))-\beta_{x}(1+\alpha_{2}^{-1})\rho^{2}\eta^{2})}_{(\blacklozenge)}\sum_{k=0}^{LT-1}\mbb E[\norm{\ybf^{k}-\bar{\ybf}^{k}}^{2}]\\
 & -\underbrace{(\beta_{y}-\beta_{x}(1+\alpha_{2}^{-1})\rho^{2}\eta^{2})}_{(\blacktriangle)}\mbb E[\norm{\ybf^{LT}-\bar{\ybf}^{LT}}^{2}]+{ 2\rho^{\Tcal}m(\sigma^{2}+L_{f}^{2})}\cdot\beta_{y}L.
\end{split}
\end{equation}
Since we choose proper initial point $y_i^0$ and $x_i^0$, then  we simplify the above inequality as 
\begin{align}
 & (\clubsuit)\sum_{k=0}^{LT-1}\mbb E[\norm{\bar{v}^{k}}^{2}]+(\spadesuit)\sum_{k=0}^{LT-1}\mbb E[\norm{\xbf^{k}-\bar{\xbf}^{k}}^{2}]+(\blacklozenge)\sum_{k=0}^{LT-1}\mbb E[\norm{\ybf^{k}-\bar{\ybf}^{k}}^{2}]+(\blacktriangle)\mbb E[\norm{\ybf^{LT}-\bar{\ybf}^{LT}}^{2}]+\tfrac{\eta}{2}\sum_{k=0}^{LT-1}\mbb E[\norm{\nabla f_{\delta}(\bar{x}^{k})}^{2}]\nonumber\\
\leq & \mbb E[f_{\delta}(\bar{x}^{0})]-\mbb E[f_{\delta}(\bar{x}^{LT})]+{ 2\rho^{\Tcal}m(\sigma^{2}+L_{f}^{2})}\cdot\beta_{y}L.\label{eq:dgfmres1}
\end{align}
Combining the parameters setting of $\alpha_1,L,T,\beta_x$, $\beta_y$ and $\eta$ yields $(\clubsuit),(\spadesuit),(\blacklozenge),(\blacktriangle)>0$ and 
\begin{equation}\label{eq:dgfmorder}
    (\clubsuit) = \mcal O(\vep),\ (\spadesuit) = \Ocal (\delta^{-1}),\ (\blacklozenge)=\Ocal (\delta),\ (\blacktriangle) = \Ocal(\delta).
\end{equation}
By~\eqref{eq:dgfmorder}, we can simplify~\eqref{eq:dgfmres1} as 
\begin{equation}\label{eq:dgfmres2}
(\spadesuit)\sum_{k=0}^{LT-1}\mbb E[\norm{\xbf^{k}-\bar{\xbf}^{k}}^{2}]+\tfrac{\eta}{2}\sum_{k=0}^{LT-1}\mbb E[\norm{\nabla f_{\delta}(\bar{x}^{k})}^{2}]\leq \Delta_{\delta} + { 2\rho^{\Tcal}m(\sigma^{2}+L_{f}^{2})}\cdot\beta_{y}L.
\end{equation}
Divide $\tfrac{\eta\cdot LT}{2}$ on both sides of~\eqref{eq:dgfmres2}, we have
\begin{equation}
    \tfrac{2 \cdot (\spadesuit)}{\eta\cdot LT}\sum_{k=0}^{LT-1}\mbb E[\norm{\xbf^{k}-\bar{\xbf}^{k}}^{2}]+\tfrac{1}{ LT}\sum_{k=0}^{LT-1}\mbb E[\norm{\nabla f_{\delta}(\bar{x}^{k})}^{2}]\leq \tfrac{2\Delta_{\delta}}{\eta \cdot LT} + \tfrac{4\rho^{\Tcal}m(\sigma^{2}+L_{f}^{2})\cdot\beta_{y}}{\eta \cdot T}.
\end{equation}
Combining the parameter setting of $\mcal T, \eta, L, T, \delta$, then we have $\tfrac{(\spadesuit)}{\eta} = \mcal O(\delta^{-2})$ and
\begin{equation}
    \tfrac{2 \cdot (\spadesuit)}{\eta\cdot LT}\sum_{k=0}^{LT-1}\mbb E[\norm{\xbf^{k}-\bar{\xbf}^{k}}^{2}]+\tfrac{1}{ LT}\sum_{k=0}^{LT-1}\mbb E[\norm{\nabla f_{\delta}(\bar{x}^{k})}^{2}] = \mcal O(\vep^2).
\end{equation}
By $\tfrac{(\spadesuit)}{\eta} = \mcal O(\delta^{-2})$, $\nabla f_{\delta}(x)\in \partial_{\delta}f(x)$ and $\norm{\nabla f_{\delta}(x_i^k)}^2\leq 2\norm{\nabla f_{\delta}(\bar{x}^k)}^2 + 2L(\delta)^{2}\norm{x_i^k-\bar{x}^k}^2$, then we complete our proof.

\end{proof}

\newpage{}

\bibliography{ref}